\documentclass[a4paper,11pt,twoside]{article}
\usepackage[T1]{fontenc}
\usepackage[utf8]{inputenc}
\usepackage[english]{babel}
\usepackage{lmodern}
\usepackage{microtype}
\usepackage{geometry}
\usepackage{setspace}
\usepackage{booktabs}
\usepackage{graphicx}
\usepackage{hyperref}
\usepackage[nottoc]{tocbibind}
\usepackage{enumitem}
\usepackage{amsmath,amssymb}
\usepackage{amsthm}
\usepackage{amscd}
\usepackage{mathtools}
\usepackage{mathrsfs}
\usepackage{braket}
\usepackage{extarrows}
\usepackage{tikz}
\usepackage{tikz-cd}
\usetikzlibrary{arrows}

\theoremstyle{plain}
\newtheorem{thm}{Theorem}[section]
\newtheorem{cor}[thm]{Corollary}
\newtheorem{prp}[thm]{Proposition}
\newtheorem{lem}[thm]{Lemma}
\newtheorem{customthm}{Theorem}
 \newenvironment{repthm}[1]
 {\renewcommand\thecustomthm{#1}\customthm}
 {\endcustomthm}
\newtheorem{customcor}{Corollary}
 \newenvironment{repcor}[1]
 {\renewcommand\thecustomcor{#1}\customcor}
 {\endcustomcor}

\theoremstyle{definition}
\newtheorem{dfn}[thm]{Definition}
\newtheorem{rmk}[thm]{Remark}
\newtheorem{ex}[thm]{Example}

\numberwithin{equation}{section}

\newcommand{\mi}{\mu}
\newcommand{\N}{\mathbb{N}}
\newcommand{\Z}{\mathbb{Z}}
\newcommand{\Q}{\mathbb{Q}}
\newcommand{\R}{\mathbb{R}}
\newcommand{\F}{\mathbb{F}}
\newcommand{\D}{\mathbb{D}}
\newcommand{\id}{\mathrm{id}}
\newcommand{\pr}{\mathrm{pr}}
\newcommand{\Sp}{\mathrm{Sp}}
\newcommand{\Fil}{\mathrm{Fil}}
\newcommand{\Isoc}{\mathrm{Isoc}}
\newcommand{\Vect}{\mathrm{Vect}}
\newcommand{\pdiv}{p\mathrm{-div}}
\newcommand{\pgr}{p\mathrm{-gr}}
\newcommand{\Rep}{\mathrm{Rep}}
\newcommand{\nr}{\mathrm{nr}}
\newcommand{\wa}{\mathrm{w\text{-}a}}
\newcommand{\et}{\text{ét}}
\newcommand{\cris}{\text{cris}}
\newcommand{\NF}{\mathcal{N}}

\renewcommand{\phi}{\varphi}
\renewcommand{\epsilon}{\varepsilon}
\renewcommand{\O}{\mathcal{O}}
\renewcommand{\k}{\mathit{k}}
\renewcommand{\r}{\mathrm{r}}
\renewcommand{\c}{\mathcal{C}}
\renewcommand{\d}{\mathcal{D}}

\DeclareMathOperator{\Ker}{Ker}
\DeclareMathOperator{\Ima}{Im}

\DeclareMathOperator{\Spec}{Spec}
\DeclareMathOperator{\Spf}{Spf}
\DeclareMathOperator{\End}{End}
\DeclareMathOperator{\Hom}{Hom}
\DeclareMathOperator{\Lie}{Lie}
\DeclareMathOperator{\Newt}{Newt}
\DeclareMathOperator{\Hdg}{Hdg}
\DeclareMathOperator{\HN}{HN}
\DeclareMathOperator{\height}{ht}
\DeclareMathOperator{\length}{lg}
\DeclareMathOperator{\gr}{gr}
\DeclareMathOperator{\rk}{rk}
\DeclareMathOperator{\Res}{Res}
\DeclareMathOperator{\GL}{GL}
\DeclareMathOperator{\GSp}{GSp}

\title{\vspace{-0.33in} Hodge-Newton filtration for $p$-divisible groups \\
 with ramified endomorphism structure}
\author{Andrea Marrama}
\date{}

\begin{document}

\maketitle

\paragraph{Abstract.}
Let $\O_K$ be a complete discrete valuation ring of
mixed characteristic $(0,p)$ with perfect residue field.
We prove the existence of the Hodge-Newton filtration for $p$-divisible groups over
$\O_K$ with additional endomorphism structure for the ring of integers of a finite,
possibly ramified field extension of $\Q_p$.
The argument is based on the Harder-Narasimhan theory for finite flat group schemes over $\O_K$.
In particular, we describe a sufficient condition for the existence of a filtration of
$p$-divisible groups over $\O_K$ associated to a break point of the Harder-Narasimhan polygon.

\tableofcontents

\clearpage

\pagestyle{headings}
\renewcommand{\sectionmark}[1]{\markboth{\MakeUppercase{#1}}{\MakeUppercase{\thesection.\ #1}}}

\section*{Introduction}
\addcontentsline{toc}{section}{Introduction}
\markboth{\MakeUppercase{Introduction}}{\MakeUppercase{Introduction}}
\thispagestyle{plain}

Let $p$ be a prime number.
This work concerns $p$-divisible groups, alias Barsotti-Tate groups,
over a complete discrete valuation ring $\O_K$ in mixed characteristic,
with perfect residue field $\k$ of characteristic $p$.
More generally, we consider $p$-divisible groups $H$ endowed with
additional endomorphism structure $\iota\colon\O_F\rightarrow\End(H)$,
where $\O_F$ is the ring of integers of a finite field extension $F$ of $\Q_p$.
To any such pair $(H,\iota)$ we associate two invariants up to ($\O_F$-equivariant) isogeny,
the \emph{Hodge polygon} $\Hdg(H,\iota)$ and the \emph{Newton polygon} $\Newt(H,\iota)$;
each of them consists in a collection of ``slopes'' (with multiplicities) and
can be visualised as a concave polygonal curve, starting from the origin, in the Euclidean plane.
The main property relating these two polygons is that
they lie one above the other and share the end point.
When they share a further point $z$ and this is a break point of the Newton polygon,
we say that $(H,\iota)$ is \emph{Hodge-Newton reducible} (\emph{at} $z$).
In this case, we prove that there exists an $\iota$-stable sub-$p$-divisible group $H_1$ of $H$
corresponding to the division of the polygons in their parts before and after $z$,
that is, a \emph{Hodge-Newton filtration} of $(H,\iota)$.
More precisely:

\begin{repthm}{\ref{thm}}
 Let $(H,\iota)$ be a $p$-divisible group over $\O_K$ with endomorphism structure for $\O_F$
 and suppose that $(H,\iota)$ is Hodge-Newton reducible at $z$.
 Then, there exists a unique $\iota$-stable sub-$p$-divisible group $H_1$ of $H$ such that,
 if $\iota_1$ denotes the restriction of $\iota$ to $H_1$,
 then $\Newt(H_1,\iota_1)$, $\Hdg(H_1,\iota_1)$ and $\HN(H_1,\iota_1)$ equal respectively
 the part of $\Newt(H,\iota)$, $\Hdg(H,\iota)$ and $\HN(H,\iota)$ between the origin and $z$.
 Furthermore, if $H_2$ denotes the quotient of $H$ by $H_1$, with induced $\O_F$-action $\iota_2$,
 then $\Newt(H_2,\iota_2)$, $\Hdg(H_2,\iota_2)$ and $\HN(H_2,\iota_2)$ equal respectively
 the rest of $\Newt(H,\iota)$, $\Hdg(H,\iota)$ and $\HN(H,\iota)$ after $z$
 (up to a shift of coordinates setting the origin in $z$).
\end{repthm}

The third polygon featuring in the statement, the \emph{Harder-Narasimhan polygon} $\HN(H,\iota)$,
is related to the strategy behind the proof.
This, indeed, is based on the \emph{Harder-Narasimhan theory}
for finite flat group schemes developed by Fargues,
but let us postpone this discussion to a second moment.

In presence of an $\O_F$-equivariant polarisation $\lambda\colon H\xrightarrow{\sim}H^\vee$
(where the dual $p$-divisible group $H^\vee$ is endowed with the dual $\O_F$-action,
possibly twisted by a field involution of~$F$), the three polygons gain some symmetry.
This leads to an enhanced version of the theorem,
where a division of the polygons into \emph{three} parts corresponds to
the existence of an $\iota$-stable filtration $H_1\subseteq H_1'\subseteq H$,
with $\lambda$ inducing crossed isomorphisms between the graded pieces and their duals
(see Corollary~\ref{cor}).

Looking at the special fibre,
the theorem reconnects with the \emph{Hodge-Newton decomposition} from \cite[\S 1]{BH};
in particular, the reduction to $\k$ of the filtration obtained here is $\O_F$-equivariantly split
(see the discussion in \S\ref{spf}).

\bigskip\noindent
In order to put the above results into context, let us give some historical coordinates.
The notion of Hodge-Newton reducibility first appeared in Katz's work~\cite[\S 1.6]{K} (1979),
in the framework of $F$-crystals over a perfect field $\k$ of characteristic $p$.
To these objects are associated a Hodge polygon and a Newton polygon,
satisfying the same main property as described before
(in this context, this is known as \emph{Mazur's inequality}).
Under an identical hypothesis on the polygons as above,
Katz proves the existence of a ``Hodge-Newton'' decomposition of
the relative $F$-crystal into two components, whose polygons correspond respectively to
the part before and after the internal contact point given by the assumption.
In case such an $F$-crystal is the Dieudonné module of a $p$-divisible group over $\k$,
this result recovers the multiplicative-bilocal-étale decomposition.

This finding was later generalised by Kottwitz, in \cite{Ko} (2003),
to $F$-crystals with additional endomorphism structure for the ring of integers of
a finite unramified extension~$F$ of $\Q_p$ and possibly endowed with a polarisation.
By this time, however, new points of view on the subject had developed.
Kottwitz's result is formulated in terms of \emph{affine Deligne-Lusztig sets},
objects constructed by means of a reductive group over a local field and
determined by a defining datum.
The additional structure is encoded in the reductive group,
in this case the restriction of scalars from $F$ to $\Q_p$ of a general linear group $\GL_n$,
or a general symplectic group $\GSp_{2n}$ in presence of a polarisation.
The definitions of the Hodge and the Newton polygons,
as well as the notion of Hodge-Newton reducibility,
are also translated in the group-theoretic language,
thus taking into account the additional structure;
for the groups in question, all this can still be visualised in terms of polygons.
An affine Deligne-Lusztig set can then be seen as a set of $F$-crystals with additional structure,
whose Hodge and Newton polygons are fixed by the defining datum.
From this point of view, the Hodge-Newton decomposition is expressed as a bijection between
the affine Deligne-Lusztig set associated to a Hodge-Newton reducible datum and
one relative to a Levi subgroup, which collects $F$-crystals admitting a decomposition.

Mantovan and Viehmann, in \cite{MV} (2010), further generalised Kottwitz's result to
endomorphism structures for more general unramified $\Z_p$-algebras and to families of
$F$-crystals in characteristic $p$
(a statement in families could also be found in Katz's original article, without additional structure,
and in Csima's work~\cite{Cs}, for polarised objects).
Most importantly, however, the two authors proved that their Hodge-Newton decomposition can be
lifted to a filtration of $p$-divisible groups over a complete Noetherian local $W(\k)$-algebra
(here, $W(\k)$ denotes the ring of Witt vectors with coefficients in $\k$).
Their argument is based on an explicit description of the universal deformation
of a $p$-divisible group over $\k$ with unramified endomorphism structure.

This result is used in Mantovan's work~\cite{M} (2008),
to study the generic fibre of certain \emph{Rapoport-Zink spaces}.
First conceived in the book~\cite{RZ} (1996), to whose authors they owe their name,
these spaces are formal schemes over (a finite extension of) $\breve{\Z}_p$,
parametrising $p$-divisible groups with additional structure and rigidified by a quasi-isogeny to
a fixed ``frame'' object modulo $p$ (here, $\breve{\Z}_p$ denotes the ring of integers of
the completion of the maximal unramified extension of $\Q_p$).
A Rapoport-Zink space is determined by a \emph{local PEL datum},
which prescribes the kind of additional structure in form of polarisation,
endomorphism structure and level structure (whence the acronym ``PEL'') and
fixes a number of combinatorial invariants, including the Hodge and the Newton polygons.

In fact, the connection with Rapoport-Zink spaces extends to a broader level.
If, on the one hand, the Hodge-Newton filtration over a good class of $W(\k)$-algebras
leads to properties of their generic fibre (for Hodge-Newton reducible PEL data), on the other hand, 
the Hodge-Newton decomposition over $\k$ corresponds to properties of their special fibre
(here, $\k$ is the algebraic closure of the finite field $\F_p$, so $W(\k)=\breve{\Z}_p$).
Indeed, the $\k$-valued points of a Rapoport-Zink space are in bijection with
a corresponding affine Deligne-Lusztig set, at least in the cases concerned up to this point;
we will come back to this matter in a more general framework.

Further progress in the same direction as Mantovan and Viehmann was made by Shen in \cite{XS} (2013).
Here, besides an application to the generic fibre of Rapoport-Zink spaces,
we find a new proof of the existence of the Hodge-Newton filtration,
for $p$-divisible groups over a complete valuation ring (of rank $1$) $\O_K$ in mixed characteristic,
with perfect residue field $\k$ of characteristic $p$.
The setup is still that of unramified endomorphism structure and the Hodge and
the Newton polygons are defined through the reduction of the $p$-divisible group to $\k$.
However, instead of obtaining the Hodge-Newton filtration over $\O_K$ by
lifting the Hodge-Newton decomposition over $\k$, Shen proves its existence directly,
by means of the Harder-Narasimhan theory for finite flat group schemes,
developed by Fargues in \cite{F1} (2010).

In the more recent articles \cite{H2} (2018) and~\cite{H1} (2019),
Hong further extended the conclusions of Mantovan, Viehmann and Shen to a wider class of
additional structures and the relative generalised Rapoport-Zink spaces
(namely of unramified Hodge type).
For this purpose, the same author developed a tool, called ``EL realisation'', to
reduce the problem to the previously known cases.

All the results mentioned so far deal with additional endomorphism structures of unramified type.
The notion of Hodge-Newton reducibility, however,
was meanwhile formalised in a more general context by Rapoport and Viehmann,
in \cite[Definition~4.28]{RV} (2014), including possible ramification.
The definition is formulated in a group-theoretic way, building upon Kottwitz's language,
as a property of \emph{local Shimura data}, a more general version of local PEL data.
A local Shimura datum determines a \emph{local Shimura variety},
a concept developed in the same paper, which, in the PEL case,
should be realised as the generic fibre of a Rapoport-Zink space.
The work of Rapoport and Viehmann also explains,
in the more general context of the Harris-Viehmann conjecture,
how the Hodge-Newton reducibility should affect the $l$-adic cohomology of
local Shimura varieties with infinite level structure.
In this sense, the results of Mantovan, Shen and Hong include advances in this direction,
in the unramified case.

A notion of Hodge-Newton reducibility in ramified settings can also be found,
relatively to $F$-crystals over $\k$, in the article~\cite{BH} (2017) by Bijakowski and Hernandez.
This does not exactly match the definition in \cite{RV},
in that the Hodge polygon considered by Bijakowski and Hernandez is a different invariant
(we will add more details on this below).
Using their notion, anyway, the authors prove the existence of a decomposition of
$F$-crystals with possibly ramified endomorphism structure, generalising Katz's original result.

The Hodge-Newton reducibility condition from \cite{RV}, instead,
features in the work of Görtz, He and Nie~\cite{GHN} (2019).
Here, allowing ramified setups, the Hodge-Newton decomposition for a very large class of
affine Deligne-Lusztig \emph{varieties} is proved;
in fact, these objects acquired in the meantime the geometric structure of perfect schemes over $\k$,
inside the Witt vector affine flag variety (cf.\ \cite{Z} and \cite{BS}, note that
we are considering here the mixed characteristic version of affine Deligne-Lusztig varieties).
In the article it is conjectured, in full generality, that the bijection between
the $\k$-valued points of Rapoport-Zink spaces and
the corresponding affine Deligne-Lusztig sets upgrades to an isomorphism of perfect schemes,
between the perfection of the special fibre on one side and
an affine Deligne-Lusztig variety on the other.
This brings us to the motivation of our own work.

\bigskip\noindent
The Hodge-Newton decomposition for affine Deligne-Lusztig varieties attached to a ramified datum,
along with the conjectural consequences on the special fibre of
Hodge-Newton reducible Rapoport-Zink spaces, suggests that,
as in the unramified case, a corresponding statement should hold at the level of the generic fibre.
As we saw in the previous overview, this investigation starts from the existence of
the Hodge-Newton filtration for $p$-divisible groups over mixed characteristic base rings.
In this sense, the present work could have a natural application to the study of
the generic fibre of Rapoport-Zink spaces with finite level structure
(application which is in fact in the author's plans).
Passing to infinite level structure,
this could lead to further progress towards the Harris-Viehmann conjecture,
comparing with the work of Gaisin and Imai~\cite{GaI} in this direction
(especially in light of the methods elaborated by Chen, Fargues and Shen in \cite{CFS},
but see also \cite{S} and \cite{C},
based on the theory of vector bundles over the Fargues-Fontaine curve).

Let us mention, in addition, that the conclusions of
this work can be interpreted in terms of $p$-adic Galois representations.
Indeed, the category of $p$-divisible groups over $\O_K$ is equivalent to the category of
Galois stable $\Z_p$-lattices in certain crystalline representations of
the absolute Galois group of $K$ (cf.\ \cite[Corollary~6.2.3]{SW}).

\bigskip\noindent
The first issue that we address in this document are the definitions of the Newton polygon and
the Hodge polygon for $p$-divisible groups over $\O_K$ with additional endomorphism structure.
In order to do this, we make use of the equivalence of categories between $p$-divisible groups over
$\O_K$ up to isogeny and a certain abelian subcategory of weakly admissible filtered isocrystals over
$K$, the field of fractions of $\O_K$.
The definitions, as well as the main property relating the two polygons,
can be dealt with at the level of (weakly admissible) filtered isocrystals,
for which reason the first section is devoted to these objects.

As for the Newton polygon, it arises from a functorial formalism
(namely the \emph{slope decomposition} for isocrystals) and is therefore not affected by
the additional endomorphism structure, except for a simple rescaling process.
The situation is different for the Hodge polygon,
which is based on the group-theoretic definition of $\bar{\mi}$ from \cite[2.4]{RV};
here, $\mi$ is a dominant geometric cocharacter of
the reductive group over $\Q_p$ encoding the additional structure
(see Remark~\ref{rmk-Nwt<Hdg-fic-end} for the translation to the group-theoretic setting).
Passing from $\mi$ to $\bar{\mi}$ amounts to an averaging process over the Galois conjugates of $\mi$;
in this case, the additional endomorphism structure plays a more decisive role.
This discrepancy between the definitions of the two polygons is also behind the fact that
a $p$-divisible group might be Hodge-Newton reducible only if considered with some
additional endomorphism structure, losing this property when neglecting the same additional structure.
In the group-theoretic language of \cite{RV},
the Newton polygon corresponds to the \emph{Newton point} from loc.\ cit.\ 2.1.
In this sense, our definitions of the Newton polygon and the Hodge polygon agree with
those of the corresponding invariants of local Shimura data.

Let us remark that the Newton polygon is actually an invariant of
the reduction of the $p$-divisible group to $\k$
(or, in terms of filtered isocrystals, an invariant of the underlying isocrystal).
For unramified endomorphism structures, even the Hodge polygon can be defined at the level of
$p$-divisible groups over $\k$ (or, more generally, of $F$-crystals over~$\k$),
although it is in general not an invariant up to isogeny at this level
(cf.\ Remark~\ref{rmk-Hdg-unr} and Example~\ref{ex-key-unr}).
This is in fact the approach that we find in the previous literature
(in particular \cite{MV} and \cite{XS}),
based on the group-theoretic concept of \emph{Hodge point} from \cite{RR}.
A generalisation of this definition to possibly ramified setups can be found in \cite{BH};
this notion, however, does not match the Hodge polygon considered here.
Indeed, our definition recovers the one for $p$-divisible groups over $\k$ only in
the unramified case, but in general it does not give an invariant of the reduction,
as illustrated in Example~\ref{ex-Hdg-ram}.

\bigskip\noindent
Concerning the proof of Theorem~\ref{thm},
we followed essentially the same approach as in Shen's work (cf.\ \cite{XS}),
fixing some details about the main input from Harder-Narasimhan theory,
even for the unramified case (see Corollary~\ref{cor-sub-pdv} below).
Given its rather formal nature, in fact, the argument in loc.\ cit.\ adapts well to
the more general situation considered here.
Let us summarise the underlying strategy.

The Hodge-Newton reducibility assumption for $p$-divisible groups over $\O_K$ with
endomorphism structure concerns whole equivariant isogeny classes of objects.
One can then use the fact that the category of $p$-divisible groups over $\O_K$ up to isogeny
admits a \emph{Harder-Narasimhan formalism} (cf.\ \cite{F2});
in other words, to each object is associated a third polygon,
the Harder-Narasimhan polygon, whose break points correspond to
unique sub-$p$-divisible groups up to isogeny.

The first step consists in proving that the point where the Hodge-Newton reducibility
assumption is realised is also a break point of the Harder-Narasimhan polygon
(or rather of a rescaled version of it, as determined by the endomorphism structure).
This yields a sub-$p$-divisible group up to isogeny, which,
due to the functorial nature of the formalism, is stable under any additional endomorphism structure.
This step is taken care of in \S\ref{S-Hdg-Nwt-fic},
at the level of weakly admissible filtered isocrystals;
in fact, both the Harder-Narasimhan formalism and
the notion of Hodge-Newton reducibility can be set up in this context.

In order to upgrade this subobject up to isogeny to an actual sub-$p$-divisible group,
we make use of the fact that finite flat group schemes of $p$-power order over $\O_K$
also admit a Harder-Narasimhan formalism (cf.\ \cite{F1}).
Moreover, the polygons associated to the $p$-power torsion parts of a $p$-divisible group
converge from above to the Harder-Narasimhan polygon of the $p$-divisible group itself.
This family of polygons provides now a finer invariant, bounded from below by
the Harder-Narasimhan polygon of the $p$-divisible group (hence uniformly over an isogeny class).

The next step consists in finding a (similarly uniform)
upper bound for the family of polygons under consideration.
This upper bound is indeed represented by the Hodge polygon, as proved in \S\ref{S-key}.
This section is the technical heart of the whole argument;
in fact, this is essentially the only point where the endomorphism structure really plays a role,
as functoriality takes care of it in the other steps
(see Example~\ref{ex-key} for an analysis of the situation in an emblematic case,
namely that of $p$-divisible $\O_F$-modules).

The double bound obtained on the Harder-Narasimhan polygons of the $p$-power torsion parts
forces these polygons to pass through the critical point as well.
In order to conclude the existence of the required sub-$p$-divisible group,
we finally prove the following statement, which can be seen as a refinement of
the multiplicative-bilocal-étale filtration for $p$-divisible groups over $\O_K$
(cf.\ Remark~\ref{rmk-sub-pdv-mbe}).

\begin{repcor}{\ref{cor-sub-pdv}}
 Let $(H,\iota)$ be a $p$-divisible group over $\O_K$ with endomorphism structure for $\O_F$.
 Suppose that $z$ is a break point of $\HN(H,\iota)$ which also lies on $\HN(H[p],\iota)$.
 Then, there exists a unique $\iota$-stable sub-$p$-divisible group $H_1$ of $H$ such that,
 if $\iota_1$ denotes the restriction of $\iota$ to $H_1$,
 then $\HN(H_1,\iota_1)$ equals the part of $\HN(H,\iota)$ between the origin and $z$.
 Furthermore, if $H_2$ denotes the quotient of $H$ by $H_1$, with induced $\O_F$-action $\iota_2$,
 then $\HN(H_2,\iota_2)$ equals the rest of $\HN(H,\iota)$ after $z$
 (up to a shift of coordinates setting the origin in $z$).
\end{repcor}

Here, $\HN(H[p],\iota)$ denotes the Harder-Narasimhan polygon of the $p$-torsion part of~$(H,\iota)$.
This result is already contained implicitly in the proof of \cite[Theorem~5.4]{XS},
although, as acknowledged to us by the author of loc.\ cit.,
the argument therein only works in the case that $z$ is a break point of
the (renormalised) Harder-Narasimhan polygon of $(H[p^i],\iota)$, for some $i\ge1$.
In order to fill this gap, we combined the argument in loc.\ cit.\ with
some methods from the algorithm in \cite[\S 3]{F2} (cf.\ Remark~\ref{rmk-sub-pdv-prf}).

\paragraph{Acknowledgements.}
For the realisation of this article I am deeply grateful to Ulrich Görtz,
who introduced me to the topic in object and played an inspiring role in addressing it.
I would also like to thank Eva Viehmann, Laurent Fargues, Xu Shen,
Jan Kohlhaase and Jochen Heinloth for helpful discussions and interesting comments.
Finally, I wish to thank the anonymous referee for several suggestions that improved the exposition.

This work coincides in large part with the author's PhD thesis,
completed at the University of Duisburg-Essen (cf.\ \cite{AM});
at the same time, it constitutes an important upgrade of the latter,
bringing the original results to a bigger generality.
The author was partially funded by the SFB/TR 45
(Sonderforschungsbereich/Transregio 45,
\emph{Periods, moduli spaces and arithmetic of algebraic varieties})
and the above-mentioned institution.
During the revision of this article,
the author was funded by the FMJH (Fondation Mathématique Jacques Hadamard),
while staying at the Centre de Mathématiques Laurent Schwartz, École Polytechnique.

\clearpage

\subsection*{Setup and notation}
\addcontentsline{toc}{subsection}{Setup and notation}
\markright{Setup and notation}

The following setup will be in force throughout the document:
\begin{itemize}
 \item $p$ is a fixed prime number;
 \item $\k$ is a perfect field of characteristic $p$;
 \item $W(\k)$ is the ring of Witt vectors with coefficients in $\k$, with fraction field~$K_0$;
 \item $K$ is a totally ramified extension of $K_0$ of degree $e$, with ring of integers~$\O_K$;
 \item $\sigma\colon W(\k)\rightarrow W(\k)$ denotes the lift of the Frobenius map
 $(x\mapsto x^p)\colon\k\rightarrow\k$,
 the same notation is used for the extension of $\sigma$ to $K_0$;
 \item $F$ is a finite extension of $\Q_p$ of degree $d$, with ring of integers $\O_F$.
\end{itemize}

\paragraph{The Newton set.}
For $n\in\N$, we denote by $\Q^n_+:=\Set{(a_i)_{i=1}^n\in\Q^n|a_1\ge\dots\ge a_n}$
the set of decreasing $n$-tuples of rational numbers.
An element $(a_i)_{i=1}^n\in\Q^n_+$ can be interpreted as the concave polygon
$[0,n]\rightarrow\R$ starting at $(0,0)$ and proceeding with slope $a_i$ on $[i-1,i]$.
Here, by a \emph{concave polygon} we mean a piecewise affine linear, continuous, concave function
$[0,N]\rightarrow\R$ (for some $N\in\N$), such that $0\mapsto 0$;
we will often make no distinction between the function and its graph.
There is an obvious notion of \emph{break point},
from which we exclude the extremal points.
The concave polygons corresponding to elements of $\Q^n_+$ are those defined on $[0,n]$
and whose break points lie in $\Z\times\Q$.

The set $\Q^n_+$ is partially ordered by the following rule:
\[
 (a_i)_{i=1}^n\le (b_i)_{i=1}^n \qquad\text{if}\qquad
 \sum_{i=1}^j a_i\le\sum_{i=1}^j b_i \quad\text{for all $1\le j\le n$ and}\quad
 \sum_{i=1}^n a_i=\sum_{i=1}^n b_i.
\]
In terms of the corresponding polygons,
the relation means ``lying below'' and sharing the end point.
With this meaning,
we extend the partial order to the set of concave polygons defined on $[0,n]$.

\paragraph{Note on the normalisation.}
This paper follows the conventions adopted for the Harder-Narasimhan formalism for
weakly admissible filtered isocrystals, as it is set up in \cite[\S 5.2.3]{F2}.
This has the following consequences:
\begin{enumerate}
 \item
 The Newton polygon of an isocrystal is defined inverting the sign of
 the slopes of its isotypical components (as opposed to the usual definition).
 Then, in the context of weakly admissible filtered isocrystals,
 the Newton polygon can be compared directly to the Harder-Narasimhan polygon
 (see Proposition~\ref{HN<Nwt-fic}).
 \item
 For consistency when dealing with filtered isocrystals,
 the type of a filtration is also defined inverting the sign of its jumps
 (as opposed to the definition in \cite[\S I.1]{DOR}).
 \item
 The functor from $p$-divisible groups over $\k$ to isocrystals is set up using
 covariant Dieudonné theory and shifting the slopes by $-1$ (see \eqref{pdv-ic});
 the functor from $p$-divisible groups over $\O_K$ to filtered isocrystals is set up adjusting
 the jumps of the filtration consistently, in order to obtain a weakly admissible object.
 Via these functors, on the one hand the Newton polygon for isocrystals considered here recovers
 the classical Newton polygon for $p$-divisible groups (see Remark~\ref{rmk-Nwt}).
 On the other hand, the Harder-Narasimhan polygons of the $p$-power torsion parts of
 a $p$-divisible group $H$ over $\O_K$ (as defined in \cite[\S 4]{F1}) converge to
 the Harder-Narasimhan polygon of the corresponding weakly admissible filtered isocrystal
 (see the proof of Proposition~\ref{HN-lim}).
 Finally, our definition of the Hodge polygon, which goes through the notion of type mentioned above,
 also recovers the classical invariant in the unramified case (see Remark~\ref{rmk-Hdg-unr}).
\end{enumerate}
We remark that this normalisation agrees with the one adopted in \cite{XS}.

\clearpage

\section{Filtered isocrystals}
\thispagestyle{plain}

\subsection{Isocrystals}

Recall that an \emph{isocrystal} over $\k$ is a finite dimensional $K_0$-vector-space $N$,
together with a $\sigma$-linear bijective endomorphism $\phi\colon N\rightarrow N$.
The \emph{height} of $(N,\phi)$ and its \emph{dimension} are respectively:
\[
 \height(N,\phi):=\dim_{K_0}N\in\N \qquad\text{and}\qquad \dim(N,\phi):=v_p(\det\phi)\in\Z;
\]
here, choosing a $K_0$-basis of $N$, one can write the matrix of $\phi$ in the usual way
and the $p$-adic valuation of its determinant will be independent of the chosen basis.
The isocrystals over $\k$ form a $\Q_p$-linear abelian category $\Isoc(\k)$,
with morphisms given by $K_0$-linear maps compatible with $\phi$.
The dimension function is additive on short exact sequences.\looseness=-1

An isocrystal $(N,\phi)$ is \emph{isotypical} (or \emph{isoclinic}) if there exist
integers $r,s$ with $s>0$ and a $W(\k)$-lattice $M\subseteq N$ such that $\phi^sM=p^rM$.
In this case, we have $r/s=\dim(N,\phi)/\height(N,\phi)\in\Q$
and call this number the \emph{slope} of $(N,\phi)$.

\paragraph{The Newton polygon.}
Because $\k$ is perfect, every isocrystal $(N,\phi)$ over $\k$ has a unique
\emph{slope decomposition}:
\begin{equation}\label{slp-dcp}
 (N,\phi)=\bigoplus_{\lambda\in\Q}(N_\lambda,\phi_\lambda)
\end{equation}
into isotypical sub-isocrystals $(N_\lambda,\phi_\lambda)$ of slope $\lambda$
(cf.\ \cite[6.22]{Z1}); this is functorial in the sense that there are no nonzero morphisms
between isotypical isocrystals of different slope (cf.\ loc.\ cit.\ 6.20).
Say that $\lambda_1<\dots<\lambda_m$ are the slopes appearing nontrivially
in the decomposition (called the \emph{Newton slopes} of $(N,\phi)$)
and let $h_i:=\height(N_{\lambda_i},\phi_{\lambda_i})$, $i=1,\dots,m$,
and $h:=\height(N,\phi)=h_1+\dots+h_m$.
Then, we define the \emph{Newton polygon} of $(N,\phi)$ to be:
\[
 \Newt(N,\phi):=(-\lambda_1^{(h_1)},\dots,-\lambda_m^{(h_m)})\in\Q^h_+,
\]
the superscript denoting the number of repetitions.
This polygon is the concave envelope of the points $(\height(N',\phi'),-\dim(N',\phi'))$
over all sub-isocrystals $(N',\phi')$ of $(N,\phi)$.
Note that the break points of $\Newt(N,\phi)$ lie in $\Z\times\Z$
and we have $\Newt(N,\phi)(h)=-\dim(N,\phi)$.
We remark that compared to the usual definition of Newton polygon for isocrystals,
here we invert the sign of the slopes.
As pointed out in the note at the end of the introduction,
this is in accordance with other conventions adopted throughout in the paper.

\subsection{Isocrystals with coefficients}\label{S-ic-end}

\begin{dfn}
 An \emph{isocrystal} over $\k$ \emph{with coefficients} in $F$ is
 a triple $(N,\phi,\iota)$ consisting of an isocrystal $(N,\phi)$ over $\k$ and
 a map of $\Q_p$-algebras $\iota\colon F\rightarrow\End(N,\phi)$.
\end{dfn}

The isocrystals over $\k$ with coefficients in $F$ form an $F$-linear abelian category
$\Isoc(\k)_F$, with morphisms given by maps of isocrystals compatible with $\iota$.
For objects of this category we still have the notions of height and dimension,
which simply refer to those of the underlying isocrystal.

It is sometimes useful (particularly in view of Remark \ref{rmk-ic-end} below)
to consider isocrystals with coefficients from another point of view,
which we borrow from \cite[\S VIII.5]{DOR}.
First of all,
note that the underlying vector space of an isocrystal over $\k$ with coefficients in $F$
has a module structure over $K_0\otimes_{\Q_p}F$, thanks to the $\Q_p$-linear $F$-action.
Let us study this ring in more detail.

Let $f_F:=f(F|\Q_p)$ be the inertia degree of $F$ over $\Q_p$
and write $K_0^{\sigma^{f_F}}$ for the $\sigma^{f_F}$-fixed subfield of $K_0$;
this is a finite unramified extension of $\Q_p$, say of degree $f$, with $f$ dividing $f_F$.
In fact, we have $K_0^{\sigma^{f_F}}=K_0^{\sigma^f}$, the $\sigma^f$-fixed subfield of $K_0$.
Note that $f$ can be strictly smaller than $f_F$,
without necessarily $K_0=K_0^{\sigma^f}$ being the case
(e.g.\ if $\k=\F_p(T^{1/p^{\infty}})$, the field obtained from the field of
rational functions $\F_p(T)$ over $\F_p$ adjoining the $p^j$-th root of $T$ for all $j\ge1$).

Fix now an embedding $\tau_0\colon K_0^{\sigma^f}\rightarrow F$.
We obtain all the embeddings of $ K_0^{\sigma^f}$ in $F$ as
$\tau_i:=\tau_0\circ\sigma^{-i}\colon K_0^{\sigma^f}\rightarrow F$, for $i\in\Z/f\Z$.
Set then $K_F^{(i)}:=K_0\otimes_{K_0^{\sigma^f}\!,\tau_i}F$, for $i\in\Z/f\Z$,
and $K_F:=K_F^{(0)}$; all these are unramified field extensions of $F$.
Indeed, $\tau_i(K_0^{\sigma^f})$ is contained in the maximal unramified subextension $F^\nr$ of
$F|\Q_p$ and the minimal polynomial of $F^\nr|K_0^{\sigma^f}$ is irreducible over $K_0$,
because the coefficients of any nontrivial monic factor would be fixed by $\sigma^{f_F}$ and
hence would lie in $K_0^{\sigma^{f_F}}=K_0^{\sigma^f}$.
Thus, $K_{F^\nr}^{(i)}:=K_0\otimes_{K_0^{\sigma^f}\!,\tau_i}F^\nr$ is a field and,
in particular, an unramified extension of $F^\nr$.
Then, $K_F^{(i)}=K_{F^\nr}^{(i)}\otimes_{F^\nr}F$ is an unramified field extension of $F$.
We have an isomorphism:
\begin{align*}
 K_0^{\sigma^f}\otimes_{\Q_p}F &\cong\prod_{i=0}^{f-1}F \\
 b\otimes a &\mapsto(\tau_i(b)a)_i,
\end{align*}
which extends to a decomposition:
\begin{equation}\label{K0-dcp}
 K_0\otimes_{\Q_p}F\cong\prod_{i=0}^{f-1}K_F^{(i)}.
\end{equation}
The automorphism $\sigma\otimes\id$ on the left-hand side
corresponds to the product of the isomorphisms:
\[
 \sigma\otimes\id\colon K_F^{(i)}=K_0\otimes_{K_0^{\sigma^f}\!,\tau_i}F\longrightarrow
 K_0\otimes_{K_0^{\sigma^f}\!,\tau_{i+1}}F=K_F^{(i+1)}
\]
on the right-hand side.
In particular, $\sigma^f\otimes\id$ induces an automorphism $\sigma_F\colon K_F\rightarrow K_F$,
the Frobenius of $K_F$ over $F$.

\begin{dfn}
 A \emph{$\sigma_F$-$K_F$-space} is a pair $(N_F,\phi_F)$
 consisting of a finite dimensional $K_F$-vector-space $N_F$ and
 a $\sigma_F$-linear bijective endomorphism $\phi_F\colon N_F\rightarrow N_F$.
\end{dfn}

The $\sigma_F$-$K_F$-spaces form an $F$-linear abelian category $\sigma_F$-$K_F$-$\Sp$,
with morphisms given by $K_F$-linear maps compatible with $\phi_F$.
This category is in fact equivalent to $\Isoc(\k)_F$, as we will see shortly.

Consider the following maps of rings:
\begin{align*}
 \epsilon_0\colon K_0\otimes_{\Q_p}F\cong\prod_{i=0}^{f-1}K_F^{(i)}\xrightarrow{{\pr}_0} K_F,
  \qquad
 \rho\colon K_F &\longrightarrow\prod_{i=0}^{f-1}K_F^{(i)}\cong K_0\otimes_{\Q_p}F \\
 x              &\mapsto((\sigma^i\otimes\id)(x))_i.
\end{align*}
We use them to define functors:
\begin{align*}
 \epsilon_{0,*}\colon\Isoc(\k)_F    &\longrightarrow\text{$\sigma_F$-$K_F$-$\Sp$} \\
 (N,\phi,\iota) &\mapsto(N\otimes_{K_0\otimes_{\Q_p}F,\epsilon_0}K_F,\phi^f\otimes\sigma_F)
\end{align*}
and:
\begin{align*}
 \rho_*\colon\text{$\sigma_F$-$K_F$-$\Sp$}  &\longrightarrow\Isoc(\k)_F \\
 (N_F,\phi_F)   &\mapsto(N_F\otimes_{K_F,\rho}(K_0\otimes_{\Q_p}F),\phi,\iota).
\end{align*}
Here,
$N_F\otimes_{K_F,\rho}(K_0\otimes_{\Q_p}F)$ has an obvious structure of $K_0$-vector-space
and $\phi$ is given with respect to the decomposition:
\[
 N_F\otimes_{K_F,\rho}(K_0\otimes_{\Q_p}F)\cong
 \prod_{i=0}^{f-1}N_F\otimes_{K_F,\sigma^i\otimes\id}K_F^{(i)}
\]
as the product of the maps:
\begin{align*}
 \id\otimes(\sigma\otimes\id)
    &\colon N_F\otimes_{K_F}K_F^{(i)}\longrightarrow N_F\otimes_{K_F}K_F^{(i+1)}
 \quad\text{for }i=0,\dots,f-2, \\
 \phi_F\otimes(\sigma\otimes\id)
    &\colon N_F\otimes_{K_F}K_F^{(f-1)}\longrightarrow N_F\otimes_{K_F}K_F=N_F.
\end{align*}
Finally, $\iota$ is given by the natural $F$-multiplication on $K_0\otimes_{\Q_p}F$.

\begin{lem}[{\cite[8.5.4]{DOR}}]
 The functors $\epsilon_{0,*}$ and $\rho_*$ are quasi-inverse $F$-linear equivalences
 of categories between $\Isoc(\k)_F$ and $\sigma_F$-$K_F$-$\Sp$.
\end{lem}

\begin{rmk}\label{rmk-ic-end}
 As an important consequence of this lemma, given any $(N,\phi,\iota)\in\Isoc(\k)_F$,
 we may always write $N\cong N_F\otimes_{K_F,\rho}(K_0\otimes_{\Q_p}F)$
 for some $K_F$-vector-space $N_F$.
 In particular, $N$ is free as a module over $K_0\otimes_{\Q_p}F$,
 so its $K_0$-dimension is a multiple of $\dim_{K_0}(K_0\otimes_{\Q_p}F)=[F:\Q_p]=d$.
 In other words, $\height(N,\phi)\in d\N$.
 
 Note that if $F$ is a totally ramified extension of $\Q_p$,
 this is more easily granted by the fact that $K_0\otimes_{\Q_p}F$ is itself a field.
 At the other extreme, assume that $F$ is unramified over $\Q_p$
 and admits an embedding into $K_0$ (e.g. if $\k$ is algebraically closed).
 Then, fixing $F\subseteq K_0$, we have a decomposition into sub-$K_0$-vector-spaces:
 \[
  N=\bigoplus_{i\in\Z/d\Z}N_i, \quad\text{with }
  N_i=\Set{v\in N|\forall a\in F\colon\iota(a)(v)=\sigma^i(a)v}.
 \]
 Moreover, $\phi$ induces $\sigma$-linear bijections $N_i\rightarrow N_{i+1}$,
 so that $\dim_{K_0}N_i$ is constant for $i\in\Z/d\Z$ and again
 $\dim_{K_0}N=d\cdot\dim_{K_0}N_0\in d\N$.
 The formalism above combines these last two arguments in the general case.
\end{rmk}

\paragraph{The Newton polygon.}
Let $(N,\phi,\iota)$ be an isocrystal over $\k$ with coefficients in $F$, of height $h=dn$
(note that $h\in d\N$ by the previous remark).
By functoriality of the slope decomposition~\eqref{slp-dcp} of $(N,\phi)$,
the action of $F$ via $\iota$ restricts to each isotypical component
$(N_\lambda,\phi_\lambda)$.
These components are hence again isocrystals with coefficients in $F$,
whose height is then a multiple of $d$.
Thus, in the Newton polygon of $(N,\phi)$, each entry is repeated a multiple of $d$ times.
In light of this, if $\Newt(N,\phi)=(-\lambda_1^{(h_1)},\dots,-\lambda_m^{(h_m)})\in\Q^h_+$,
we define the \emph{Newton polygon} of $(N,\phi,\iota)$ to be:
\[
 \Newt(N,\phi,\iota):=(-\lambda_1^{(h_1/d)},\dots,-\lambda_m^{(h_m/d)})\in\Q^n_+.
\]
Equivalently:
\[
 \Newt(N,\phi,\iota)\colon x\longmapsto\frac{1}{d}\Newt(N,\phi)(dx).
\]
Note that the break points of $\Newt(N,\phi,\iota)$ lie in $\Z\times\frac{1}{d}\Z$.
Furthermore, we have that $\Newt(N,\phi,\iota)(n)=-\dim(N,\phi)/d$.

\subsection{Filtered vector spaces}

Let here $K_2|K_1$ be any field extension and let us recall the category
$\Fil\Vect_{K_2|K_1}$ of \emph{$K_2$-filtered $K_1$-vector-spaces}.
Objects are finite dimensional $K_1$-vector-spaces~$V$ equipped with
a $\Z$-filtration $\Fil^\bullet V_{K_2}=(\Fil^i V_{K_2})_{i\in\Z}$ of
$V_{K_2}:=V\otimes_{K_1}K_2$ by sub-$K_2$-vector-spaces, which is decreasing
(i.e.\ $\Fil^i V_{K_2}\supseteq\Fil^{i+1}V_{K_2}$ for every $i\in\Z$), exhaustive and separated
(i.e.\ respectively $\Fil^i V_{K_2}=V_{K_2}$ and $\Fil^j V_{K_2}=0$ for some integers $i\le j$).
A morphism between two objects $(V'\!,\Fil^\bullet V'_{K_2}),(V,\Fil^\bullet V_{K_2})$
is given by a $K_1$-linear map $f\colon V'\rightarrow V$ whose base change
$f_{K_2}\colon V'_{K_2}\rightarrow V_{K_2}$ is compatible with the filtrations,
meaning that $f_{K_2}(\Fil^i V'_{K_2})\subseteq\Fil^i V_{K_2}$ for all $i\in\Z$;
if $f_{K_2}(\Fil^i V'_{K_2})=f_{K_2}(V'_{K_2})\cap\Fil^i V_{K_2}$ for all $i\in\Z$,
then we say that $f$ is a \emph{strict} morphism.

The category $\Fil\Vect_{K_2|K_1}$ is $K_1$-linear and quasi-abelian,
with short exact sequences given by short sequences of strict morphisms
which are exact on the underlying vector spaces.
In this case, the first term and the last term of the sequence are called respectively
a \emph{subobject} and a \emph{quotient object} of the middle term.

For $(V,\Fil^\bullet V_{K_2})\in\Fil\Vect_{K_2|K_1}$ and $i\in\Z$, we set:
\[
\gr^i V_{K_2}:=\Fil^i V_{K_2}/\Fil^{i+1} V_{K_2} \qquad\text{and}\qquad
\deg(V,\Fil^\bullet V_{K_2}):=\sum_{i\in\Z}i\cdot\dim_{K_2}\gr^i V_{K_2}\in\Z,
\]
called respectively the \emph{$i$-th graded piece} and the \emph{degree}
of $(V,\Fil^\bullet V_{K_2})$.
Note that a short sequence is exact if and only if
it induces exact sequences on all the $i$-th graded pieces.
In particular, the degree function is additive.

The indices $i\in\Z$ such that $\gr^i V_{K_2}\ne 0$ are called the \emph{jumps}
of the filtration.
If these are, say, $i_1<\dots<i_m$ and we let $n_j:=\dim_{K_2}\gr^{i_j}V_{K_2}$,
$j=1,\dots,m$, and $n:=\dim_{K_1}V=n_1+\dots+n_m$,
then the \emph{type} of $(V,\Fil^\bullet V_{K_2})$ is defined to be:
\[
 f(V,\Fil^\bullet V_{K_2}):=(-i_1^{(n_1)},\dots,-i_m^{(n_m)})\in\Q^n_+.
\]
The break points of $f(V,\Fil^\bullet V_{K_2})$ lie in $\Z\times\Z$ and
$f(V,\Fil^\bullet V_{K_2})(n)=-\deg(V,\Fil^\bullet V_{K_2})$.
Moreover, if:
\[
 0\longrightarrow(V'\!,\Fil^\bullet V'_{K_2})\longrightarrow(V,\Fil^\bullet V_{K_2})
  \longrightarrow(V''\!,\Fil^\bullet V''_{K_2})\longrightarrow0
\]
is a short exact sequence of $K_2$-filtered $K_1$-vector-spaces, then:
\[
 f(V'\!,\Fil^\bullet V'_{K_2})=(-i_1^{(n'_1)},\dots,-i_m^{(n'_m)}) \qquad\text{and}\qquad
 f(V''\!,\Fil^\bullet V''_{K_2})=(-i_1^{(n''_1)},\dots,-i_m^{(n''_m)})
\]
with $n'_j+n''_j=n_j$ for all $1\le j\le m$.
Note that compared to the definition of the type of a filtration from \cite[\S I.1]{DOR},
here we invert the sign of the jumps.
This is for consistency with our definition of the Newton polygon for isocrystals,
especially in consideration of the following sections.

We remark that if $K_2'$ is a field extension of $K_2$
and $K_1'$ is a field extension of $K_1$ contained in $K_2'$,
then the obvious base change functor $\Fil\Vect_{K_2|K_1}\rightarrow\Fil\Vect_{K_2'|K_1'}$
is exact and preserves the type.

\begin{lem}\label{lem-fvs}
 Let $(V,\Fil^\bullet V_{K_2})$ be a $K_2$-filtered $K_1$-vector-space
 and $(V'\!,\Fil^\bullet V'_{K_2})$ a subobject. Set $n':=\dim_{K_1}V'$. Then:
 \[
  -\deg(V'\!,\Fil^\bullet V'_{K_2})\le f(V,\Fil^\bullet V_{K_2})(n'),
 \]
 with equality if and only if the type of $(V'\!,\Fil^\bullet V'_{K_2})$ equals
 the restriction of $f(V,\Fil^\bullet V_{K_2})$ to $[0,n']$.
\end{lem}

\begin{proof}
 It is enough to observe that the type of $(V'\!,\Fil^\bullet V'_{K_2})$
 is a polygon with the same slopes as $f(V,\Fil^\bullet V_{K_2})$,
 but with lower (possibly zero) multiplicity.
 
 In numbers, let $i_1<\dots<i_m$ be the jumps of $(V,\Fil^\bullet V_{K_2})$
 and $n_j:=\dim_{K_2}\gr^{i_j}V_{K_2}$, for $j=1,\dots,m$.
 The jumps of $(V'\!,\Fil^\bullet V'_{K_2})$ are among those of $(V,\Fil^\bullet V_{K_2})$,
 with $n'_j:=\dim_{K_2}\gr^{i_j}V'_{K_2}\le n_j$, $j=1,\dots,m$, and $n'=n'_1+\dots+n'_m$.
 Let $l\in\Set{1,\dots,m}$ be such that $n_1+\dots+n_{l-1}<n'\le n_1+\dots+n_l$
 (the case $n'=0$ being trivial).
 Then, the claimed inequality reads:
 \[
  -\sum_{j=1}^m i_jn'_j\le-\sum_{j=1}^{l-1}i_jn_j-i_l(n'-(n_1+\dots+n_{l-1})).
 \]
 Now, for $j\le l-1$ we have $i_j<i_l$, whereas for $j\ge l+1$ we have $i_j>i_l$. Thus:
 \begin{equation}\label{lem-fvs-f1}
  \sum_{j=1}^{l-1}i_j(n_j-n'_j)\le i_l\sum_{j=1}^{l-1}(n_j-n'_j)
  \qquad\text{and}\qquad
  \sum_{j=l+1}^m i_jn'_j\ge i_l\sum_{j=l+1}^m n'_j.
 \end{equation}
 Altogether:
 \begin{multline}\label{lem-fvs-f2}
  \sum_{j=1}^{l-1}i_jn_j+i_l(n'-(n_1+\dots+n_{l-1}))-\sum_{j=1}^m i_jn'_j= \\
  =\sum_{j=1}^{l-1}i_j(n_j-n'_j)+i_l(n'-(n_1+\dots+n_{l-1})-n'_l)
   -\sum_{j=l+1}^m i_jn'_j\le \\
  \le i_l(n'-(n'_1+\dots+n'_l))=0,
 \end{multline}
 which is the desired inequality.
 If $f(V'\!,\Fil^\bullet V'_{K_2})$ equals the restriction of $f(V,\Fil^\bullet V_{K_2})$ to 
 $[0,n']$, then:
 \[
  -\deg(V'\!,\Fil^\bullet V'_{K_2})=f(V'\!,\Fil^\bullet V'_{K_2})(n')
   =f(V,\Fil^\bullet V_{K_2})(n').
 \]
 Conversely, if $-\deg(V'\!,\Fil^\bullet V'_{K_2})=f(V,\Fil^\bullet V_{K_2})(n')$,
 then we have equality in \eqref{lem-fvs-f2} and hence in \eqref{lem-fvs-f1}.
 In particular, $n'_j=n_j$ for $j\le l-1$ and $n'_j=0$ for $j\ge l+1$,
 which implies that $f(V'\!,\Fil^\bullet V'_{K_2})$ equals
 the restriction of $f(V,\Fil^\bullet V_{K_2})$ to $[0,n']$.
\end{proof}

\subsection{Filtered isocrystals}

A \emph{filtered isocrystal} over $K$ is an isocrystal $(N,\phi)$ over $\k$,
together with a decreasing, exhaustive and separated $\Z$-filtration
$\Fil^\bullet N_K$ of $N_K=N\otimes_{K_0}K$ by sub-$K$-vector-spaces.
The filtered isocrystals over $K$ form a $\Q_p$-linear category $\Fil\Isoc_K$,
with morphisms given by maps of isocrystals
whose base change to $K$ is compatible with the filtrations;
a morphism is \emph{strict} if so is the resulting map in $\Fil\Vect_{K|K_0}$.
In fact, $\Fil\Isoc_K$ is a quasi-abelian category too,
with short exact sequences given by short sequences of strict morphisms
which are exact on the underlying isocrystals;
we deduce the meaning of \emph{subobject} and \emph{quotient object}.

The forgetful functors:
\begin{align*}
 \Fil\Isoc_K    &\longrightarrow\Isoc(\k)   &\text{and} &&
 \Fil\Isoc_K    &\longrightarrow\Fil\Vect_{K|K_0}   \\
 (N,\phi,\Fil^\bullet N_K)  &\longmapsto(N,\phi)    &   &&
 (N,\phi,\Fil^\bullet N_K)  &\longmapsto(N,\Fil^\bullet N_K)
\end{align*}
are exact.
Given a filtered isocrystal $(N,\phi,\Fil^\bullet N_K)$ over $K$,
say with $\height(N,\phi)=h$, we define its \emph{Newton polygon} to be:
\[
 \Newt(N,\phi,\Fil^\bullet N_K):=\Newt(N,\phi)\in\Q^h_+
\]
and its \emph{Hodge polygon} to be:
\[
 \Hdg(N,\phi,\Fil^\bullet N_K):=f(N,\Fil^\bullet N_K)\in\Q^h_+.
\]
Furthermore, we define its \emph{Newton number} to be:
\[
 t_N(N,\phi,\Fil^\bullet N_K):=\dim(N,\phi)\in\Z
\]
and its \emph{Hodge number} to be:
\[
 t_H(N,\phi,\Fil^\bullet N_K):=\deg(N,\Fil^\bullet N_K)\in\Z.
\]
In this way,
the Hodge polygon and number only depend on the underlying filtered vector space,
whereas the Newton polygon and number are invariants of the underlying isocrystal.
By the properties already known,
we see that the break points of both polygons lie in $\Z\times\Z$
and that the Newton number and the Hodge number are additive on short exact sequences.
Moreover:
\[
 \Newt(N,\phi,\Fil^\bullet N_K)(h)=-t_N(N,\phi,\Fil^\bullet N_K)
\]
and:
\[
 \Hdg(N,\phi,\Fil^\bullet N_K)(h)=-t_H(N,\phi,\Fil^\bullet N_K).
\]

\subsection{Weakly admissible filtered isocrystals}

A \emph{filtered isocrystal} $(N,\phi,\Fil^\bullet N_K)$ over $K$ is called
\emph{weakly admissible} if for all subobjects $(N'\!,\phi'\!,\Fil^\bullet N'_K)$ we have:
\[
 t_H(N'\!,\phi'\!,\Fil^\bullet N'_K)\le t_N(N'\!,\phi'\!,\Fil^\bullet N'_K),
\]
with equality holding for $(N,\phi,\Fil^\bullet N_K)$ itself.
Let $\Fil\Isoc_K^\wa$ denote the full subcategory of $\Fil\Isoc_K$
consisting of the weakly admissible objects.
This is an abelian category,
with kernels and cokernels coinciding with those in $\Fil\Isoc_K$;
moreover, if two out of three objects of a short exact sequence in $\Fil\Isoc_K$
are weakly admissible, then the third object is weakly admissible too (cf.\ \cite[\S 4.2]{Fo2}).

A fundamental implication of weak admissibility
is expressed by the following inequality of polygons,
which is part of the characterisation given in \cite[\S 4.3]{Fo2}
(taking into account a different normalisation).
It can also be seen as an easy consequence of Lemma~\ref{lem-fvs}
(see the proof of Proposition~\ref{Nwt<Hdg-fic-end} for the argument in a more general setting).

\begin{prp}\label{Nwt<Hdg-fic}
 Let $(N,\phi,\Fil^\bullet N_K)$ be a weakly admissible filtered isocrystal over $K$.
 Then:
 \[
  \Newt(N,\phi,\Fil^\bullet N_K)\le\Hdg(N,\phi,\Fil^\bullet N_K).
 \]
\end{prp}

\paragraph{The Harder-Narasimhan polygon.}
As explained in \cite[\S 5.2.3]{F2},
the category of weakly admissible filtered isocrystals over $K$
admits a Harder-Narasimhan formalism for the slope function $\mi=-t_N/\height=-t_H/\height$.
More precisely, to each nonzero object $\NF=(N,\phi,\Fil^\bullet N_K)\in\Fil\Isoc_K^\wa$
we associate its \emph{slope}:
\[
 \mi(\NF):=\frac{-t_N(N,\phi,\Fil^\bullet N_K)}{\height(N,\phi)}
 =\frac{-t_H(N,\phi,\Fil^\bullet N_K)}{\height(N,\phi)}\in\Q
\]
and say that $\NF$ is \emph{semi-stable} of slope $\mi(\NF)$ if
for every subobject $\NF'\subseteq\NF$ we have $\mi(\NF')\le\mi(\NF)$.
In general, there exists a unique \emph{Harder-Narasimhan filtration}:
\begin{equation}\label{HN-fil-fic}
 0=\NF_0\subsetneq\NF_1\subsetneq\dots\subsetneq\NF_m=\NF
\end{equation}
in $\Fil\Isoc_K^\wa$, such that each $\NF_i/\NF_{i-1}$ is semi-stable, say of slope $\mi_i$,
with $\mi_1>\dots>\mi_m$.
Some functoriality of this filtration follows from the fact that if $\NF'$ and $\NF''$
are two semi-stable objects with $\mi(\NF')>\mi(\NF'')$,
then there are no nontrivial morphisms $\NF'\rightarrow\NF''$.
Now, for $\NF$ as above, with Harder-Narasimhan filtration as in \eqref{HN-fil-fic},
let $h:=\height(N,\phi)$ and let $h_i$ be the height of the underlying isocrystal of
$\NF_i/\NF_{i-1}$, for $i=1,\dots,m$.
Then, we define the \emph{Harder-Narasimhan polygon} of $\NF$ to be:
\[
 \HN(\NF):=(\mi_1^{(h_1)},\dots,\mi_m^{(h_m)})\in\Q^h_+.
\]
This polygon is the concave envelope of the points $(\height(N'\!,\phi'),-t_N(\NF'))$ over
all subobjects $\NF'=(N'\!,\phi'\!,\Fil^\bullet N'_K)$ of $\NF$ in $\Fil\Isoc_K^\wa$.
Its break points lie in $\Z\times\Z$ and we have $\HN(\NF)(h)=-t_N(\NF)=-t_H(\NF)$.

\bigskip\noindent
Given $\NF=(N,\phi,\Fil^\bullet N_K)\in\Fil\Isoc_K^\wa$, every subobject
$\NF'=(N'\!,\phi'\!,\Fil^\bullet N'_K)$ gives rise to
a sub-isocrystal $(N'\!,\phi')$ of $(N,\phi)$, hence the point $(\height(N'\!,\phi'),-t_N(\NF'))$,
which is equal to $(\height(N'\!,\phi'),-\dim(N'\!,\phi'))$, lies below $\Newt(N,\phi)=\Newt(\NF)$.
We easily deduce the following inequality of polygons.

\begin{prp}\label{HN<Nwt-fic}
 Let $(N,\phi,\Fil^\bullet N_K)$ be a weakly admissible filtered isocrystal over $K$.
 Then:
 \[
  \HN(N,\phi,\Fil^\bullet N_K)\le\Newt(N,\phi,\Fil^\bullet N_K).
 \]
\end{prp}

\begin{rmk}
 It may be useful, at this point, to mention another instance of Harder-Narasimhan formalism which
 is implicitely related to the one introduced above,
 although it will not play an active role in our discussion.
 Namely, given any field extension $K_2|K_1$,
 the category of $K_2$-filtered $K_1$-vector-spaces admits a Harder-Narasimhan formalism for
 the slope function $\mi=\deg/\dim$ (cf.\ \cite[\S I.3]{DOR}).
 Similarly as above, to each nonzero object~$(V,\Fil^\bullet V_{K_2})\in\Fil\Vect_{K_2|K_1}$ is
 associated an element of the Newton set $\Q^n_+$, where $n=\dim_{K_1}V$,
 called the \emph{HN-vector} of $(V,\Fil^\bullet V_{K_2})$.
 This is not to be confused with the type of $(V,\Fil^\bullet V_{K_2})$;
 in fact, Lemma~\ref{lem-fvs} implies that,
 setting up the formalism with respect to $-\mi$ (in order to comply with our normalisation),
 then the HN-vector of $(V,\Fil^\bullet V_{K_2})$ lies below its type.

 Now, given a weakly admissible filtered isocrystal $(N,\phi,\Fil^\bullet N_K)$ over $K$,
 its Harder-Narasimhan polygon $\HN(N,\phi,\Fil^\bullet N_K)$ can be recovered as
 the HN-vector of a certain associated $C$-filtered $\Q_p$-vector-space,
 where $C$ is the completion of an algebraic closure of $K$ (note that we are not talking about
 the underlying $K$-filtered $K_0$-vector-space~$(N,\Fil^\bullet N_K)$);
 for the details of this, which go beyond the scope of this article,
 we refer to \cite[Propositions~10 and~11]{F2}.
 The HN-vector (with respect to $-\mu$) of the underlying filtration $(N,\Fil^\bullet N_K)$,
 instead, yields yet another polygon, which lies above $\HN(N,\phi,\Fil^\bullet N_K)$
 (by a similar argument as for Proposition~\ref{HN<Nwt-fic},
 since $-t_N=-t_H$ for weakly admissible objects)
 and below $\Hdg(N,\phi,\Fil^\bullet N_K)$ (by Lemma~\ref{lem-fvs}).
 These observations will not be used in the sequel.
\end{rmk}

\subsection{Filtered isocrystals with coefficients}

\begin{dfn}
 A \emph{filtered isocrystal} over $K$ \emph{with coefficients} in $F$ is a pair $(\NF,\iota)$
 consisting of a filtered isocrystal $\NF=(N,\phi,\Fil^\bullet N_K)$ over $K$
 and a map of $\Q_p$-algebras $\iota\colon F\rightarrow\End(N,\phi,\Fil^\bullet N_K)$.
\end{dfn}

The filtered isocrystals over $K$ with coefficients in $F$ form
an $F$-linear quasi-abelian category $\Fil\Isoc_{K,F}$,
with morphisms given by maps of filtered isocrystals compatible with~$\iota$;
the notions of exactness, subobject and quotient object come from those in $\Fil\Isoc_K$
by neglecting $\iota$.

We have an exact forgetful functor:
\begin{align*}
 \Fil\Isoc_{K,F}    &\longrightarrow\Isoc(\k)_F \\
 (N,\phi,\Fil^\bullet N_K,\iota)    &\longmapsto(N,\phi,\iota),
\end{align*}
where $\iota$ also denotes the induced $F$-action on the underlying isocrystal $(N,\phi)$.
Given $(\NF,\iota)=(N,\phi,\Fil^\bullet N_K,\iota)\in\Fil\Isoc_{K,F}$,
say with $\height(N,\phi)=dn$, we define its \emph{Newton polygon} to be:
\[
 \Newt(\NF,\iota):=\Newt(N,\phi,\iota)\in\Q^n_+.
\]
Equivalently:
\[
 \Newt(\NF,\iota)\colon x\longmapsto\frac{1}{d}\Newt(\NF)(dx).
\]
Thus, the Newton polygon only depends on the underlying isocrystal.
Moreover, as we already know for isocrystals with coefficients in $F$,
the break points of $\Newt(\NF,\iota)$ lie in $\Z\times\frac{1}{d}\Z$
and we have $\Newt(\NF,\iota)(n)=-\dim(N,\phi)/d=-t_N(\NF)/d$.

\paragraph{The Hodge polygon.}
Let $(\NF,\iota)=(N,\phi,\Fil^\bullet N_K,\iota)$ be a filtered isocrystal over $K$
with coefficients in $F$ and pick a field extension $K'$ of $K$
containing all embeddings $\tau$ of $F$ in an algebraic closure of $K$.
Then, $\iota$ extends linearly to an action on the $K'$-vector-space $N_{K'}=N\otimes_{K_0}K'$,
which respects the induced filtration $\Fil^\bullet N_{K'}$.
We obtain decompositions of $K'$-vector-spaces:
\begin{equation}\label{egn-dcp-fic}
\begin{gathered}
 N_{K'}=\bigoplus_{\tau\colon F\rightarrow K'}N_\tau,
 \quad\text{with }N_\tau=\Set{v\in N_{K'}|\forall a\in F\colon\iota(a)(v)=\tau(a)v}; \\
 \Fil^\bullet N_{K'}=\bigoplus_{\tau\colon F\rightarrow K'}\Fil^\bullet N_\tau,
 \quad\text{with }\Fil^\bullet N_\tau=N_\tau\cap\Fil^\bullet N_{K'}.
\end{gathered}
\end{equation}
Now, as seen in Remark~\ref{rmk-ic-end}, $N$ is free as a module over $K_0\otimes_{\Q_p}F$,
so that $N_{K'}$ is free over $K'\otimes_{\Q_p}F\cong\prod_\tau K'$
($b\otimes a\mapsto (b\tau(a))_\tau$);
in particular, $\dim_{K'}N_\tau$ is constant over all embeddings $\tau$, say equal to $n$
(thus, $\height(N,\phi)=dn$).
Let $f_\tau\in\Q^n_+$ be the type of $(N_\tau,\Fil^\bullet N_\tau)\in\Fil\Vect_{K'|K'}$.
Then, we define the \emph{Hodge polygon} of $(\NF,\iota)$ to be:
\[
 \Hdg(\NF,\iota):=\frac{1}{d}\sum_\tau f_\tau\in\Q^n_+.
\]
This is independent of the choice of $K'$,
due to the invariance of the type of a filtration under base change.
In fact, $\Hdg(\NF,\iota)$ only depends on the underlying filtered vector space of $\NF$
and the compatible $F$-action on it;
however, we used that the same action respects the operator $\phi$ on $N$,
in order to ensure the equidimensionality of $N_\tau$ for varying $\tau$.
Finally, note that the break points of $\Hdg(\NF,\iota)$ lie in $\Z\times\frac{1}{d}\Z$
and we have $\Hdg(\NF,\iota)(n)=-t_H(\NF)/d$.

\subsection{Weakly admissible filtered isocrystals with coefficients}

A filtered isocrystal over $K$ with coefficients in $F$ is called \emph{weakly admissible}
if the underlying filtered isocrystal is weakly admissible.
We denote by $\Fil\Isoc_{K,F}^\wa$ the full subcategory of $\Fil\Isoc_{K,F}$
consisting of the weakly admissible objects.
This is an abelian category,
with kernels and cokernels coinciding with those in $\Fil\Isoc_{K,F}$.

Using Lemma~\ref{lem-fvs}, we can easily generalise
the fundamental inequality between the Newton polygon and the Hodge polygon to this setting.

\begin{prp}\label{Nwt<Hdg-fic-end}
 Let $(\NF,\iota)=(N,\phi,\Fil^\bullet N_K,\iota)$ be
 a weakly admissible filtered isocrystal over $K$ with coefficients in $F$.
 Then:
 \[
  \Newt(\NF,\iota)\le\Hdg(\NF,\iota).
 \]
\end{prp}

\begin{proof}
 Let $n:=\height(N,\phi)/d$ and note first that:
 \[
  \Newt(\NF,\iota)(n)=-t_N(\NF)/d=-t_H(\NF)/d=\Hdg(\NF,\iota)(n),
 \]
 the central equality due to weak admissibility; thus, the two polygons share their end point.
 By concavity, it is then enough to show that
 each break point of $\Newt(\NF,\iota)$ lies below $\Hdg(\NF,\iota)$.
 
 Let $(x,y)$ be a break point of $\Newt(\NF,\iota)$;
 by definition, we find an $F$-stable sub-isocrystal $(N'\!,\phi')$ of $(N,\phi)$ such that
 $x=\height(N'\!,\phi')/d$ and $y=-\dim(N'\!,\phi')/d$.
 Let $\Fil^\bullet N'_K:=N'_K\cap\Fil^\bullet N_K$ be the induced filtration,
 so that $\NF':=(N'\!,\phi'\!,\Fil^\bullet N'_K)$ is a sub-filtered-isocrystal of $\NF$;
 moreover, $\iota$ restricts to an $F$-action $\iota'$ on $\NF'$.
 
 Fix now a field extension $K'$ of $K$
 containing all embeddings $\tau$ of $F$ in an algebraic closure of $K$ and let:
 \begin{gather*}
   N_{K'}=\bigoplus_{\tau\colon F\rightarrow K'}N_\tau,
   \qquad\Fil^\bullet N_{K'}=\bigoplus_{\tau\colon F\rightarrow K'}\Fil^\bullet N_\tau, \\
   N'_{K'}=\bigoplus_{\tau\colon F\rightarrow K'}N'_\tau,
   \qquad\Fil^\bullet N'_{K'}=\bigoplus_{\tau\colon F\rightarrow K'}\Fil^\bullet N'_\tau
 \end{gather*}
 be the decompositions as in \eqref{egn-dcp-fic}.
 Note that $N'_\tau=N'_{K'}\cap N_\tau$
 and $\Fil^\bullet N'_\tau=N'_\tau\cap\Fil^\bullet N_\tau$,
 so $(N'_\tau,\Fil^\bullet N'_\tau)$ is a subobject
 of $(N_\tau,\Fil^\bullet N_\tau)$ in $\Fil\Vect_{K'|K'}$, with $\dim_{K'}N'_\tau=x$.
 By Lemma~\ref{lem-fvs}, then:
 \[
  -\deg(N'_\tau,\Fil^\bullet N'_\tau)\le f_\tau(x)
 \]
 for every $\tau$, where $f_\tau$ is the type of $(N_\tau,\Fil^\bullet N_\tau)$.
 On the other hand, weak admissibility gives:
 \[
  dy=-\dim(N'\!,\phi')=-t_N(\NF')\le-t_H(\NF')=-\deg(N'\!,\Fil^\bullet N'_K).
 \]
 Finally, since $\deg(N'\!,\Fil^\bullet N'_K)=\sum_\tau\deg(N'_\tau,\Fil^\bullet N'_\tau)$,
 we get:
 \[
  y\le\frac{1}{d}\sum_\tau f_\tau(x)=\Hdg(\NF,\iota)(x),
 \]
 which means exactly that the point $(x,y)$ lies below the polygon $\Hdg(\NF,\iota)$.
\end{proof}

\begin{rmk}\label{rmk-Nwt<Hdg-fic-end}
 If $K_0=\breve{\Q}_p$ is the completion of the maximal unramified extension of $\Q_p$,
 then the inequality proved above is an instance of the more general fact that
 the $p$-adic period domain associated to a pair $(b,\{\mi\})$ for a reductive group
 $G$ over $\Q_p$ is nonempty if and only if the $\sigma$-conjugacy class $[b]$ is
 ``acceptable'' with respect to $\{\mi\}$ (cf.\ \cite[3.1]{RV}).
 Here, the group $G$ is given by the restriction of scalars $\Res_{F|\Q_p}\GL(V)$,
 where $V$ is any $F$-vector-space such that
 $V\otimes_{\Q_p}K_0$ identifies with the $K_0\otimes_{\Q_p}F$-module~$N$.
 Then, $\phi$ corresponds to an element $b\in G(K_0)$ and $\{\mi\}$ is the conjugacy class of
 any geometric cocharacter of $G$ whose induced grading on $N_{K'}$ splits the filtration
 $\Fil^\bullet N_{K'}$, for a suitable field extension $K'$ of $K$.
 In this setup, the filtration $\Fil^\bullet N_K$ determines a $K$-valued point of
 the $p$-adic period domain and $[b]$ being acceptable with respect to $\{\mi\}$
 translates into the inequality stated in the proposition.
\end{rmk}

\paragraph{The Harder-Narasimhan polygon.}
Let $(\NF,\iota)=(N,\phi,\Fil^\bullet N_K,\iota)$ be a weakly admissible
filtered isocrystal over $K$ with coefficients in $F$, say with $\height(N,\phi)=h=dn$.
By functoriality of the Harder-Narasimhan filtration~\eqref{HN-fil-fic} of $\NF$,
the action of $F$ via $\iota$ restricts to each piece of this filtration,
forcing the height of the respective underlying isocrystal to be a multiple of $d$.
Thus, in the Harder-Narasimhan polygon of $\NF$,
each entry is repeated a multiple of $d$ times.
In light of this, if $\HN(\NF)=(\mi_1^{(h_1)},\dots,\mi_m^{(h_m)})\in\Q^h_+$,
we define the \emph{Harder-Narasimhan polygon} of $(\NF,\iota)$ to be:
\[
 \HN(\NF,\iota):=(\mi_1^{(h_1/d)},\dots,\mi_m^{(h_m/d)})\in\Q^n_+.
\]
Equivalently:
\[
 \HN(\NF,\iota)\colon x\longmapsto\frac{1}{d}\HN(\NF)(dx).
\]
Note that this polygon is the concave envelope of the points
$(\height(N'\!,\phi')/d,-t_N(\NF')/d)$ over all subobjects
$(\NF'\!,\iota')=(N'\!,\phi'\!,\Fil^\bullet N'_K,\iota')$ of $(\NF,\iota)$
in $\Fil\Isoc_{K,F}^\wa$.
Moreover, its break points lie in $\Z\times\frac{1}{d}\Z$
and we have $\HN(\NF,\iota)(n)=-t_N(\NF)/d=-t_H(\NF)/d$.

\bigskip\noindent
Because both the Newton polygon and the Harder-Narasimhan polygon of
a weakly admissible filtered isocrystal with coefficients in $F$ are a rescaled version
(by the same factor $1/d$) of their counterparts obtained neglecting the action of $F$,
Proposition~\ref{HN<Nwt-fic} implies directly the same inequality in this setting.

\begin{prp}\label{HN<Nwt-fic-end}
 Let $(\NF,\iota)=(N,\phi,\Fil^\bullet N_K,\iota)$ be
 a weakly admissible filtered isocrystal over $K$ with coefficients in $F$. Then:
 \[
  \HN(\NF,\iota)\le\Newt(\NF,\iota).
 \]
\end{prp}

\clearpage

\section{\texorpdfstring{$p$}{p}-divisible groups}
\thispagestyle{plain}

Let $R$ be a complete Noetherian commutative local ring,
with residue field of characteristic~$p$.
We denote by $\pdiv_R$ the category of $p$-divisible groups,
alias Barsotti-Tate groups, over $\Spec R$ (or ``over $R$'' for short);
we write $\height H$ for the height and
$H^\vee$ for the Cartier dual of any object $H\in\pdiv_R$ (as defined in \cite[\S 2.3]{T}).
This is a $\Z_p$-linear category, with the hom-sets being torsion free $\Z_p$-modules.
Let then $\pdiv_R\otimes\Q_p$ denote the $\Q_p$-linear category of $p$-divisible groups
over $R$ ``up to isogeny'': objects are again $p$-divisible groups over $R$,
but for $H_1,H_2$ two of them we set
$\Hom_{\pdiv_R\otimes\Q_p}(H_1,H_2):=\Hom_{\pdiv_R}(H_1,H_2)\otimes_{\Z_p}\Q_p
=\Hom_{\pdiv_R}(H_1,H_2)[\frac{1}{p}]$ and extend the composition maps by linearity.
We have an obvious functor $\pdiv_R\rightarrow\pdiv_R\otimes\Q_p$.

Recall that every $p$-divisible group $H=(H[p^i])_{i\ge1}$ over $R$,
say with $H[p^i]=\Spec A_i$,
can be viewed as a formal group $\Spf A$ over $R$, where $A=\varprojlim_{i\ge1}A_i$.
If $H$ is connected, then $\Spf A$ is a formal Lie group, i.e.\ $A$
is isomorphic to the profinite $R$-algebra of formal power series $R[[X_1,\dots,X_r]]$,
for some $r\ge1$ called the \emph{dimension} of $\Spf A$ (cf.\ \cite[\S 2.2]{T}).
In general, there exists a unique exact sequence of $p$-divisible groups over $R$:
\begin{equation}\label{xct-con-et}
 0\longrightarrow H^\circ\longrightarrow H\longrightarrow H^\et\longrightarrow0
\end{equation}
where $H^\circ$ is connected and $H^\et$ is étale (cf.\ loc.\ cit.).
The \emph{dimension} of $H$, denoted by $\dim H$,
is by definition the dimension of the formal Lie group corresponding to $H^\circ$;
we have $\dim H+\dim H^\vee=\height H$ (loc.\ cit.\ Proposition~3).
Next, let $I\subseteq A$ be the augmentation ideal and set $\omega_H:=I/I^2$.
Then, $\omega_H=\omega_{H^\circ}$;
in particular, this is a free $R$-module of rank $\dim H$.
On the other hand,
set $\Lie(H):=\Ker((\Spf A)(R[\epsilon]/(\epsilon^2))\rightarrow(\Spf A)(R))$,
the map being induced by $\epsilon\mapsto0$;
this is again a free $R$-module of rank $\dim H$,
which in fact identifies naturally with the dual of $\omega_H$.
The constructions of $\omega_H$ and $\Lie(H)$ define additive functors to
the category of finitely generated $R$-modules and are compatible with base change along
maps of local rings (that is, ring homomorphisms respecting the maximal ideal).
Finally, note that if $p$ is nilpotent in $R$,
then $H^\circ$ is the same as the formal completion of $H$ along the identity section,
as considered in \cite[\S II]{Me} (this follows from loc.\ cit.\ 4.4, 4.7, 4.11).

Let us next analyse some specific cases of our setup,
namely when $R=\k$ is a perfect field of characteristic $p$ or $R=\O_K$ is
the ring of integers of a finite totally ramified extension $K$ of $K_0=W(\k)[\frac{1}{p}]$.

\paragraph{$p$-divisible groups over $\k$.}
Recall that to each $p$-divisible group $H$ over $\k$ Dieudonné theory associates
a finite free $W(\k)$-module $\D(H)$ of rank $\height H$,
together with an injective $\sigma$-linear endomorphism $\phi_H\colon\D(H)\rightarrow\D(H)$
such that $\phi_H\D(H)\supseteq p\D(H)$.
This kind of objects are called \emph{Dieudonné modules} over $\k$;
they form a $\Z_p$-linear category, with morphisms given by
$W(\k)$-linear maps compatible with the $\sigma$-linear endomorphism.
The association above is in fact functorial and induces a $\Z_p$-linear equivalence of categories between $\pdiv_\k$ and Dieudonné modules over $\k$;
moreover, we have a natural identification of $\k$-vector-spaces:
\begin{equation}\label{Die-ctg}
 \D(H)/\phi_H\D(H)\cong\omega_{H^\vee}
\end{equation}
(cf.\ \cite[\S III]{Fo1}; we are considering here the covariant version of the equivalence,
which is obtained by composing the contravariant functor with duality).
Now, every Dieudonné module $(\D,\phi)$ gives rise to an isocrystal
$(N,\phi):=(\D\otimes_{W(\k)}K_0,\phi\otimes\id)$, whose Newton slopes lie in $[0,1]$.
Indeed, each isotypical component $(N_\lambda,\phi_\lambda)$ of $(N,\phi)$ contains
a $W(\k)$-lattice $M_\lambda:=\D\cap N_\lambda$ with the property that
$pM_\lambda\subseteq\phi_\lambda M_\lambda\subseteq M_\lambda$.
Here, the second inclusion implies that the slope $\lambda$ of $(N_\lambda,\phi_\lambda)$ is
nonnegative, while the first inclusion (which is equivalent to
$p\phi_\lambda^{-1}M_\lambda\subseteq M_\lambda$) implies that $\lambda\le1$.
In fact, the converse statement also holds, namely,
if $(N_\lambda,\phi_\lambda)$ is an isotypical isocrystal of slope $\lambda\in[0,1]$,
then it contains a $W(\k)$-lattice $M_\lambda\subseteq N_\lambda$ such that
$pM_\lambda\subseteq\phi_\lambda M_\lambda\subseteq M_\lambda$.
Indeed, by definition there exist integers $r,s$ and a $W(\k)$-lattice $M\subseteq N_\lambda$ 
such that $\phi_\lambda^sM=p^rM$, with $s>0$ and $r/s=\lambda\in[0,1]$, i.e.\ $0\le r\le s$.
Then, the $W(\k)$-lattice $M_\lambda:=\sum_{i=0}^{s-r}\phi_\lambda^iM+
\sum_{i=1}^{r-1}p^{-i}\phi_\lambda^{s-r+i}M\subseteq N_\lambda$ satisfies the requirements.
Thus, shifting the slopes by $-1$ (for normalisation reasons),
Dieudonné theory induces a $\Q_p$-linear equivalence of categories:
\begin{equation}\label{pdv-ic}
\begin{aligned}
 \pdiv_\k\otimes\Q_p   &\xlongrightarrow{\sim}\Isoc(\k)^{[-1,0]} \\
 H          &\longmapsto(\D(H)\otimes_{W(\k)}K_0,p^{-1}(\phi_H\otimes\id)),
\end{aligned}
\end{equation}
where $\Isoc(\k)^{[-1,0]}$ denotes the full subcategory of $\Isoc(\k)$
consisting of the isocrystals whose Newton slopes lie in $[-1,0]$,
an abelian subcategory of $\Isoc(\k)$.
Note that, if $(N,\phi)$ is the isocrystal associated to a $p$-divisible group $H$
via the functor above, then, taking $\k$-dimensions in \eqref{Die-ctg}, we get that
$\dim(N,\phi)=-\height H+\dim H^\vee=-\dim H $.

\paragraph{$p$-divisible groups over $\O_K$.}
Every $p$-divisible group $H$ over $\O_K$ gives rise to a natural exact sequence
of finite free $\O_K$-modules:
\begin{equation}\label{Hdg-fil}
 0\longrightarrow\omega_{H^\vee}\longrightarrow M(H)\longrightarrow\Lie(H)\longrightarrow0,
\end{equation}
where $M(H)$ is the Lie algebra of the universal extension of $H$
(cf.\ \cite[\S IV.1]{Me}; we obtain the sequence above by taking
the projective limit of all the similar sequences over $\O_K/p^r\O_K$, for $r\ge1$).
In addition, there is a natural isomorphism of $\O_K$-modules
$M(H)\cong\D(H_\k)\otimes_{W(\k)}\O_K$,
whose reduction to $\k$ identifies the image of $\omega_{H_\k^\vee}$ in 
$M(H)\otimes_{\O_K}\k$ with $V\D(H_\k)/p\D(H_\k)\subseteq\D(H_\k)/p\D(H_\k)$;
here, $H_\k$ denotes the reduction of $H$ to $\k$ and $V=p\phi_{H_\k}^{-1}$
the ``Verschiebung'' map of $\D(H_\k)$ (cf.\ \cite[II.15.3]{MM}, \cite[4.3.10]{BGM}).
Thus, if $(N,\phi)$ is the isocrystal over $\k$ associated to $H_\k$ as in \eqref{pdv-ic},
we get a filtered isocrystal $\NF:=(N,\phi,\Fil^\bullet N_K)$ over $K$ by setting
$\Fil^0 N_K$ to be the image of $\omega_{H^\vee}\otimes_{\O_K}K$ in
$N_K\cong M(H)\otimes_{\O_K}K$, with $\Fil^{-i}N_K=N_K$ and $\Fil^i N_K=0$ for all $i\ge1$;
let us check that this is weakly admissible.
First of all:
\begin{align*}
t_H(\NF) &=-(\dim_K N_K-\dim_K(\omega_{H^\vee}\otimes_{\O_K}K)) \\
 &=-\height H+\dim H^\vee=-\dim H=\dim(N,\phi)=t_N(\NF).
\end{align*}
Next, a subobject $\NF'\subseteq\NF$ is given by
a sub-isocrystal $(N'\!,\phi')\subseteq(N,\phi)$,
together with the induced filtration $\Fil^\bullet N'_K=N'_K\cap\Fil^\bullet N_K$.
Set $D':=N'\cap\D(H_\k)$, a direct summand of the free $W(\k)$-module $\D(H_\k)$, and
$E':=N'_K\cap\omega_{H^\vee}$, a direct summand of the free $\O_K$-module $\omega_{H^\vee}$.
Let then $\bar{D'}:=D'/pD'$ and $\bar{E'}:=E'\otimes_{\O_K}\k$.
We have a commutative diagram of inclusions:
\[
\begin{tikzcd}
\omega_{H_\k^\vee} \arrow{r}{\sim} & V\D(H_\k)/p\D(H_\k) \arrow{r}{} & \D(H_\k)/p\D(H_\k) \\
\bar{E'} \arrow{u}{} \arrow{rr}{} & {} & \bar{D'} \arrow{u}{},
\end{tikzcd}
\]
which shows that $\bar{E'}\subseteq VD'/pD'$ (as $VD'=D'\cap V\D(H_\k)$).
In addition, $V^{-1}$ induces an isomorphism
$VD'/pD'\cong D'/\phi_{H_\k}D'$ of $\k$-vector-spaces, so:
\begin{multline*}
 \dim_\k(\bar{E'})\le\dim_k(D'/\phi_{H_\k}D')=v_p(\det(\phi_{H_\k}|D'))=v_p(\det p\phi')= \\
  =\height(N'\!,\phi')+\dim(N'\!,\phi').
\end{multline*}
Then:
\begin{multline*}
 t_H(\NF')=-\dim_K N'_K+\dim_K(N'_K\cap\omega_{H^\vee}\otimes_{\O_K}K)
  =-\height(N'\!,\phi')+\dim_\k(\bar{E'})\le \\
  \le\dim(N'\!,\phi')=t_N(\NF').
\end{multline*}

We obtain a $\Q_p$-linear functor:
\begin{equation}\label{pdv-fic}
\begin{aligned}
 \pdiv_{\O_K}\otimes\Q_p    &\longrightarrow\Fil\Isoc_K^{\wa,[-1,0]} \\
 H  &\longmapsto(N,\phi,\Fil^\bullet N_K),
\end{aligned}
\end{equation}
where $\Fil\Isoc_K^{\wa,[-1,0]}$ denotes the full subcategory of $\Fil\Isoc_K^\wa$
consisting of those weakly admissible filtered isocrystals whose underlying filtration
has jumps in $\Set{-1,0}$, an abelian subcategory of $\Fil\Isoc_K^\wa$;
note that this restriction on the jumps implies that the Newton slopes of
the underlying isocrystal lie in $[-1,0]$, by Proposition~\ref{Nwt<Hdg-fic}
(the converse does not hold: for example, assume that $\k$ is algebraically closed and
consider the simple isotypical isocrystal of height $2$ and dimension $-1$
as in \cite[6.27]{Z1}, together with any filtration with jumps $-2$ and $1$).
The functor just defined is compatible with \eqref{pdv-ic},
if we reduce $p$-divisible groups along $\O_K\rightarrow\k$ on one side
and forget the filtration on the other.
Moreover, if $H\in\pdiv_{\O_K}\otimes\Q_p$ maps to the filtered isocrystal $\NF$,
we have $t_H(\NF)=t_N(\NF)=-\dim H$.

\begin{rmk}\label{rmk-pdv-fic}
 The functor~\eqref{pdv-fic} is in fact an equivalence of categories.
 This was first conjectured by Fontaine in \cite[\S 5.2]{Fo2},
 who already observed fully faithfulness.
 Later, in \cite{Fo3}, he introduced the category of ``crystalline'' $p$-adic representations
 $\Rep_{\cris}(G_K)$ of the absolute Galois group $G_K$ of $K$,
 together with a functor $D_{\cris}\colon\Rep_{\cris}(G_K)\rightarrow\Fil\Isoc_K^\wa$
 such that, for $H$ a $p$-divisible group over $\O_K$ and $V_p(H)$ its rational Tate module,
 $D_{\cris}(V_p(H))$ is the filtered isocrystal associated to $H^\vee$ as in \eqref{pdv-fic}
 (after shifting the filtration and the Newton slopes of the latter by $+1$).
 Now, a theorem of Colmez and Fontaine (cf.\ \cite{CF} and consider there the case $N=0$)
 ensures that every weakly admissible filtered isocrystal over $K$ is ``admissible'',
 meaning that it belongs to the essential image of $D_{\cris}$.
 In turn, a result of Kisin says that every crystalline representation $V$ of $G_K$
 with Hodge-Tate weights in $\Set{0,1}$ (which is equivalent to the underlying filtration
 of $D_{\cris}(V)$ having jumps in $\Set{0,1}$) is isomorphic to the rational Tate module
 of a $p$-divisible group over $\O_K$ (cf.\ \cite[2.2.6]{Ki}).
 It follows that \eqref{pdv-fic} is also essentially surjective.
\end{rmk}

\subsection{\texorpdfstring{$p$}{p}-divisible groups with endomorphism structure}

\begin{dfn}
 Let $R$ be a complete Noetherian commutative local ring,
 with residue field of characteristic $p$.
 A \emph{$p$-divisible group} over $R$ \emph{with endomorphism structure} for $\O_F$
 is a pair $(H,\iota)$ consisting of a $p$-divisible group $H$ over $R$
 and a map of $\Z_p$-algebras $\iota\colon\O_F\rightarrow\End(H)$.
\end{dfn}

Note that, because $\End(H)$ is a torsion free $\Z_p$-module,
the map $\iota$ is automatically injective;
in other words, the $\O_F$-action on $H$ corresponding to $\iota$ is faithful.
We denote by $\pdiv_{R,\O_F}$ the $\O_F$-linear category formed by the objects just defined,
with morphisms given by maps of $p$-divisible groups over $R$ compatible with $\iota$
(or $\O_F$-\emph{equivariant}).
Let then $\pdiv_{R,\O_F}\otimes F$ be the $F$-linear category of $p$-divisible groups over $R$
with endomorphism structure for $\O_F$ ``up to equivariant isogeny'':
this is constructed, as before, from $\pdiv_{R,\O_F}$ by inverting $p$ in the homomorphisms.

The functors \eqref{pdv-ic} and~\eqref{pdv-fic} introduced above
upgrade to $F$-linear equivalences of categories:
\begin{equation}\label{pdv-ic-end}
 \pdiv_{\k,\O_F}\otimes F\xlongrightarrow{\sim}\Isoc(\k)_F^{[-1,0]},
\end{equation}
where $\Isoc(\k)_F^{[-1,0]}$ denotes the full subcategory of $\Isoc(\k)_F$
consisting of objects whose underlying isocrystal has Newton slopes in $[-1,0]$,
an abelian subcategory of $\Isoc(\k)_F$, and:
\begin{equation}\label{pdv-fic-end}
 \pdiv_{\O_K,\O_F}\otimes F\xlongrightarrow{\sim}\Fil\Isoc_{K,F}^{\wa,[-1,0]},
\end{equation}
where $\Fil\Isoc_{K,F}^{\wa,[-1,0]}$ denotes the full subcategory of $\Fil\Isoc_{K,F}^\wa$
consisting of objects whose underlying filtration
has jumps in $\Set{-1,0}$, an abelian subcategory of $\Fil\Isoc_{K,F}^\wa$.
The fact that these functors are equivalences of categories is a formal consequence of
\eqref{pdv-ic} and \eqref{pdv-fic} being themselves equivalences
(see Remark~\ref{rmk-pdv-fic} for the latter),
apart from a small consideration about essential surjectivity.
Namely, an object of the right-hand side category corresponds, by what we formally know,
to some $p$-divisible group $H$ together with
a map $\iota\colon F\rightarrow\End(H)\otimes_{\Z_p}\Q_p$ of $\Q_p$-algebras.
Now, since $\End(H)$ is a finitely generated $\Z_p$-module, e.g.\ because we know that
the hom-sets of (filtered) isocrystals are finite-dimensional $\Q_p$-vector-spaces,
we have that $p^r\iota(\O_F)\subseteq\End(H)$ for $r\ge0$ sufficiently large.
Then, $(H,p^r\iota)$ belongs to the left-hand side category and
maps to some object which is isomorphic (via $p^{-r}$) to our target.

The two functors \eqref{pdv-ic-end} and~\eqref{pdv-fic-end} are compatible with respect to
the reduction of $p$-divisible groups along $\O_K\rightarrow\k$ on one side and
forgetting the filtration on the other.

\paragraph{The polygons.}
For $(H,\iota)$ a $p$-divisible group over $\O_K$ with endomorphism structure for $\O_F$,
we define its \emph{Newton polygon}, its \emph{Hodge polygon} and its
\emph{Harder-Narasimhan polygon} through the functor~\eqref{pdv-fic-end},
to be those of the corresponding
weakly admissible filtered isocrystal $(\NF,\iota)$ over $K$ with coefficients in $F$;
they live in $\Q^n_+$, where $dn=\height H$ is also the height of the underlying isocrystal
of $\NF$, a multiple of $d$ by Remark~\ref{rmk-ic-end}.
By definition, all these are invariants up to $\O_F$-equivariant isogeny.
Moreover, tracking the normalisation of the above-mentioned functor,
we see that $\Newt(H,\iota)(n)=\Hdg(H,\iota)(n)=\HN(H,\iota)(n)=\dim H/d$.
Finally, note that $\Newt(H,\iota)$ only depends on the reduction $(H_\k,\iota)$ of
$(H,\iota)$ to $\k$ (denoting again by $\iota$ the induced $\O_F$-action on $H_\k$).

\begin{rmk}\label{rmk-Hdg-unr}
 If $F$ is unramified over $\Q_p$, the Hodge polygon
 is classically defined at the level of $p$-divisible groups over $\k$ as follows.
 Assume for simplicity that $\k$ is algebraically closed and fix an embedding $F\subseteq K_0$;
 note that we have an isomorphism of $W(\k)$-algebras:
 \[
  W(\k)\otimes_{\Z_p}\O_F\cong\prod_{i\in\Z/d\Z}W(\k),\quad b\otimes a\mapsto(b\sigma^i(a))_i.
 \]
 Now, for $(H,\iota)\in\pdiv_{\k,\O_F}$, say with $\height H=dn$,
 the Dieudonné module $(\D,\phi)$ of $H$ inherits a $\Z_p$-linear $\O_F$-action from $\iota$,
 which makes $\D$ a module over the above ring.
 We obtain a decomposition of $W(\k)$-modules:
 \[
  \D=\bigoplus_{i\in\Z/d\Z}\D_i, \quad\text{with }
  \D_i=\Set{v\in\D|\forall a\in\O_F\colon\iota(a)(v)=\sigma^i(a)v}.
 \]
 Moreover, $\phi$ restricts to $\sigma$-linear injective maps
 $\phi\colon\D_i\rightarrow\D_{i+1}$, so that $\rk_{W(\k)}\D_i=n$ for all $i\in\Z/d\Z$.
 Let then $r_i:=\dim_\k\D_i/\phi\D_{i-1}$ and $f_i:=(1^{(n-r_i)},0^{(r_i)})\in\Q^n_+$ and set:
 \[
  \Hdg(H,\iota):=\frac{1}{d}\sum_{i\in\Z/d\Z} f_i\in\Q^n_+.
 \]
 This definition agrees with the one for $p$-divisible groups over $\O_K$, in the sense that
 for $(H,\iota)\in\pdiv_{\O_K,\O_F}$ we have $\Hdg(H,\iota)=\Hdg(H_{\bar{\k}},\iota)$,
 where $(H_{\bar{\k}},\iota)$ is the reduction of $(H,\iota)$
 to an algebraic closure $\bar{\k}$ of $\k$.
 Indeed, let $K'$ be a finite field extension of $K$ containing $F$ and fix $F\subseteq K'$.
 The $\O_{K'}$-module $\omega_{H^\vee\!,\O_{K'}}:=\omega_{H^\vee}\otimes_{\O_K}\O_{K'}$
 has a $\Z_p$-linear $\O_F$-action induced by $\iota$ and,
 since $\O_{K'}\otimes_{\Z_p}\O_F\cong\prod_{i\in\Z/d\Z}\O_{K'}$ (similarly to before),
 we obtain a decomposition of $\O_{K'}$-modules:
 \[
  \omega_{H^\vee\!,\O_{K'}}=\bigoplus_{i\in\Z/d\Z}\omega_{H^\vee\!,i}, \quad\text{with }
  \omega_{H^\vee\!,i}=
  \Set{v\in\omega_{H^\vee\!,\O_{K'}}|\forall a\in\O_F\colon\iota(a)(v)=\sigma^i(a)v},
 \]
 where $\sigma$ denotes, without ambiguity, the Frobenius automorphism of $F$ over $\Q_p$.
 Now, on the one hand $\omega_{H^\vee\!,i}\otimes_{\O_{K'}}\bar{\k}
 \cong\D(H_{\bar{\k}})_i/\phi_{H_{\bar{\k}}}\D(H_{\bar{\k}})_{i-1}$
 via $\omega_{H^\vee\!,\O_{K'}}\otimes_{\O_{K'}}\bar{\k}=\omega_{H^\vee_{\bar{\k}}}$
 and the natural identification~\eqref{Die-ctg}.
 On the other hand, if $(N,\phi,\Fil^\bullet N_K,\iota)$ is the filtered isocrystal
 with coefficients in $F$  associated to $(H,\iota)$,
 we see that $\omega_{H^\vee\!,i}\otimes_{\O_{K'}}K'\cong\Fil^0 N_{\sigma^i}$,
 with notation as in \eqref{egn-dcp-fic}.
 Hence, for all $i\in\Z/d\Z$, we have that
 $\dim_{\bar{\k}}\D(H_{\bar{\k}})_i/\phi_{H_{\bar{\k}}}\D(H_{\bar{\k}})_{i-1}
 =\dim_{K'}\Fil^0 N_{\sigma^i}$, which yields the desired equality of polygons.
 
 In particular, this shows that if $F$ is unramified over $\Q_p$,
 then the Hodge polygon of a $p$-divisible group over $\O_K$
 with endomorphism structure for $\O_F$ only depends on its reduction $(H_\k,\iota)$ to $\k$.
 This is not true for a general extension $F|\Q_p$, as the following example illustrates.
\end{rmk}

\begin{ex}\label{ex-Hdg-ram}
 Let $\k=\bar{\F}_p$ be an algebraic closure of $\F_p$,
 so that $K_0=\breve{\Q}_p$ is the completion of the maximal unramified extension of $\Q_p$.
 Choose $\sqrt{p}$ a square root of $p$ and let $K=K_0(\sqrt{p})$.
 Let then $F=\Q_p(\pi)$, also with $\pi^2=p$.
 We have two embeddings $\tau_0\colon\pi\mapsto\sqrt{p}$ and
 $\tau_1\colon\pi\mapsto-\sqrt{p}$ of $F$ in $K$.
 
 Consider the Dieudonné module $(\D,\phi)$ over $\k$ given by $\D=W(\k)^2$ and
 the $\sigma$-linear endomorphism $\phi\colon e_1\mapsto e_2, e_2\mapsto pe_1$,
 where $e_1,e_2$ is the standard basis.
 We let $\pi$ act on $(\D,\phi)$ by $\pi\colon e_1\mapsto e_2, e_2\mapsto pe_1$;
 since $\pi^2$ acts as $p$, this defines a map of $\Z_p$-algebras
 $\iota\colon\O_F\rightarrow\End(\D,\phi)$.
 Set now $v_0:=\sqrt{p}e_1+e_2$ and $v_1:=-\sqrt{p}e_1+e_2$,
 elements of $\D_{\O_K}:=\D\otimes_{W(\k)}\O_K=\O_K^2$;
 note that these are eigenvectors for the operator $\pi$,
 with eigenvalue respectively $\sqrt{p}$ and $-\sqrt{p}$.
 In particular, the sub-$\O_K$-modules $L_0:=\O_K\cdot v_0$ and $L_1:=\O_K\cdot v_1$
 of $\D_{\O_K}$ are $\pi$-stable.
 Observe, in addition, that these submodules are direct summands of $\D_{\O_K}$.
 
 Define now a new Dieudonné module $(\D'\!,\phi'):=(\D,\phi)\oplus(\D,\phi)$ over $\k$
 and let $H'$ be the corresponding $p$-divisible group over $\k$.
 We endow $(\D'\!,\phi')$ with the diagonal $\O_F$-action induced by $\iota$;
 this reflects into a $\Z_p$-linear $\O_F$-action on $H'$ (which we denote again by $\iota$),
 making $(H'\!,\iota)$ a $p$-divisible group with endomorphism structure for $\O_F$.
 Consider, for $i=0,1$, the filtration
 $L_{0,i}\subseteq\D'_{\O_K}:=\D'\otimes_{W(\k)}\O_K=\D_{\O_K}\oplus\D_{\O_K}$
 given by $L_0$ in the first factor and $L_i$ in the second factor;
 we have:
 \[
  L_{0,i}\otimes_{\O_K}\k=\k\cdot e_2\oplus\k\cdot e_2=V\D/p\D\oplus V\D/p\D=V'\D'/p\D',
 \]
 where $V=p\phi^{-1}$, $V'=p\phi'^{-1}$.
 Then, by Grothendieck-Messing theory, there exist
 $p$-divisible groups $H_0$, $H_1$ over $\O_K$ which reduce to $H'$ via $\O_K\rightarrow\k$
 and whose corresponding exact sequence as in \eqref{Hdg-fil}
 is given by the filtration $L_{0,i}\subseteq\D'_{\O_K}$, for $i=0,1$ respectively
 (cf.\ \cite[V.1.6]{Me}; we apply the theory over all the rings $\O_K/p^r\O_K$, $r\ge1$;
 the resulting compatible families can be assembled to $p$-divisible groups over $\O_K$
 by \cite[II.4.16]{Me}).
 Since the $\O_F$-action on $\D'_{\O_K}$ respects these filtrations,
 $H_0$ and $H_1$ will carry a $\Z_p$-linear $\O_F$-action lifting $\iota$
 (and hence deserving the same name again), making $(H_0,\iota)$ and $(H_1,\iota)$
 two $p$-divisible groups over $\O_K$ with endomorphism structure for $\O_F$.
 Finally, although these two objects have the same reduction to $\k$, we see that:
 \[
  \Hdg(H_0,\iota)=(1/2,1/2)\in\Q^2_+, \quad\text{while }
  \Hdg(H_1,\iota)=(1,0)\in\Q^2_+.
 \]
 Indeed, setting $N:=\D\otimes_{W(\k)}K_0$, and taking $K'=K$ in \eqref{egn-dcp-fic},
 we have that $N_{\tau_0}=L_0\otimes_{\O_K}K$ and $N_{\tau_1}=L_1\otimes_{\O_K}K$.
\end{ex}

Back to the main discussion,
let us specialise Propositions \ref{Nwt<Hdg-fic-end} and~\ref{HN<Nwt-fic-end}
to the context of $p$-divisible groups and write the following cumulative statement.

\begin{prp}
 Let $(H,\iota)$ be a $p$-divisible group over $\O_K$ with endomorphism structure for $\O_F$.
 Then:
 \[
  \HN(H,\iota)\le\Newt(H,\iota)\le\Hdg(H,\iota).
 \]
\end{prp}

\begin{rmk}
 If $F$ is unramified over $\Q_p$, then the inequality $\Newt(H,\iota)\le\Hdg(H,\iota)$
 is classically seen as a consequence of the generalised Mazur inequality (cf.\ \cite[4.2]{RR}).
 Indeed, as we saw in Remark~\ref{rmk-Hdg-unr}, the Hodge polygon of $(H,\iota)$
 is determined in this case by the ``relative position'' of the pair of $W(\bar{\k})$-lattices
 $(\D(H_{\bar{\k}}),\phi_{H_{\bar{\k}}}\D(H_{\bar{\k}}))$ in $\D(H_{\bar{\k}})[\frac{1}{p}]$.
 The two inequalities, however, are conceptually of a different nature.
 Indeed, the one stated here relates to the nonemptiness of
 $p$-adic period domains (cf.\ Remark~\ref{rmk-Nwt<Hdg-fic-end}),
 which are geometric objects living over (a finite extension of) $K_0$.
 Mazur's inequality, instead, relates to the nonemptiness of
 affine Deligne-Lusztig varieties (cf.\ \cite{KR}), which live over $\k$.
\end{rmk}

\subsection{\texorpdfstring{$p$}{p}-groups}

Let $R$ be a complete Noetherian commutative local ring, with residue field of characteristic~$p$.
We denote by $\pgr_R$ the category of finite flat group schemes of $p$-power order
over $\Spec R$ and call \emph{$p$-groups} over $R$ its objects.
For $X\in\pgr_R$, we write $X^\vee$ for its Cartier dual (as defined in \cite[\S 1.2]{T}) and
$\height X$ for the \emph{height} of $X$, i.e.\ the logarithm to base $p$ of its order;
the height function is additive on short exact sequences.
Moreover, if $X$ has height $h$, then $X=X[p^h]:=\Ker(p^h\colon X\rightarrow X)$,
that is, $X$ is $p^h$-torsion (cf.\ \cite[\S 1]{TO}).

Given a $p$-group $X=\Spec A$ over $R$, with augmentation ideal $I\subseteq A$,
set $\omega_X:=I/I^2$; this defines an additive functor
to the category of finitely generated $R$-modules, which is compatible with base change.
If $X$ is the kernel of an isogeny $H\rightarrow H'$ of $p$-divisible groups over $R$,
then we have an exact sequence of $R$-modules:
\begin{equation}\label{xct-ctg}
 \omega_{H'}\longrightarrow\omega_H\longrightarrow\omega_X\longrightarrow0.
\end{equation}
This follows essentially from the fact that, on the level of formal groups,
the isogeny $H\rightarrow H'$ is a ``topologically flat'' map
(cf.\ \cite[VII\textsubscript{B}, 1.3.1]{SGA3} and loc.\ cit.\ 2.4).

Assume now that $R$ is a discrete valuation ring.
For $X\in\pgr_R$, denote by $X_\eta$ its generic fibre,
that is, its base change to the field of fractions of $R$.
Every closed sub-group-scheme $Y$ of $X_\eta$ yields a closed sub-$p$-group $\bar{Y}$ of $X$,
through the operation of schematic closure in $X$ (cf.\ \cite[\S 2]{R}).
This induces a bijection between the closed sub-group-schemes of $X_\eta$ and
the closed sub-$p$-groups of $X$, preserving closed embeddings;
the inverse is given by taking the generic fibre.
Finally, if $X$ is a closed sub-$p$-group of another object $X'\in\pgr_R$,
then the schematic closure in $X$ coincides with that in $X'$.

\paragraph{The degree.}
Let $X$ be a $p$-group over $\O_K$.
Then, because $X$ is $p$-power torsion and the functor $\omega$ is additive,
$\omega_X$ is a finitely generated torsion $\O_K$-module.
Thus, we may define the \emph{degree} of $X$ to be:
\[
 \deg X:=\frac{1}{e}\length\omega_X\in\frac{1}{e}\Z,
\]
where $\length\omega_X$ means the length of $\omega_X$ as $\O_K$-module
(and remember that $e=[K:K_0]$).
Explicitely,
\[
 \text{if:}\quad\omega_X\cong\bigoplus_{i=1}^r\O_K/b_i\O_K,\qquad
 \text{then:}\quad\deg X=\sum_{i=1}^r v(b_i),
\]
where the $b_i$'s are elements of $\O_K$ and $v$ is the valuation of $K$, normalised at $v(p)=1$.
The degree function is additive on short exact sequences
and satisfies $\deg X+\deg X^\vee=\height X$ (cf.\ \cite[Lemme~4]{F1}).
Moreover, if $X'\xrightarrow{u}X\xrightarrow{v}X''$ is a sequence of $p$-groups over $\O_K$
such that $u$ is a closed embedding, $v\circ u=0$ and $v$ induces an isomorphism
$X_\eta/X'_\eta\xrightarrow{\sim}X''_\eta$ on the generic fibre,
then $\deg X\le\deg X'+\deg X''$, with equality if and only if the sequence:
\[
 0\longrightarrow X'\xlongrightarrow{u}X\xlongrightarrow{v}X''\rightarrow0
\]
is exact (cf.\ \cite[Corollaire~3]{F1}).

\paragraph{The Harder-Narasimhan polygon.}
As explained in \cite[\S 4]{F1}, the category of $p$-groups over $\O_K$
admits a Harder-Narasimhan formalism for the slope function $\mi=\deg/\height$.
More precisely, to each nontrivial object $X\in\pgr_{\O_K}$ we associate its \emph{slope}:
\[
 \mi(X):=\frac{\deg X}{\height X}\in\Q
\]
and say that $X$ is \emph{semi-stable} of slope $\mi(X)$ if
for every closed sub-$p$-group $X'\subseteq X$ we have $\mi(X')\le\mi(X)$.
In general, there exists a unique \emph{Harder-Narasimhan filtration}:
\begin{equation}\label{HN-fil-pgr}
 0=X_0\subsetneq X_1\subsetneq\dots\subsetneq X_m=X
\end{equation}
by closed sub-$p$-groups, such that each $X_i/X_{i-1}$ is semi-stable,
say of slope $\mi_i$, with $\mi_1>\dots>\mi_m$.
Some functoriality of this filtration follows from the fact that if $X'$ and $X''$
are two semi-stable objects with $\mi(X')>\mi(X'')$,
then there are no nontrivial morphisms $X'\rightarrow X''$.
Now, for $X$ as above, with Harder-Narasimhan filtration as in \eqref{HN-fil-pgr},
let $h:=\height X$ and let $h_i:=\height X_i/X_{i-1}$ for $i=1,\dots,m$.
Then, we define the \emph{Harder-Narasimhan polygon} of $X$ to be:
\[
 \HN(X):=(\mi_1^{(h_1)},\dots,\mi_m^{(h_m)})\in\Q^h_+.
\]
This polygon is the concave envelope of the points $(\height X',\deg X')$
over all closed sub-$p$-groups $X'$ of $X$.
Its break points lie in $\Z\times\frac{1}{e}\Z$ and we have $\HN(X)(h)=\deg X$.
Moreover, the following compatibility with respect to duality holds (cf.\ \cite[Corollaire~8]{F1}):
\begin{align*}
 &\HN(X^\vee)=((1-\mi_m)^{(h_m)},\dots,(1-\mi_1)^{(h_1)}), \qquad\text{or:} \\
 &\HN(X^\vee)\colon x\longmapsto x+\HN(X)(h-x)-\deg X.
\end{align*}

\subsection{\texorpdfstring{$p$}{p}-groups with endomorphism structure}

\begin{dfn}
 Let $R$ be a complete Noetherian commutative local ring,
 with residue field of characteristic $p$.
 A \emph{$p$-group} over $R$ \emph{with endomorphism structure} for $\O_F$
 is a pair $(X,\iota)$ consisting of a $p$-group $X$ over $R$
 and a map of rings $\iota\colon\O_F\rightarrow\End(H)$.
\end{dfn}

We denote by $\pgr_{R,\O_F}$ the category formed by the objects just defined,
with morphisms given by maps of $p$-groups over $R$ compatible with $\iota$
(or $\O_F$-\emph{equivariant}).

\paragraph{The Harder-Narasimhan polygon.}
Let $(X,\iota)$ be a $p$-group over $\O_K$ with endomorphism structure for $\O_F$,
say with $\height X=h$.
We define the \emph{Harder-Narasimhan polygon} of $(X,\iota)$ to be:
\begin{align*}
 \HN(X,\iota)\colon[0,h/d]  &\longrightarrow\R \\
 x  &\longmapsto\frac{1}{d}\HN(X)(dx).
\end{align*}
This is a concave polygon, whose break points lie in $\frac{1}{d}\Z\times\frac{1}{de}\Z$;
in particular, we are not talking about an element of the Newton set, in general.
Anyway, the $\O_F$-action on $X$ given by $\iota$ restricts to each piece of its
Harder-Narasimhan filtration~\eqref{HN-fil-pgr}, by functoriality of the latter
(see also \cite[\S 5.3]{F1});
this means that $\HN(X,\iota)$ is the concave envelope of the points
$(\height X'/d,\deg X'/d)$
over the closed sub-$p$-groups $X'$ of $X$ that are stable under $\iota$.
Finally, we have $\HN(X,\iota)(h/d)=\deg X/d$ and
the following compatibility with respect to duality:
\[
 \HN(X^\vee\!,\iota^\vee)\colon x\longmapsto x+\HN(X,\iota)(h/d-x)-\deg X/d,
\]
where $\iota^\vee$ denotes the $\O_F$-action induced by $\iota$ on $H^\vee$ through
functoriality of Cartier duality.

\paragraph{Truncated $p$-divisible groups.}
Let $(H,\iota)$ be a $p$-divisible group over $\O_K$ with endomorphism structure for $\O_F$,
say with $\height H=dn$.
For $i\ge1$, the $p$-groups $H[p^i]$ over $\O_K$ inherit a linear $\O_F$-action from $\iota$;
denoting this action again by $\iota$, we obtain a family $(H[p^i],\iota)_{i\ge1}$
of $p$-groups over $\O_K$ with endomorphism structure for $\O_F$.
We have $\height H[p^i]=idn$.
On the other hand, looking at the exact sequence~\eqref{xct-ctg} associated to
the isogeny $p^i\colon H\rightarrow H$, we see that $\omega_{H[p^i]}=\omega_H/p^i\omega_{H}$;
since $\omega_H$ a free $\O_K$-module of rank $\dim H$, we get $\deg H[p^i]=i\dim H$.
Now, for $i\ge1$, consider the \emph{renormalised Harder-Narasimhan polygons}:
\begin{align*}
 \HN^\r(H[p^i],\iota)\colon[0,n] &\longrightarrow\R \\
 x  &\longmapsto\frac{1}{i}\HN(H[p^i],\iota)(ix).
\end{align*}
These are concave polygons, whose break points lie in $\frac{1}{id}\Z\times\frac{1}{ide}\Z$.
Moreover, we have $\HN^\r(H[p^i],\iota)(n)=\dim H/d$ for every $i\ge1$.
The following result from \cite{F2} relates these polygons to the Harder-Narasimhan polygon of
$(H,\iota)$.

\begin{prp}\label{HN-lim}
 Let $(H,\iota)$ be a $p$-divisible group over $\O_K$ with endomorphism structure for $\O_F$
 and set $n:=\height H/d$.
 Then, for $i\to\infty$, the sequence of functions:
 \[
  \HN^\r(H[p^i],\iota)\colon[0,n]\rightarrow\R
 \]
 converges uniformly to the function:
 \[ 
  \HN(H,\iota)\colon[0,n]\rightarrow\R,
 \]
 which is equal to their infimum.
 In fact, for $i\ge1$ and $k\ge1$, we have:
 \[
  \HN^\r(H[p^{ki}],\iota)\le\HN^\r(H[p^i],\iota).
 \]
\end{prp}

\begin{proof}
 Note first that the polygons under consideration are just a rescaled version
 (by the same factor $1/d$) of their counterparts obtained neglecting the action of $\O_F$;
 since the rescaling process does not affect the properties stated here,
 we may indeed neglect the mentioned action.
 
 The convergence statement is \cite[Théorème~1]{F2}.
 Then, it follows from loc.\ cit.\ Théorème 5 that the limit coincides with
 the Harder-Narasimhan polygon of $H$ as it is defined here;
 indeed, the functor~\eqref{pdv-fic} is an equivalence of categories
 (cf.\ Remark~\ref{rmk-pdv-fic}),
 which makes the function $\dim/\height$ on $\pdiv_{\O_K}\otimes\Q_p$
 correspond to the slope function $\mi$ on $\Fil\Isoc_K^\wa$.
 For the final statement, see the proof of loc.\ cit.\ Proposition 2.
\end{proof}

\begin{rmk}\label{HN-vee-pdv}
 From the convergence statement and the compatibility of the Harder-Narasimhan polygon of
 $p$-groups with respect to duality, we deduce that similarly,
 for $(H,\iota)\in\pdiv_{\O_K,\O_F}$, we have:
 \[
  \HN(H^\vee\!,\iota^\vee)\colon x\longmapsto x+\HN(H,\iota)(\height H/d-x)-\dim H/d,
 \]
 where $\iota^\vee$ denotes the $\O_F$-action induced by $\iota$ on $H^\vee$ through
 functoriality of Cartier duality.
\end{rmk}

The family of polygons $(\HN^\r(H[p^i],\iota))_{i\ge1}$ concerns $(H,\iota)$
as an object of $\pdiv_{\O_K,\O_F}$ and not just up to isogeny; in this sense,
it provides more expendable information about the structure of our $p$-divisible group.
As we just saw, these polygons are bounded from below by $\HN(H,\iota)$,
which is instead determined up to isogeny.
We now come to finding an upper bound, similarly uniform over the whole isogeny class
(cf.\ Proposition~\ref{key}).

\subsection{Comparison between Harder-Narasimhan and Hodge polygon}\label{S-key}

Given a $p$-divisible group $(H,\iota)$ over $\O_K$ with endomorphism structure for $\O_F$,
we would like to compare the two polygons $\HN(H[p],\iota)$ and $\Hdg(H,\iota)$ and
ultimately prove Proposition~\ref{key}.
The way the endomorphism structure affects the definitions of
these two polygons is radically different,
namely a rescaling process for $\HN(H[p],\iota)$ and an averaging process for $\Hdg(H,\iota)$.
In order to address this discrepancy, we make use of different factors
depending on whether the endomorphism structure is of an unramified or a ramified nature.

Before proving the proposition, let us recall from \cite{XS} the main argument for the unramified case.
After that, we will review the algebra required to isolate the ramified component in the general case
and prove the key lemma that will allow us to deal with it.

\bigskip\noindent
Let $H\rightarrow H'$ be an isogeny of $p$-divisible groups over $\k$ with kernel $X\in\pgr_\k$,
consider the dual isogeny~$H'^{\,\vee}\rightarrow H^{\,\vee}$ and
let $\D(H'^{\,\vee})\rightarrow\D(H^\vee)$ be the corresponding map of Dieudonné modules.
Because this map induces an isomorphism on the isocrystals,
it is injective and its cokernel $C$ is a $W(\k)$-module of finite length.
We obtain an exact sequence of $W(\k)$-modules:
\begin{equation}\label{xct-Die}
 0\longrightarrow\D(H'^{\,\vee})\longrightarrow\D(H^\vee)\longrightarrow C\longrightarrow0.
\end{equation}
Here, in fact, $\D(H'^{\,\vee})$ and $\D(H^\vee)$ coincide with
the contravariant Dieudonné modules of $H'$ and $H$ respectively.
Hence, $C$ can be identified with the contravariant Dieudonné module of
the $p$-group~$X$ (cf.\ \cite[\S III.1]{Fo1});
in particular, we have that $\length_{W(\k)}C=\height X$ (loc.\ cit.\ Théorème~1(iii)).

Assume now that $H$ and $H'$ are endowed with endomorphism structure for $\O_F$
and that $H\rightarrow H'$ is $\O_F$-equivariant.
Then, the sequence above has an induced $\Z_p$-linear $\O_F$-action,
which makes it a sequence of $W(\k)\otimes_{\Z_p}\O_F$-modules.

If we assume further that $F$ is unramified over $\Q_p$ and admits an embedding into $K_0$,
then, fixing $F\subseteq K_0$, we have an isomorphism of $W(\k)$-algebras:
\begin{equation}\label{W-dcp-unr}
 W(\k)\otimes_{\Z_p}\O_F\cong\prod_{i\in\Z/d\Z}W(\k),\qquad b\otimes a\mapsto(b\sigma^i(a))_i.
\end{equation}
Thus, the sequence~\eqref{xct-Die} splits as a direct sum of exact sequences:
\begin{equation}\label{xct-Die-egn}
 0\longrightarrow\D(H'^{\,\vee})_i\longrightarrow\D(H^\vee)_i\longrightarrow C_i\longrightarrow0,
 \qquad i\in\Z/d\Z,
\end{equation}
with $\O_F$ acting through $\sigma^i\colon\O_F\rightarrow W(\k)$ on the $i$-th component.

The following lemma applies when these sequences arise from a situation over $\O_K$.
It is the core of the proof of \cite[3.10]{XS};
for convenience, we report it here as an isolated statement, together with its argument.
The example that comes after illustrates that the same conclusion might fail,
if we do not make sure that the isogeny in question possesses an $\O_F$-equivariant lift to $\O_K$.

\begin{lem}\label{key-unr}
 Assume that $F$ is unramified over $\Q_p$ and admits an embedding into $K_0$; fix $F\subseteq K_0$.
 Let $(H,\iota)$ and $(H'\!,\iota')$ be
 $p$-divisible groups over $\O_K$ with endomorphism structure for $\O_F$
 and $H\rightarrow H'$ an $\O_F$-equivariant isogeny with kernel $X$.
 Let then $H_\k\rightarrow H_\k'$ denote the reduction of the isogeny to $\k$ and
 consider the exact sequences of $W(\k)$-modules:
 \[
  0\longrightarrow\D(H_\k'^{\,\vee})_i\longrightarrow\D(H_\k^\vee)_i\longrightarrow
  C_i\longrightarrow0,\qquad i\in\Z/d\Z,
 \]
 as in \eqref{xct-Die-egn}.
 Then, $\length_{W(\k)}C_i=\frac{1}{d}\height X$ for every $i\in\Z/d\Z$.
\end{lem}

\begin{proof}
 For $i\in\Z/d\Z$, the $\sigma$-linear endomorphism $\phi:=\phi_{H_\k^\vee}$ of $\D(H_\k^\vee)$
 restricts to injective maps $\phi\colon\D(H_\k^\vee)_{i-1}\rightarrow\D(H_\k^\vee)_i$.
 Reasoning similarly for $\D(H_\k'^{\,\vee})$ and using that the map
 $\D(H_\k'^{\,\vee})\rightarrow\D(H_\k^\vee)$ is compatible with the $\sigma$-linear endomorphisms,
 we obtain commutative diagrams:
 \[
 \begin{tikzcd}
  0 \arrow{r}{} & \D(H_\k'^{\,\vee})_{i-1} \arrow{r}{} \arrow{d}{\phi}
   & \D(H_\k^\vee)_{i-1} \arrow{r}{} \arrow{d}{\phi}
   & C_{i-1} \arrow{r}{} \arrow{d}{\phi} & 0 \\
  0 \arrow{r}{} & \D(H_\k'^{\,\vee})_i \arrow{r}{}
   & \D(H_\k^\vee)_i \arrow{r}{}    & C_i \arrow{r}{} & 0
 \end{tikzcd}
 \]
 with exact rows and the first two columns being injective.
 Now, the argument of Remark~\ref{rmk-Hdg-unr} shows that the numbers
 $\dim_\k\D(H_\k^\vee)_i/\phi\D(H_\k^\vee)_{i-1}$ are determined by the filtered isocrystal
 with coefficients in $F$ associated to $(H^\vee\!,\iota^\vee)$,
 where $\iota^\vee$ denotes the dual action induced by $\iota$.
 However, since $H^\vee$ and $H'^{\,\vee}$ are $\O_F$-equivariantly isogenous,
 the respective associated filtered isocrystals with coefficients in $F$ are isomorphic.
 In particular:
 \[
  \dim_\k\D(H_\k^\vee)_i/\phi\D(H_\k^\vee)_{i-1}=
  \dim_\k\D(H_\k'^{\,\vee})_i/\phi\D(H_\k'^{\,\vee})_{i-1}
 \]
 for every $i\in\Z/d\Z$.
 By the commutative diagram above, we conclude that $\length C_i=\length C_{i-1}$,
 i.e.\ $\length C_i$ is constant over $i\in\Z/d\Z$.
 Finally, since these lengths sum to $\height X$,
 we must have $\length C_i=\frac{1}{d}\height X$ for every $i\in\Z/d\Z$.
\end{proof}

\begin{ex}\label{ex-key-unr}
 Let $\k=\bar{\F}_p$, so that $K_0=\breve{\Q}_p$,
 and let $F=\Q_{p^2}$ be the quadratic unramified extension of $\Q_p$;
 fix an embedding $F\subseteq K_0$.
 Let $(N,\phi)$ be the isocrystal over $\k$ given by $N:=K_0^4$,
 with canonical basis $e_1,f_1,e_2,f_2$, and:
 \[
  \phi\colon\quad e_1\mapsto pf_1,\quad f_1\mapsto e_1,\quad e_2\mapsto f_2,\quad f_2\mapsto pe_2.
 \]
 Let then $F$ act on $N$ through the fixed embedding $F\subseteq K_0$ on $\braket{e_1,e_2}$ and
 through $\sigma$ on $\braket{f_1,f_2}$;
 note that this action respects the $\sigma$-linear endomorphism $\phi$ and
 hence defines a map of $\Q_p$-algebras $\iota\colon F\rightarrow\End(N,\phi)$.
 Consider now the two $\phi$-stable and $\O_F$-stable $W(\k)$-lattices $M,M'\subseteq N$
 generated respectively by $\Set{e_1,f_1,e_2,f_2}$ and $\Set{e_1,pf_1,e_2,f_2}$.
 Let then $(H,\iota)$ and $(H',\iota')$ be the $p$-divisible groups over $\k$ with
 endomorphism structure for $\O_F$ whose contravariant Dieudonné modules are respectively
 $M$ and $M'$, together with the restriction of $\phi$, where
 $\iota$ and $\iota'$ correspond to the restriction of $\iota$ to the respective Dieudonné module.
 Now, the inclusion $M'\subseteq M$ translates into an $\O_F$-equivariant isogeny $H\rightarrow H'$.
 However, the cokernel $M/M'$ (which represents the module $C$ of \eqref{xct-Die})
 is given by a $\k$-vector-space of dimension 1, generated by $f_1$.
 Thus, the conclusion of the previous lemma does not hold in this case
 (only the component relative to $i=1$ has nonzero length).
 In particular, the isogeny $H\rightarrow H'$ cannot afford an $\O_F$-equivariant lift
 to $\O_K$.
 
 Let us also point out that, computing the Hodge polygons with the definition outlined in
 Remark~\ref{rmk-Hdg-unr} (after dualising, since there we used the covariant Dieudonné module),
 we obtain:
 \[
  \Hdg(H^\vee\!,\iota^\vee)=(1,0)\qquad\text{and}\qquad\Hdg(H'^{\,\vee}\!,\iota'^{\,\vee})=(1/2,1/2),
 \]
 where $\iota^\vee,\iota'^{\,\vee}$ denote the dual $\O_F$-action of $\iota,\iota'$ respectively.
 Thus, the classical definition of the Hodge polygon for
 $p$-divisible groups over $\k$ with endomorphism structure for $\O_F$ does not provide,
 in general, an invariant up to $\O_F$-equivariant isogeny.
\end{ex}

Back to the main discussion, let now $F$ be any finite extension of $\Q_p$.
Denote by $F^\nr$ the maximal unramified subextension (or \emph{inertia subfield})
of $F|\Q_p$ and by $\O_{F^\nr}$ its ring of integers,
so that $F^\nr|\Q_p$ is an unramified extension of degree $f(F|\Q_p)$, the inertia degree of $F|\Q_p$,
and $F|F^\nr$ is a totally ramified extension of degree $e(F|\Q_p)$,
the ramification index of $F|\Q_p$.

Assume that $K$ contains all embeddings $\tau$ of $F$ in an algebraic closure.
In particular, $F^\nr$ admits an embedding into $K_0$;
note that, given $\upsilon_0\colon F^\nr\rightarrow K_0$,
we obtain all the embeddings of $F^\nr$ in $K_0$ as
$\sigma^i\circ\upsilon_0$ for $1\le i\le f(F|\Q_p)$.
Then, applying \eqref{W-dcp-unr} to $F^\nr$ and base changing along $\O_{F^\nr}\rightarrow\O_F$,
we get the following isomorphism of $W(\k)$-algebras:
\begin{equation}\label{W-dcp}
 W(\k)\otimes_{\Z_p}\O_F\cong
  \prod_{\upsilon\colon F^\nr\rightarrow K_0}W(\k)\otimes_{\upsilon,\O_{F^\nr}}\O_F,
\end{equation}
where $\upsilon$ runs through all the embeddings of $F^\nr$ in $K_0$.
In the language of \S\ref{S-ic-end}, here we have that $f=f_F$,
as any embedding of $F^\nr$ into $K_0$ identifies $F^\nr$ with $K_0^{\sigma^f}$.
In fact, the isomorphism above is compatible with \eqref{K0-dcp},
up to inverting $p$ and identifying each
factor~$W(\k)\otimes_{\upsilon,\O_{F^\nr}}\O_F[\frac{1}{p}]=K_0\otimes_{\upsilon,F^\nr}F$ with
the corresponding $K_F^{(i)}$.
A consequence of \eqref{W-dcp} is that every $W(\k)\otimes_{\Z_p}\O_F$-module $P$ decomposes as:
\begin{equation}\label{egn-dcp-unr}
 P=\bigoplus_\upsilon P_\upsilon, \quad\text{with }
  P_\upsilon=\Set{w\in P|\forall a\in\O_{F^\nr}\colon(1\otimes a)w=(\upsilon(a)\otimes1)w},
\end{equation}
where each $P_\upsilon$ is a module over the ring $W(\k)\otimes_{\upsilon,\O_{F^\nr}}\O_F$.
Note that this is a discrete valuation ring; in fact,
it is the ring of integers of a totally ramified extension of $K_0$ of degree $e(F|\Q_p)$.
After base change along $W(\k)\rightarrow\O_K$, the isomorphism~\eqref{W-dcp} extends to:
\begin{equation}\label{OK-dcp}
 \O_K\otimes_{\Z_p}\O_F\cong\prod_\upsilon\O_K\otimes_{\upsilon,\O_{F^\nr}}\O_F,
\end{equation}
inducing a similar decomposition of every $\O_K\otimes_{\Z_p}\O_F$-module
into a direct sum of modules over the rings $\O_K\otimes_{\upsilon,\O_{F^\nr}}\O_F$.

Fix now $\upsilon\colon F^\nr\rightarrow K_0$ and consider the following isomorphism of $K$-algebras:
\[
 K\otimes_{\upsilon,F^\nr}F\cong\prod_{\tau|\upsilon}K,\qquad b\otimes a\mapsto(b\tau(a))_\tau,
\]
where $\tau|\upsilon$ stands for the embeddings of $F$ in $K$ which restrict to $\upsilon$ on $F^\nr$.
If $M$ is an $\O_K\otimes_{\upsilon,\O_{F^\nr}}\O_F$-module,
then $M_K:=M\otimes_{\O_K}K$ is a module over the above ring and therefore we have a decomposition:
\begin{equation}\label{egn-dcp-ram}
 M_K=\bigoplus_{\tau|\upsilon}M_{K,\tau}, \enskip\text{with }
  M_{K,\tau}=\Set{w\in M_K|\forall a\in F\colon(1\otimes a)w=(\tau(a)\otimes1)w}.
\end{equation}
Note that even when $M$ is torsion free, so that $M\subseteq M_K$,
this decomposition does not, in general, descend to a decomposition of $M$ itself.
For instance, let $M=\O_K\otimes_{\upsilon,\O_{F^\nr}}\O_F$ and
consider the same quadratic setup as in Example~\ref{ex-Hdg-ram}, i.e.\ $F=\Q_p(\pi)$ with $\pi^2=p$.
Our current assumptions imply that $K$ contains a square root $\sqrt{p}$ of $p$.
Then, $\O_F$ acts via the embedding~$\tau_1\colon\pi\mapsto\sqrt{p}$ on
$\O_K\cdot(\sqrt{p}\otimes1+1\otimes\pi)\subseteq M$ and
via the other embedding~$\tau_2\colon\pi\mapsto-\sqrt{p}$ on
$\O_K\cdot(-\sqrt{p}\otimes1+1\otimes\pi)\subseteq M$;
these two submodules, however, do not sum to the whole of $M$.
Nevertheless, this example describes essentially the case where we can apply the following lemma,
that is, when the $\O_K\otimes_{\upsilon,\O_{F^\nr}}\O_F$-module $M$ arises as
the base change of a finite free $W(\k)\otimes_{\upsilon,\O_{F^\nr}}\O_F$-module $P$.
In this situation, the next statement ensures a similar regularity property to
the one established by the previous lemma.

\begin{lem}\label{key-ram}
 Assume that $K$ contains all embeddings $\tau$ of $F$ in an algebraic closure and
 fix $\upsilon\colon F^\nr\rightarrow K_0$.
 Let $g\colon P\rightarrow Q$ be a surjective homomorphism of modules over the ring
 $W(\k)\otimes_{\upsilon,\O_{F^\nr}}\O_F$,
 with $P$ finitely generated and free and $Q$ of finite length.
 Consider the base change $g'\colon P_{\O_K}\rightarrow Q_{\O_K}$ along $W(\k)\rightarrow\O_K$ and
 write:
 \[
  P_K:=P_{\O_K}\otimes_{\O_K}K=\bigoplus_{\tau|\upsilon}P_{K,\tau}
 \]
 as in \eqref{egn-dcp-ram}.
 Let $I\subseteq\Set{\tau\colon F\rightarrow K|\tau|\upsilon}$ and
 set $P_I:=P_{\O_K}\cap\bigoplus_{\tau\in I}P_{K,\tau}\subseteq P_K$.
 Then:
 \[
  \length_{\O_K}g'(P_I)=\frac{|I|}{e(F|\Q_p)}\length_{\O_K}Q_{\O_K},
 \]
 where $|I|$ denotes the cardinality of $I$.
\end{lem}

\begin{proof}
 Set, for short, $W_{\O_F,\upsilon}(\k):=W(\k)\otimes_{\upsilon,\O_{F^\nr}}\O_F$.
 This being a discrete valuation ring and, hence, a principal ideal domain,
 we may find a $W_{\O_F,\upsilon}(\k)$-basis of $P$ such that,
 under the corresponding identification $P\cong W_{\O_F,\upsilon}(\k)^r$
 (if $r\in\N$ is the rank of $P$), the homomorphism $g$ is represented by a canonical projection:
 \[
  g\colon W_{\O_F,\upsilon}(\k)^r\longrightarrow\bigoplus_{j=1}^r W_{\O_F,\upsilon}(\k)/(a_j),
 \]
 for some elements $a_j\in W_{\O_F,\upsilon}(\k)$ different from zero (as $Q$ is of finite length).
 Then, we may write:
 \[
  g'\colon (\O_K\otimes_{\upsilon,\O_{F^\nr}}\O_F)^r\longrightarrow
   \bigoplus_{j=1}^r (\O_K\otimes_{\upsilon,\O_{F^\nr}}\O_F)/(a_j'),
 \]
 where $a_j'$ is the image of $a_j$ under the inclusion
 $W_{\O_F,\upsilon}(\k)\subseteq\O_K\otimes_{\upsilon,\O_{F^\nr}}\O_F$, for $j=1,\dots,r$.
 The construction of $P_I$, as well as $g'(P_I)$,
 is compatible with this presentation of the map as a direct sum of $r$ components.
 Thus, by additivity of the length function, we may work independently on each factor and assume,
 without loss of generality, that the map~$g$ is of the form:
 \[
  g\colon W_{\O_F,\upsilon}(\k)\rightarrow W_{\O_F,\upsilon}(\k)/(a),
 \]
 for some element $a\in W_{\O_F,\upsilon}(\k)$.
 Letting $a'\in \O_K\otimes_{\upsilon,\O_{F^\nr}}\O_F$ be the image of $a$,
 we then have a short exact sequence of $\O_K\otimes_{\upsilon,\O_{F^\nr}}\O_F$-modules:
 \begin{equation}\label{ses-key-ram}
  0\longrightarrow\O_K\otimes_{\upsilon,\O_{F^\nr}}\O_F\xlongrightarrow{a'}
   \O_K\otimes_{\upsilon,\O_{F^\nr}}\O_F\xlongrightarrow{g'}
   (\O_K\otimes_{\upsilon,\O_{F^\nr}}\O_F)/(a')\longrightarrow 0.
 \end{equation}
 Set now $d':=e(F|\Q_p)$ and choose an ordering $\tau_1,\dots,\tau_{d'}$ of the set
 $\Set{\tau\colon F\rightarrow K|\tau|\upsilon}$ such that $I=\Set{\tau_1,\dots,\tau_{\vert I\vert}}$.
 We consider the filtration of $\O_K\otimes_{\upsilon,\O_{F^\nr}}\O_F$-modules:
 \begin{equation}\label{fil-key-ram}
  \O_K\otimes_{\upsilon,\O_{F^\nr}}\O_F=P_{\O_K}=P_{d'}\supseteq\dots\supseteq P_0=0
 \end{equation}
 given by $P_s:=P_{\O_K}\cap\bigoplus_{l=1}^s P_{K,\tau_l}\subseteq P_K$, $s=0,\dots,d'$,
 so that $P_{\vert I\vert}=P_I$.
 Regarding these objects as $\O_K$-modules, the above definition gives a filtration by direct summands;
 in fact, for $s=1,\dots,d'$, the inclusion $P_{\O_K}\subseteq P_K$ induces an embedding of
 the graded piece $P_s/P_{s-1}$ into $P_{K,\tau_s}$, ensuring that $P_s/P_{s-1}$ is torsion free.
 In addition, since this embedding is $\O_F$-linear,
 we see that the ring $\O_K\otimes_{\upsilon,\O_{F^\nr}}\O_F$ acts on $P_s/P_{s-1}$ via the map:
 \[
  \tau_s'\colon\O_K\otimes_{\upsilon,\O_{F^\nr}}\O_F\rightarrow
   \O_K,\qquad b\otimes c\mapsto b\tau_s(c).
 \]
 Choose now an $\O_K$-basis of $P_{\O_K}$ adapted to the above filtration.
 Then, as an $\O_K$-linear homomorphism, the first map of \eqref{ses-key-ram}
 (that is, multiplication by $a'$ on $\O_K\otimes_{\upsilon,\O_{F^\nr}}\O_F$)
 is represented by a triangular matrix $A$,
 with the entries $\tau_1'(a'),\dots,\tau_{d'}'(a')$ on the diagonal.
 In particular, we have:
 \begin{equation}\label{lg1-key-ram}
  \length_{\O_K}Q_{\O_K}=\length_{\O_K}(\O_K\otimes_{\upsilon,\O_{F^\nr}}\O_F)/(a')=
   e\cdot v(\det A)=e\cdot\sum_{l=1}^{d'}v(\tau_l'(a')),
 \end{equation}
 where $v$ is the valuation of $K$, normalised at $v(p)=1$.
 On the other hand, restricting $g'$ to $P_I\subseteq P_{\O_K}$ and observing that
 $P_I\cap a'P_{\O_K}=a'P_I$, we obtain a new exact sequence of $\O_K$-modules:
 \[
  0\longrightarrow P_I\xlongrightarrow{a'} P_I\xlongrightarrow{g'}g'(P_I)\longrightarrow 0.
 \]
 Here, the endomorphism $a'\colon P_I\rightarrow P_I$ may be represented by a triangular matrix $A_I$,
 with the entries $\tau_1'(a'),\dots,\tau_{\vert I\vert}'(a')$ on the diagonal.
 In particular, we have:
 \begin{equation}\label{lg2-key-ram}
  \length_{\O_K}g'(P_I)=e\cdot v(\det A_I)=e\cdot\sum_{l=1}^{\vert I\vert}v(\tau_l'(a')).
 \end{equation}
 Recall now that $a'$ comes from an element $a$ of $W_{\O_F,\upsilon}(\k)$,
 which is the ring of integers of a field extension $K_0'$ of $K_0$ of degree $d'$.
 Moreover, the restrictions of the maps $\tau_l'$ to $W_{\O_F,\upsilon}(\k)$,
 for $l=1,\dots,d'$, correspond to the different embeddings of $K_0'$ in $K$ over $K_0$.
 Thus, all the elements $\tau_l'(a')\in\O_K$ have the same valuation.
 The desired formula follows then from \eqref{lg1-key-ram} and \eqref{lg2-key-ram}.
\end{proof}

\begin{rmk}
 In the preceding proof, it is essential that
 $a'\in\O_K\otimes_{\upsilon,\O_{F^\nr}}\O_F$ comes from some $a\in W_{\O_F,\upsilon}(\k)$,
 in order for the various elements $\tau_l'(a')\in\O_K$ to have the same valuation.
 Consider, for instance, the usual quadratic case $F=\Q_p(\pi)$ with $\pi^2=p$,
 where we have two embeddings $\tau_1\colon\pi\mapsto\sqrt{p}$ and
 $\tau_2\colon\pi\mapsto-\sqrt{p}$ of $F$ in $K$.
 For $a'=\sqrt{p}\otimes1+1\otimes\pi\in\O_K\otimes_{\Z_p}\O_F$, then,
 we get that $\tau_1'(a')=2\sqrt{p}$, whereas $\tau_2'(a')=0$.
\end{rmk}

We can now prove the following crucial proposition.

\begin{prp}\label{key}
 Let $(H,\iota)$ be a $p$-divisible group over $\O_K$ with endomorphism structure for $\O_F$.
 Then:
 \[
  \HN(H[p],\iota)\le\Hdg(H,\iota).
 \]
\end{prp}

\begin{proof}
 Set $h:=\height H$ and let $(\NF,\iota)$ be the filtered isocrystal over $K$ with
 coefficients in $F$ associated to $(H,\iota)$ via \eqref{pdv-fic-end},
 so that $\Hdg(H,\iota)=\Hdg(\NF,\iota)$.
 Up to replacing $K$ by a finite extension $K'$,
 we may assume that $K$ contains all embeddings $\tau$ of $F$ in an algebraic closure.
 Indeed, on the one hand, this step is anyway embedded in the definition of $\Hdg(\NF,\iota)$;
 moreover, the underlying filtration of $\NF$ is determined by $H$ through the embedding
 $\omega_{H^\vee}\hookrightarrow M(H)$ as in \eqref{Hdg-fil},
 which is compatible with respect to base change along $\O_K\rightarrow\O_{K'}$.
 On the other hand, $\HN(H[p],\iota)$ can only increase after the same base change
 (but in fact it does not, cf.\ \cite[Proposition~6]{F1}).
 
 We need to prove that for every $\iota$-stable closed sub-$p$-group $X'$ of $H[p]$ we have:
 \begin{equation}\label{key-clm1}
  \frac{\deg X'}{d}\le\Hdg(\NF,\iota)\left(\frac{\height X'}{d}\right).
 \end{equation}
 Set $X:=H[p]/X'$ and consider the dual $p$-group $X^\vee$ and the dual $p$-divisible group
 $H^\vee\!$, with the dual action $\iota^\vee$ induced by $\iota$.
 Then, $X^\vee$ is an $\iota^\vee$-stable closed sub-$p$-group of $H[p]^\vee\!$, inducing an
 $\O_F$-equivariant isogeny $H^\vee\rightarrow H^\vee/X^\vee$ with kernel $X^\vee$.
 After reducing to $\k$,
 we obtain an exact sequence of $W(\k)\otimes_{\Z_p}\O_F$-modules:
 \begin{equation}\label{key-1}
  0\longrightarrow\D(H'_\k)\longrightarrow\D(H_\k)\longrightarrow C\longrightarrow0
 \end{equation}
 as in \eqref{xct-Die}, where $H':=(H^\vee\!/X^\vee)^\vee$.
 Moreover, because $X^\vee\subseteq H[p]^\vee$ and the functor $\D$ is $\Z_p$-linear,
 we have that $C$ is a $p$-torsion module.
 
 Let now $F^\nr$ denote the inertia subfield of $F|\Q_p$, with ring of integers $\O_{F^\nr}$.
 By the decomposition in \eqref{egn-dcp-unr},
 the exact sequence above splits as a direct sum of exact sequences:
 \begin{equation}\label{key-2}
  0\longrightarrow\D(H'_\k)_\upsilon\longrightarrow\D(H_\k)_\upsilon
   \longrightarrow C_\upsilon\longrightarrow0
 \end{equation}
 of $W(\k)\otimes_{\upsilon,\O_{F^\nr}}\O_F$-modules,
 where $\upsilon$ runs through all the embeddings of $F^\nr$ in $K_0$.
 This decomposition coincides with that in \eqref{xct-Die-egn},
 if we restrict $\iota$ to $\O_{F^\nr}$.
 In particular, by Lemma~\ref{key-unr}, we have:
 \begin{equation}\label{key-3}
  \length_{W(\k)}C_\upsilon=\frac{1}{[F^\nr:\Q_p]}\height X^\vee=\frac{1}{[F^\nr:\Q_p]}\height X
 \end{equation}
 for all $\upsilon$'s.
 
 The $\O_F$-equivariant isogeny $H^\vee\rightarrow H^\vee\!/X^\vee$ also induces
 an exact sequence of modules over $\O_K\otimes_{\Z_p}\O_F$:
 \[
  0\longrightarrow\omega_{H^\vee\!/X^\vee}\longrightarrow\omega_{H^\vee}
   \longrightarrow\omega_{X^\vee}\longrightarrow0
 \]
 as in \eqref{xct-ctg}, the first map being injective because $\omega_{H^\vee\!/X^\vee}$ and
 $\omega_{H^\vee}$ are free $\O_K$-modules of the same rank and $\omega_{X^\vee}$ is torsion.
 Taking into consideration the natural exact sequence~\eqref{Hdg-fil},
 we obtain a commutative diagram of $\O_K\otimes_{\Z_p}\O_F$-modules:
 \[
 \begin{tikzcd}
  0 \arrow{r}{} & \omega_{H^\vee\!/X^\vee} \arrow{r}{} \arrow{d}{}
   & \omega_{H^\vee} \arrow{r}{} \arrow{d}{} & \omega_{X^\vee} \arrow{r}{} \arrow{d}{} & 0 \\
  0 \arrow{r}{} & M(H') \arrow{r}{} & M(H) \arrow{r}{} & C_{\O_K} \arrow{r}{} & 0
 \end{tikzcd}
 \]
 with exact rows, where the first two columns are direct summands (as free $\O_K$-modules) and
 hence the last column is injective as well;
 the notation $C_{\O_K}$ for the quotient of $M(H)$ by $M(H')$ is justified by the fact that
 the lower row identifies with the base change of \eqref{key-1} along $W(\k)\rightarrow\O_K$.
 Due to the isomorphism~\eqref{OK-dcp},
 this diagram splits as a direct sum of commutative diagrams:
 \begin{equation}\label{key-dgr}
 \begin{tikzcd}
  0 \arrow{r}{} & \omega_{H^\vee\!/X^\vee\!,\upsilon} \arrow{r}{} \arrow{d}{}
   & \omega_{H^\vee\!,\upsilon} \arrow{r}{} \arrow{d}{}
   & \omega_{X^\vee\!,\upsilon} \arrow{r}{} \arrow{d}{} & 0 \\
  0 \arrow{r}{} & M(H')_\upsilon \arrow{r}{} & M(H)_\upsilon \arrow{r}{}
   & C_{\O_K,\upsilon} \arrow{r}{} & 0
 \end{tikzcd}
 \end{equation}
 of modules over $\O_K\otimes_{\upsilon,\O_{F^\nr}}\O_F$,
 for $\upsilon$ varying as before.
 The properties of the previous diagram are preserved and the lower row identifies now with
 the base change of \eqref{key-2} along $W(\k)\rightarrow\O_K$.
 
 Recall that the underlying filtered vector space of $(\NF,\iota)$ is given by the
 $K_0$-vector-space $N:=\D(H_\k)[\frac{1}{p}]$
 (whose base change $N_K:=N\otimes_{K_0}K$ identifies with $M(H)_K:=M(H)\otimes_{\O_K}K$),
 with $\Fil^0 N_K=\omega_{H^\vee\!,K}:=\omega_{H^\vee}\otimes_{\O_K}K$,
 $\Fil^{-i}N_K=M(H)_K$ and $\Fil^i N_K=0$ for all $i\ge1$.
 Write:
 \[
   M(H)_{K}=\bigoplus_{\tau\colon F\rightarrow K}M(H)_{K,\tau},
   \qquad\omega_{H^\vee\!,K}=\bigoplus_{\tau\colon F\rightarrow K}\omega_{H^\vee\!,K,\tau}
 \]
 as in \eqref{egn-dcp-fic}, set $h_\tau:=\dim_K\omega_{H^\vee\!,K,\tau}$ and remember that
 $\dim_K M(H)_{K,\tau}=h/d$ for all $\tau$'s.
 Then:
 \[
  \Hdg(\NF,\iota)\left(\frac{\height X'}{d}\right)=
   \Hdg(\NF,\iota)\left(\frac{h-\height X}{d}\right)=
   \frac{1}{d}\sum_\tau\min\Set{\frac{h}{d}-h_\tau,\frac{h-\height X}{d}}.
 \]
 On the other side:
 \[
  \frac{\deg X'}{d}=\frac{1}{d}(\dim H-\deg X)=
   \frac{1}{d}(h-\dim H^\vee-\height X+\deg X^\vee).
 \]
 However,
 $\dim H^\vee=\rk_{\O_K}\omega_{H^\vee}=\dim_K\omega_{H^\vee\!,K}=\sum_\tau h_\tau$, so:
 \[
  h-\dim H^\vee-\height X=\sum_\tau\left(\frac{h-\height X}{d}-h_\tau\right).
 \]
 Thus, \eqref{key-clm1} may be rewritten as:
 \begin{equation}\label{key-clm1.5}
  \deg X^\vee\le\sum_\tau\min\Set{\frac{\height X}{d},h_\tau}.
 \end{equation}
 Observe now that:
 \[
  \deg X^\vee=\frac{1}{e}\length\omega_{X^\vee}=
   \frac{1}{e}\sum_\upsilon\length\omega_{X^\vee\!,\upsilon}
 \]
 and, using \eqref{key-3}:
 \[
  \frac{\height X}{d}=\frac{\length C_\upsilon}{d'}=
   \frac{\length C_{\O_K,\upsilon}}{ed'},
 \]
 where $d':=[F:F^\nr]$ is the ramification index of $F$ over $\Q_p$ and
 the lengths are meant as $W(\k)$-modules or, when applicable, as $\O_K$-modules
 (in fact, since $K$ is totally ramified over $K_0$, this does not really matter,
 as the only simple module is $\k$ in both cases).
 As a consequence, grouping the right-hand side of \eqref{key-clm1.5} to a sum over $\upsilon$,
 it suffices to show that for every embedding $\upsilon$ of $F^\nr$ in $K_0$ we have:
 \begin{equation}\label{key-clm2}
  \length\omega_{X^\vee\!,\upsilon}\le
   \sum_{\tau|\upsilon}\min\Set{\frac{\length C_{\O_K,\upsilon}}{d'},eh_\tau},
 \end{equation}
 where $\tau|\upsilon$ are the embeddings of $F$ in $K$ which restrict to $\upsilon$ on $F^\nr$.
 
 Fix $\upsilon\colon F^\nr\rightarrow K_0$ and, for the sake of brevity,
 set:
 \[
  \omega:=\omega_{H^\vee\!,\upsilon},\qquad\bar{\omega}:=\omega_{X^\vee\!,\upsilon},\qquad
  M:=M(H)_\upsilon,\qquad\bar{M}:=C_{\O_K,\upsilon}
 \]
 for the $\O_K\otimes_{\upsilon,\O_{F^\nr}}\O_F$-modules in the right square of \eqref{key-dgr}.
 Note that:
 \[
  M_K:=M\otimes_{\O_K}K=\bigoplus_{\tau|\upsilon}M(H)_{K,\tau},\qquad
  \omega_K:=\omega\otimes_{\O_K}K=\bigoplus_{\tau|\upsilon}\omega_{H^\vee\!,K,\tau}
 \]
 and this coincides with the decomposition described in \eqref{egn-dcp-ram}.
 We shall prove that for every subset
 $I\subseteq\Set{\tau\colon F\rightarrow K|\tau|\upsilon}$ we have:
 \[
  \length\bar{\omega}\le\frac{|I|}{d'}\length\bar{M}+\sum_{\tau|\upsilon,\tau\notin I}eh_\tau,
 \]
 where $|I|$ denotes the cardinality of $I$.
 
 Fix then such a subset $I$ and set $M_I:=M\cap\bigoplus_{\tau\in I}M(H)_{K,\tau}\subseteq M_K$;
 here, since $M$ is a torsion free $\O_K$-module, we have an inclusion $M\subseteq M_K$.
 Moreover, this inclusion induces an embedding of $M/M_I$ into
 $\bigoplus_{\tau|\upsilon,\tau\notin I}M(H)_{K,\tau}$.
 In particular, the $\O_K$-module~$M/M_I$ is torsion free,
 meaning that $M_I$ is a direct summand of $M$.
 In a similar fashion,
 set $\omega_I:=\omega\cap\bigoplus_{\tau\in I}\omega_{H^\vee\!,K,\tau}\subseteq\omega_K$,
 a direct summand of the free $\O_K$-module $\omega$ by the same argument as above.
 Observe, in addition, that $\omega_I\otimes_{\O_K}K=\bigoplus_{\tau\in I}\omega_{H^\vee\!,K,\tau}$,
 so that:
 \[
  \rk_{\O_K}\omega_I=\sum_{\tau\in I}h_\tau\qquad\text{and}
   \qquad\rk_{\O_K}\omega/\omega_I=\sum_{\tau|\upsilon,\tau\notin I}h_\tau.
 \]
 Denote by $\bar{M}_I$ the image of $M_I$ in $\bar{M}$ and
 by $\bar{\omega}_I$ the image of $\omega_I$ in $\bar{\omega}$.
 In order to estimate the length of $\bar{\omega}$, we are going to analyse separately
 the contributions of $\bar{\omega}_I$ and $\bar{\omega}/\bar{\omega}_I$.
 
 On the one hand, the surjection $\omega\twoheadrightarrow\bar{\omega}$ factors to
 $\omega/\omega_I\twoheadrightarrow\bar{\omega}/\bar{\omega}_I$.
 Since $\bar{\omega}/\bar{\omega}_I$ is a $p$-torsion module,
 the latter map factors further through $(\omega/\omega_I)/p(\omega/\omega_I)$.
 Thus:
 \[
  \length\bar{\omega}/\bar{\omega}_I\le\length(\omega/\omega_I)/p(\omega/\omega_I)=
   \length\O_K/p\O_K\cdot\rk_{\O_K}\omega/\omega_I=e\cdot\sum_{\tau|\upsilon,\tau\notin I}h_\tau.
 \]
 On the other hand, the injection $\bar{\omega}\hookrightarrow\bar{M}$ restricts to
 $\bar{\omega}_I\hookrightarrow\bar{M}_I$;
 indeed, because the map $\omega\rightarrow M$ is $\O_F$-linear, it sends $\omega_I$ to $M_I$.
 Hence, we have that $\length\bar{\omega}_I\le\length\bar{M}_I$.
 Recall now that the map $M\rightarrow\bar{M}$ identifies with
 the base change along $W(\k)\rightarrow\O_K$ of the surjective homomorphism of
 $W(\k)\otimes_{\upsilon,\O_{F^\nr}}\O_F$-modules
 $\D(H_\k)_\upsilon\rightarrow C_\upsilon$ from \eqref{key-2}.
 Recall, in addition, that $\D(H_\k)_\upsilon$ is finitely generated and free as a module over $W(\k)$.
 But $W(\k)\otimes_{\upsilon,\O_{F^\nr}}\O_F$ is again a discrete valuation ring,
 finite over $W(\k)$.
 Therefore, as a $W(\k)\otimes_{\upsilon,\O_{F^\nr}}\O_F$-module,
 $\D(H_\k)_\upsilon$ is still torsion free and, hence, free.
 In turn, being finitely generated and $p$-torsion,
 $C_\upsilon$ is a $W(\k)\otimes_{\upsilon,\O_{F^\nr}}\O_F$-module of finite length.
 We may then apply Lemma~\ref{key-ram} and deduce that:
 \[
  \length\bar{M}_I=\frac{|I|}{d'}\length\bar{M}.
 \]
 Altogether:
 \[
  \length\bar{\omega}=\length\bar{\omega}_I+\length\bar{\omega}/\bar{\omega}_I\le
   \length\bar{M}_I+e\cdot\sum_{\tau|\upsilon,\tau\notin I}h_\tau=
   \frac{|I|}{d'}\length\bar{M}+\sum_{\tau|\upsilon,\tau\notin I}eh_\tau
 \]
 and this concludes the proof of the proposition.
\end{proof}

\begin{ex}\label{ex-key}
 Let us take a closer look at a particular class of
 $p$-divisible groups with endomorphism structure for $\O_F$,
 namely that of $p$-divisible ``$\O_F$-modules'' (cf.\ \cite[3.57]{RZ}).
 In our context of base ring $\O_K$,
 this class is defined when $\O_K$ has the structure of an $\O_F$-algebra:
 it consists of objects~$(H,\iota)\in\pdiv_{\O_K,\O_F}$ such that
 the $\O_F$-action induced by $\iota$ on $\Lie(H)$ is via the structure map~$\O_F\rightarrow\O_K$;
 for instance, Lubin-Tate formal groups satisfy this condition.
 We warn the reader that the terminology ``$p$-divisible $\O_F$-module'' is used in
 \cite{MV} with reference to a general $p$-divisible group with endomorphism structure for $\O_F$.

 Now, if $(H,\iota)$ is a $p$-divisible $\O_F$-module over $\O_K$, say with $\height H=dn$,
 then $\Hdg(H,\iota)$ has only two slopes, namely $1/d$ and $0$.
 Indeed, pick a field extension~$K'$ of $K$ containing all embeddings~$\tau$ of $F$ in
 an algebraic closure of $K$ and consider the exact sequence of $K'$-vector-spaces:
 \[
  0\longrightarrow\omega_{H^\vee\!,K'}\longrightarrow M(H)_{K'}
   \longrightarrow\Lie(H)_{K'}\longrightarrow0,
 \]
 obtained as the base change of \eqref{Hdg-fil} along $\O_K\rightarrow K'$.
 This sequence carries a $\Q_p$-linear $F$-action induced by $\iota$ and
 splits as a direct sum of exact sequences:
 \[
  0\longrightarrow\omega_{H^\vee\!,K',\tau}\longrightarrow M(H)_{K',\tau}
   \longrightarrow\Lie(H)_{K',\tau}\longrightarrow0
 \]
 indexed by the embeddings $\tau$ of $F$ in $K'$,
 with $F$ acting through $\tau\colon F\rightarrow K'$ on the respective component.
 Recall that $\Hdg(H,\iota)$ is defined to be the average over $\tau$ of the types $f_\tau$ of
 the filtered vector spaces $(M(H)_{K',\tau},\Fil^\bullet M(H)_{K',\tau})$ given by
 $\Fil^0 M(H)_{K',\tau}=\omega_{H^\vee\!,K',\tau}$,
 with $\Fil^{-i} M(H)_{K',\tau}=M(H)_{K',\tau}$ and $\Fil^i M(H)_{K',\tau}=0$ for $i\ge1$.
 In our case, we have that $\Lie(H)_{K',\tau}=0$ for all $\tau$'s except for
 the embedding~$\tau_0$ corresponding to the structure map $\O_F\rightarrow\O_K$.
 Thus, $f_\tau=(0,\dots,0)\in\Q^n_+$ for all $\tau\neq\tau_0$ and
 $f_{\tau_0}=(1,\dots,1,0,\dots,0)\in\Q^n_+$,
 the number of $1$'s being equal to $\dim_{K'}\Lie(H)_{K',\tau_0}=\dim_{K'}\Lie(H)_{K'}=\dim H$.
 So we get:
 \[
  \Hdg(H,\iota)=\frac{1}{d}\sum_\tau f_\tau
   =\left(\frac{1}{d},\dots,\frac{1}{d},0,\dots,0\right)\in\Q^n_+,
 \]
 the number of $1/d$'s being again equal to $\dim H$.

 This particular case exemplifies the main difficulty inherent to the proof of the above proposition:
 here, it lies in the denominator $d$ of the first slope of $\Hdg(H,\iota)$.
 Indeed, in view of the statement of the proposition,
 this denominator means that the degree of any ($\iota$-stable) sub-$p$-group~$X'$ of
 $H[p]$ must be controlled by a factor $1/d$, compared to the height of $X'$ itself.
 Note that in absence of any endomorphism structure,
 this factor is normally $1$ (cf.\ \cite[Corollaire~2]{F1}).
 In our argument, the necessary control is achieved through the machinery of the preceding lemmas.

 The opposite situation to consider is when $\Lie(H)$ is free as an $\O_K\otimes_{\Z_p}\O_F$-module.
 Then, with notation as above, each component $\Lie(H)_{K',\tau}$ has the same dimension,
 so that $\Hdg(H,\iota)$ has only slopes $1$ and $0$,
 the number of $1$'s being equal to $\dim H/d=\rk_{\O_K\otimes_{\Z_p}\O_F}\Lie(H)$.
 In this case, the previous proposition simply recovers the analogous statement neglecting $\iota$.
\end{ex}

\clearpage

\section{Hodge-Newton filtration}
\thispagestyle{plain}

\subsection{Filtrations of \texorpdfstring{$p$}{p}-divisible groups via Harder-Narasimhan theory}

Let $(H,\iota)$ be a $p$-divisible group over $\O_K$ with endomorphism structure for $\O_F$ and
let $(\NF,\iota)$ be its associated filtered isocrystal over $K$ with coefficients in $F$,
via \eqref{pdv-fic-end}.
Suppose that $z=(x,y)$ is a break point of the Harder-Narasimhan polygon $\HN(H,\iota)=\HN(\NF,\iota)$.
By definition (and compatibility of the Harder-Narasimhan filtration with the $F$-action),
this corresponds to a subobject $(\NF_1,\iota_1)\subseteq(\NF,\iota)$ in $\Fil\Isoc_{K,F}^\wa$,
such that $\HN(\NF_1,\iota_1)$ is the restriction of $\HN(\NF,\iota)$ to $[0,x]$ and,
if $(\NF_2,\iota_2)$ denotes the quotient of $(\NF,\iota)$ by $(\NF_1,\iota_1)$,
then $\HN(\NF_2,\iota_2)$ is the rest of $\HN(\NF,\iota)$ after $z$,
up to a shift of coordinates setting the origin in $z$, i.e.:
\[
 \HN(\NF_2,\iota_2)\colon t\mapsto\HN(\NF,\iota)(t+x)-y,\qquad t\in[0,\height H/d-x].
\]
A crucial question for our purpose is the following:
when is $(\NF_1,\iota_1)$ the filtered isocrystal with coefficients in $F$
associated to an $\iota$-stable sub-$p$-divisible group $H_1$ of $H$?

For instance, this is the case when $H$ is ``of HN type'', that is,
when $\HN(H)=\HN(H[p^i])$ for all $i\ge1$ (cf.\ \cite[\S 2.3]{F2}).
In this situation, the Harder-Narasimhan filtrations of the $p$-groups $H[p^i]$ build up to
a filtration of $H$ by sub-$p$-divisible groups,
which are $\iota$-stable by functoriality of the former filtrations.
Taking the associated filtered isocrystals, we obtain the Harder-Narasimhan filtration of $\NF$.
More generally, there may exist a sub-$p$-divisible group $H_1$ of $H$ satisfying our requirements,
even without the $p$-groups $H_1[p^i]$ being necessarily part of
the Harder-Narasimhan filtration of $H[p^i]$.
In fact, because the valuation of $K$ is discrete, \cite[Théorème~4]{F2} grants that
every $p$-divisible group $H$ over $\O_K$ is isogenous to some $H'$ of HN type;
a closer look at the proof of the statement, i.e.\ the algorithm in \cite[\S 3]{F2},
reveals that this can be made compatible with an $\O_F$-action.
If $e<p-1$, then, we know from \cite[3.3.6]{R} that the kernel and the cokernel of
every map in $\pgr_{\O_K}$ are again $p$-groups (that is, they are flat over $\O_K$).
This allows us to take the image, through any $\O_F$-equivariant isogeny $H'\rightarrow H$,
of the filtration of $H'$ obtained as above;
the result will give, once more, a filtration of $H$ as desired.
Without putting any restriction on $e$, instead, a sufficient condition answering our question
can be found in the configuration of $\HN(H[p])$, relatively to $\HN(H)$.
The next proposition goes in this direction,
moving from a similar procedure as the mentioned algorithm,
but let us first make some further observations.

\begin{rmk}\phantomsection\label{rmk-sub-pdv}
\begin{enumerate}
 \item
 The property required in the question characterises $H_1$ uniquely among
 the $\iota$-stable sub-$p$-divisible groups of $H$.
 Indeed, if $H_1'$ is another candidate, then the identity of $(\NF_1,\iota_1)$ corresponds,
 by fully faithfulness of \eqref{pdv-fic-end}, to morphisms between $H_1$ and $H_1'$ that
 are compatible with the inclusion in $H$, so we must have $H_1'=H_1$.
 Furthermore, the composition of functors
 $\pdiv_{\O_K}\rightarrow\pdiv_{\O_K}\otimes\Q_p\rightarrow\Fil\Isoc_K^{\wa,[-1,0]}$
 sends exact sequences of $p$-divisible groups to exact sequences of filtered isocrystals;
 this can be checked formally, considering that the target is an abelian category and that
 the induced maps on the hom-sets are just given by inverting $p$ in finite free $\Z_p$-modules
 (recall that the second functor is an equivalence of categories, cf.\ Remark~\ref{rmk-pdv-fic}).
 In particular, if $H_2$ denotes the quotient of $H$ by $H_1$,
 with induced $\O_F$-action $\iota_2$, then $(\NF_2,\iota_2)$ is
 the filtered isocrystal with coefficients in $F$ associated to $(H_2,\iota_2)$.
 
 \item
 Suppose now that $H_1$ is an $\iota$-stable sub-$p$-divisible group of $H$ with
 the property that $(\height H_1/d,\dim H_1/d)=z$ and
 denote by $\iota_1$ the restriction of $\iota$ to $H_1$.
 Let $(\NF',\iota')$ be the filtered isocrystal with coefficients in $F$ associated to
 $(H_1,\iota_1)$; this is a subobject of $(\NF,\iota)$ in $\Fil\Isoc_{K,F}^\wa$.
 Then, the Harder-Narasimhan polygon of $(\NF',\iota')$ lies below $\HN(\NF,\iota)$;
 moreover, its end point is $z$, so that $\HN(\NF',\iota')\le\HN(\NF_1,\iota_1)$.
 This implies that the minimal slope of $\HN(\NF',\iota')$ is at least that of
 $\HN(\NF_1,\iota_1)$,
 which in turn is strictly greater than the maximal slope of $\HN(\NF_2,\iota_2)$.
 By functoriality of the Harder-Narasimhan filtration,
 it follows that $(\NF',\iota')\subseteq(\NF_1,\iota_1)$.
 However, since the underlying isocrystals have the same height $dx$,
 this is in fact an equality.
 Thus, the condition that $(\height H_1/d,\dim H_1/d)=z$ is enough to ensure that
 $(\NF_1,\iota_1)$ is the filtered isocrystal with coefficients in $F$ associated to
 $(H_1,\iota_1)$.
\end{enumerate}
\end{rmk}

\begin{prp}\label{sub-pdv}
 Let $(H,\iota)$ be a $p$-divisible group over $\O_K$ with endomorphism structure for $\O_F$ and
 let $(\NF,\iota)$ be its associated filtered isocrystal over $K$ with coefficients in $F$.
 Suppose that $z=(x,y)$ is a break point of $\HN(H,\iota)=\HN(\NF,\iota)$ and let
 $(\NF_1,\iota_1)\subseteq(\NF,\iota)$ be the corresponding subobject in $\Fil\Isoc_{K,F}^\wa$.
 If $z$ lies on $\HN(H[p],\iota)$,
 then there exists a unique $\iota$-stable sub-$p$-divisible group $H_1$ of $H$ whose
 associated filtered isocrystal with coefficients in $F$ is $(\NF_1,\iota_1)$.
\end{prp}

\begin{proof}
 We first reduce to the following situation:
 either $z$ is a break point of $\HN(H[p],\iota)$ or the slope of $\HN(H[p],\iota)$ at $z$ is
 strictly greater than the first slope of $\HN(H,\iota)$ after $z$.
 
 If none of these is the case, then,
 since $\HN(H,\iota)\le\HN(H[p],\iota)$ by Proposition~\ref{HN-lim} and
 $z$ is a break point of $\HN(H,\iota)$, the slope of $\HN(H[p],\iota)$
 at $z$ must be strictly less than the last slope of $\HN(H,\iota)$ before $z$.
 Consider then the dual $p$-divisible group $H^\vee$,
 with the dual action $\iota^\vee$ induced by $\iota$.
 Recalling the compatibility of the Harder-Narasimhan polygon with respect to duality
 (cf.\ Remark~\ref{HN-vee-pdv}), we see that $(H^\vee\!,\iota^\vee)$ satisfies the assumptions of
 the proposition at the point $z^\vee=(x^\vee\!,y^\vee)$ given by:
 \begin{align*}
  x^\vee &=\height H/d-x, \\
  y^\vee &=\height H/d-x+\HN(H,\iota)(x)-\dim H/d=\dim H^\vee\!/d-x+y.
 \end{align*}
 Moreover, the slope of $\HN(H[p]^\vee\!,\iota^\vee)$ at $z^\vee$ is
 strictly greater than the first slope of $\HN(H^\vee\!,\iota^\vee)$ after $z^\vee$,
 so we are in the situation described above.
 
 Suppose now that we find an $\iota^\vee$-stable sub-$p$-divisible group $H'_1$ of $H^\vee$
 as claimed in the statement relative to $(H^\vee\!,\iota^\vee)$ and $z^\vee$.
 Then, $H_1:=(H^\vee\!/H'_1)^\vee$ is an $\iota$-stable sub-$p$-divisible group of $H$, with:
 \[
  \height H_1=\height H-\height H'_1=\height H-dx^\vee=dx
 \]
 and:
 \begin{align*}
  \dim H_1 &=\height H_1-\dim H_1^\vee \\
   &=dx-\dim H^\vee+\dim H'_1=dx-\dim H^\vee+dy^\vee=dy.
 \end{align*}
 By the previous remark, this is enough to prove the proposition for $(H,\iota)$.
 Thus, by possibly passing to the dual $p$-divisible group, we may assume that we are in the
 situation described above.
 
 \bigskip\noindent
 Let us now define a sequence of $\iota$-stable closed sub-$p$-groups $G_i$ of $H[p^i]$,
 for $i\ge1$; these will form a first approximation of $H_1$.
 
 Since $\HN(H,\iota)$ and $\HN(H[p],\iota)$ share the point $z$ and
 they meet at their end point, there must be a break point of
 $\HN(H[p],\iota)$ at $z$ or after (by the previous reduction step);
 let $z_1=(x_1,y_1)$ be the first such point (possibly $z_1=z$).
 We denote by $\mi$ the last slope of $\HN(H[p],\iota)$ before $z_1$,
 so that $\mi$ is strictly greater than the first slope of $\HN(H,\iota)$ after~$z$.
 Recall that, by Proposition~\ref{HN-lim}, we have:
 \[
  \HN(H,\iota)\le\HN^\r(H[p^i],\iota)\le\HN(H[p],\iota)
 \]
 for every $i\ge1$.
 In particular, $z$ lies on every polygon $\HN^\r(H[p^i],\iota)$,
 which then has both a slope valued at least $\mi$ (before $z$) and a slope with
 value strictly less than $\mi$ (as it meets the other polygons at the end point).
 For $i\ge1$, let $z_i=(x_i,y_i)$ be the break point of $\HN^\r(H[p^i],\iota)$ such that
 all its slopes before $z_i$ are greater than or equal to $\mi$ and
 all its slopes after $z_i$ are strictly less than $\mi$;
 then, $x\le x_i\le x_1$ and $z_i$ lies on $\HN(H[p],\iota)$,
 i.e.\ $\HN(H[p],\iota)(x_i)=y_i$.
 Indeed, if the slope of $\HN^\r(H[p^i],\iota)$ right after $z$ is strictly less than $\mu$,
 then $z_i=z$.
 Otherwise, it means that both $\HN^\r(H[p^i],\iota)$ and $\HN(H[p],\iota)$ proceed with
 slope~$\mu$ after $z$, the former until $z_i$ and the latter until $z_1$;
 but since $\HN^\r(H[p^i],\iota)$ lies below the other polygon, then we must have that $x_i\le x_1$,
 i.e.\ $z_i$ must lie on the segment between $z$ and $z_1$ (possibly $z_i=z_1$).
 Note that in particular, unless $z_i=z$,
 the slope of $\HN^\r(H[p^i],\iota)$ between $x$ and $x_i$ is $\mu$,
 which we recall being strictly greater than the slope of $\HN(H,\iota)$ on the same interval.
 Thus, because the polygons~$\HN^\r(H[p^i],\iota)$ converge uniformly to
 $\HN(H,\iota)$ as $i\to\infty$ (by Proposition~\ref{HN-lim} again),
 we also have that $z=\lim_{i\to\infty}z_i$.

 For $i\ge1$, let $G_i\subseteq H[p^i]$ be the $\iota$-stable closed sub-$p$-group
 corresponding to the break point $z_i$ of $\HN^\r(H[p^i],\iota)$;
 by definition, $\height G_i=idx_i$ and $\deg G_i=idy_i$.
 For every other index $j>i$, the choice of $z_i$ implies that
 the minimal slope of $G_i$ is at least $\mi$, which in turn, by the choice of $z_j$,
 is strictly greater than the maximal slope of $H[p^j]/G_j$.
 By functoriality of the Harder-Narasimhan filtration (see also \cite[Proposition~8]{F1}),
 the closed embedding $H[p^i]\rightarrow H[p^j]$ restricts to $G_i\rightarrow G_j$;
 similarly, the map $p^i\colon H[p^j]\rightarrow H[p^{j-i}]$ restricts to
 $p^i\colon G_j\rightarrow G_{j-i}$.
 
 \bigskip\noindent
 We claim that, for $j>i\ge1$, we have:
 \[
  G_i=G_j[p^i]:=\Ker\left(p^i\colon G_j\rightarrow G_{j-i}\right).
 \]
 
 Consider the restriction $p^i_\eta\colon G_{j,\eta}\rightarrow G_{j-i,\eta}$ of the map
 $p^i\colon G_j\rightarrow G_{j-i}$ to the generic fibre (that is, its base change to $K$).
 Let $\c$ be the schematic closure in $G_j$ of the kernel $\Ker(p^i_\eta)$ and
 let $\d$ be the schematic closure in $G_{j-i}$ of the image $\Ima(p^i_\eta)$.
 Note that $G_{i,\eta}\subseteq\Ker(p^i_\eta)\subseteq H[p^i]_\eta$,
 which gives the sequence of closed embeddings
 $G_i\subseteq\c\subseteq H[p^i]$ of $p$-groups over $\O_K$.
 Moreover, the map $p^i\colon G_j\rightarrow G_{j-i}$ factors through
 the closed sub-$p$-group $\d\subseteq G_{j-i}$ and we have a sequence:
 \begin{equation}\label{xct-G_j}
  0\longrightarrow\c\xlongrightarrow{u}G_j\xlongrightarrow{p^i}\d\longrightarrow0
 \end{equation}
 of $p$-groups over $\O_K$, with $u$ a closed embedding, $p^i\circ u=0$ and such that
 $p^i$ induces an isomorphism $G_{j,\eta}/\c_\eta\xrightarrow{\sim}\d_\eta$ on
 the generic fibre, as $\c_\eta=\Ker(p^i_\eta)$ and $\d_\eta=\Ima(p^i_\eta)$.
 Thus:
 \begin{equation}\label{ht-G_j}
  \height G_j=\height\c+\height\d
 \end{equation}
 and:
 \begin{equation}\label{deg-G_j}
  \deg G_j\le\deg\c+\deg\d,
 \end{equation}
 with equality if and only if the sequence~\eqref{xct-G_j} is exact (cf.\ \cite[Corollaire~3]{F1}).
 Now, because $\c$ and $\d$ are closed sub-$p$-groups of $H[p^i]$ and $H[p^{j-i}]$ respectively,
 we have:
 \begin{equation}\label{deg-c}
 \begin{aligned}
  \deg\c &\le\HN(H[p^i])(\height\c)=
   id\HN^\r(H[p^i],\iota)(\height\c/id) \\
  &\le id\HN(H[p],\iota)(\height\c/id)
 \end{aligned}
 \end{equation}
 and:
 \begin{equation}\label{deg-d}
 \begin{aligned}
  \deg\d &\le\HN(H[p^{j-i}])(\height\d)=
   (j-i)d\HN^\r(H[p^{j-i}],\iota)(\height\d/(j-i)d) \\
  &\le(j-i)d\HN(H[p],\iota)(\height\d/(j-i)d).
 \end{aligned}
 \end{equation}
 Altogether, using the concavity of $\HN(H[p],\iota)$ and \eqref{ht-G_j}:
 \begin{equation}\label{deg-G_j-2}
 \begin{aligned}
  \deg\c+\deg\d &\le jd\left(\frac{i}{j}
   \HN(H[p],\iota)(\height\c/id)+\frac{j-i}{j}\HN(H[p],\iota)(\height\d/(j-i)d)\right) \\
  &\le jd\HN(H[p],\iota)(\height\c/jd+\height\d/jd) \\
  &=jd\HN(H[p],\iota)(\height G_j/jd) \\
  &=jd\HN(H[p],\iota)(x_j) \\
  &=jdy_j=\deg G_j.
 \end{aligned}
 \end{equation}
 Hence, we have equality in \eqref{deg-G_j} and the sequence~\eqref{xct-G_j} is exact.
 In particular, $\c=G_j[p^i]$; we remark that this already implies the flatness of $G_j[p^i]$.
 Since $G_i$ is a closed sub-$p$-group of $\c$, it suffices to show that
 $\height\c\le\height G_i$ in order to conclude that $G_i=\c=G_j[p^i]$.
 
 Note that, as a consequence of the previous argument, we have equality all over in
 \eqref{deg-G_j-2} and hence in \eqref{deg-c} and \eqref{deg-d} as well.
 In particular, the polygon $\HN(H[p],\iota)$ is a straight line between
 $\height\d/(j-i)d$ and $\height\c/id$, with $x_j$ lying in the interior of this segment,
 unless the three points coincide.
 
 We first show that $\height\c/id\le x_1$.
 Assume by contradiction that $\height\c/id>x_1$.
 Since $\d$ is a closed sub-$p$-group of $G_{j-i}$, we have:
 \[
  \height\d/(j-i)d\le\height G_{j-i}/(j-i)d=x_{j-i}\le x_1<\height\c/id,
 \]
 with $z_1=(x_1,y_1)$ a break point of $\HN(H[p],\iota)$.
 But this polygon is a straight line on $[\height\d/(j-i)d,\height\c/id]$,
 so we must have $\height\d/(j-i)d=x_1$.
 However, this would contradict $x_j$ being in the interior of the mentioned segment,
 as $x_j\le x_1$.
 
 We can now show that $\height\c\le\height G_i$ or, equivalently, that $\height\c/id\le x_i$.
 Assume by contradiction that $\height\c/id>x_i$.
 The polygons $\HN(H[p],\iota)$ and $\HN^\r(H[p^i],\iota)$ share the point $z_i=(x_i,y_i)$;
 however, the slope of $\HN(H[p],\iota)$ on $[x_i,\height\c/id]$ is at least $\mi$
 (as $\height\c/id\le x_1$), whereas the slope of $\HN^\r(H[p^i],\iota)$ on the same segment is
 strictly less than $\mi$ (by definition of $z_i$).
 Therefore, $\HN^\r(H[p^i],\iota)(\height\c/id)<\HN(H[p],\iota)(\height\c/id)$,
 contradicting the fact that we have an equality in \eqref{deg-c}.
 This proves that $\height\c\le\height G_i$ and hence that $G_i=\c=G_j[p^i]$ as claimed.
 
 \bigskip\noindent
 We now proceed to refining the sequence of $p$-groups $G_i$, $i\ge1$,
 to an $\iota$-stable sub-$p$-divisible group $H_1=(K_i)_{i\ge1}$ of $H$,
 with $(\height H_1/d,\dim H_1/d)=z$.
 Given the observations of Remark~\ref{rmk-sub-pdv},
 this is sufficient in order to conclude the proof of the proposition.
 
 By the previous step, multiplication by $p$ induces closed embeddings:
 \[
  p\colon G_{i+1}/G_i\longrightarrow G_i/G_{i-1}
 \]
 for $i>1$.
 Thus, the sequence of numbers
 $a_i:=\height G_{i+1}/G_i=\height G_{i+1}-\height G_i$, $i\ge1$, is nonincreasing;
 since we are talking about natural numbers,
 there exists an index $i_0\ge1$ such that $a_i=a_{i_0}$ for every $i\ge i_0$.
 We obtain the following formula:
 \[
  \height G_i=\height G_{i_0}+(i-i_0)a_{i_0}=ia_{i_0}+\height G_{i_0}-i_0a_{i_0}
 \]
 for every $i\ge i_0$.
 Then:
 \[
  x_i=\frac{\height G_i}{id}=\frac{a_{i_0}}{d}+\frac{\height G_{i_0}-i_0a_{i_0}}{id}
   \xlongrightarrow[i\to\infty]{}\frac{a_{i_0}}{d}.
 \]
 At the same time, we already observed that the sequence of points $z_i=(x_i,y_i)$, $i\ge1$,
 converges to $z=(x,y)$ as $i\to\infty$; hence, $a_{i_0}/d=x$.
 
 For $i\ge1$, we define the following $p$-group over $\O_K$:
 \[
  K_i:=G_{i+i_0}/G_{i_0}.
 \]
 Note that:
 \[
  \height K_i=\height G_{i+i_0}-\height G_{i_0}=ia_{i_0}=idx.
 \]
 Moreover, because $x_{i_0},x_{i+i_0}\in[x,x_1]$ and, unless $x=x_1$, the polygon
 $\HN(H[p],\iota)$ is a straight line (of slope $\mi$) on this segment, we also have:
 \begin{align*}
  \deg K_i &=\deg G_{i+i_0}-\deg G_{i_0}=(i+i_0)dy_{i+i_0}-i_0dy_{i_0} \\
   &=id\left(\frac{i+i_0}{i}\HN(H[p],\iota)(x_{i+i_0})-
    \frac{i_0}{i}\HN(H[p],\iota)(x_{i_0})\right) \\
   &=id\HN(H[p],\iota)\left(\frac{i+i_0}{i}x_{i+i_0}-\frac{i_0}{i}x_{i_0}\right) \\
   &=id\HN(H[p],\iota)(x)=idy.
 \end{align*}
 For $j>i\ge1$, we have obvious closed embeddings $K_i\rightarrow K_j$
 which induce the identification $K_i=K_j[p^i]$.
 Thus, the family $(K_i)_{i\ge1}$ of $p$-groups over $\O_K$ is a $p$-divisible group $H_1$
 of height $dx$ and dimension $dy$, with $H_1[p^i]=K_i$.
 Finally, for $i\ge1$, the map:
 \[
  p^{i_0}\colon K_i=G_{i+i_0}/G_{i_0}\longrightarrow G_i\subseteq H[p^i]
 \]
 is a closed embedding, which defines an $\iota$-stable closed sub-$p$-group of $H[p^i]$.
 Therefore, $H_1$ is an $\iota$-stable sub-$p$-divisible group of $H$,
 with $(\height H_1/d,\dim H_1/d)=(x,y)=z$.
 This concludes the proof of the proposition.
\end{proof}

\begin{rmk}\label{rmk-sub-pdv-prf}
 The proof of the previous proposition follows a similar procedure to
 the first step of the algorithm in \cite[\S 3]{F2},
 with the difference that here we ``jump'' to a given break point $z$ of the Harder-Narasimhan polygon.
 The crucial point is to obtain a family of closed sub-$p$-groups $G_i$ of $H[p^i]$,
 with the property that $G_i=G_j[p^i]$ for $j>i\ge1$.
 In order to show that this property holds for the constructed family,
 we used a similar argument to that in the proof of \cite[5.4]{XS}.
 Here, we also deal with the case that the polygon $\HN^\r(H[p^i])$ does not have a break point at
 $z$ for any index $i\ge1$, as for instance in the following configuration (not to scale).
 
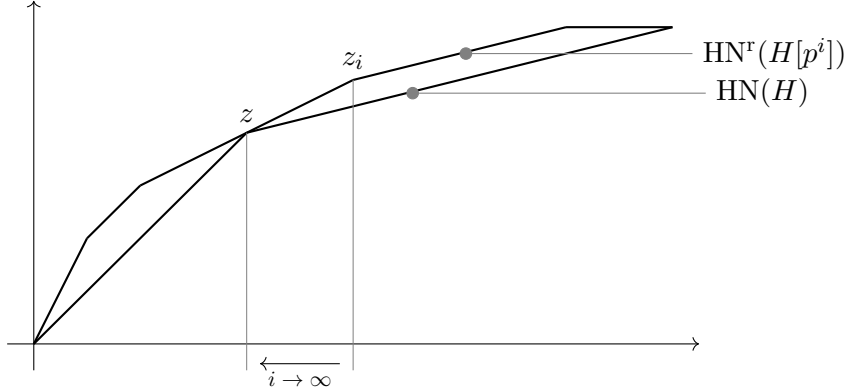
\[
 \begin{tikzpicture}[scale=0.7]
  \draw[->] (-0.5,0) -- (12.5,0);
  \draw[->] (0,-0.5) -- (0,6.5);
  
  \draw[thick] (0,0)  -- (4,4) node[anchor=south] {$z$};
  \draw[thick] (4,4)  -- (12,6);
  
  \draw[thick] (0,0)  -- (1,2);
  \draw[thick] (1,2)  -- (2,3);
  \draw[thick] (2,3)  -- (6,5) node[anchor=south] {$z_i$};
  \draw[thick] (6,5)  -- (10,6);
  \draw[thick] (10,6) -- (12,6);

  \draw[help lines] (4,-0.5) -- (4,4);
  \draw[help lines] (6,-0.5) -- (6,5);
  
  \draw[->] (5.75,-0.375) -- node[below, scale=0.75] {$i\to\infty$} (4.25,-0.375);
  
  \draw[gray, *-] (7,4.75) -- (12.625,4.75) node[black, anchor=west] {$\HN(H)$};
  \draw[gray, *-] (8,5.5) -- (12.375,5.5) node[black, anchor=west] {$\HN^\r(H[p^i])$};
 \end{tikzpicture}
 \]
\end{rmk}

The previous proposition has its own relevance for the study of
$p$-divisible groups over $\O_K$ via Harder-Narasimhan theory.
Let us make this more explicit by reformulating it in a more self-contained fashion.

\begin{cor}\label{cor-sub-pdv}
 Let $(H,\iota)$ be a $p$-divisible group over $\O_K$ with endomorphism structure for $\O_F$.
 Suppose that $z$ is a break point of $\HN(H,\iota)$ which also lies on $\HN(H[p],\iota)$.
 Then, there exists a unique $\iota$-stable sub-$p$-divisible group $H_1$ of $H$ such that,
 if $\iota_1$ denotes the restriction of $\iota$ to $H_1$,
 then $\HN(H_1,\iota_1)$ equals the part of $\HN(H,\iota)$ between the origin and $z$.
 Furthermore, if $H_2$ denotes the quotient of $H$ by $H_1$, with induced $\O_F$-action $\iota_2$,
 then $\HN(H_2,\iota_2)$ equals the rest of $\HN(H,\iota)$ after $z$
 (up to a shift of coordinates setting the origin in $z$).
\end{cor}

\begin{rmk}\label{rmk-sub-pdv-mbe}
 Note that, for every $p$-divisible group $H$ over $\O_K$, the slopes of $\HN(H[p])$ are
 bounded between $0$ and $1$ (e.g.\ because $\HN(H[p])\le\Hdg(H)$ and
 the latter polygon has only slopes $1$ and $0$, alternatively see \cite[Corollaire~2]{F1}).
 In particular, since $\HN(H)\le\HN(H[p])$,
 the assumption of the previous corollary is always verified for
 the first break point of $\HN(H)$ after the slope $1$ and its last break point before the slope $0$.
 In this case, the statement recovers the multiplicative-bilocal-étale filtration of $H$
 (a refined version of \eqref{xct-con-et} obtained by further considering the same sequence for
 $(H^\circ)^\vee$ and dualising back).
 A posteriori, knowing that this filtration is compatible with
 reducing to $H_\k$ via $\O_K\rightarrow\k$ and with passing to the $p^i$-torsion $H[p^i]$,
 one deduces that the break points in question are also break points of
 $\Newt(H)$ and $\HN(H[p^i])$, for $i\ge1$.
\end{rmk}

\subsection{Hodge-Newton reducible filtered isocrystals}\label{S-Hdg-Nwt-fic}

Let us now introduce the Hodge-Newton reducibility hypothesis in the discussion.
This hypothesis concerns the Hodge and the Newton polygon and
can therefore by formulated at the level of filtered isocrystals.
We first analyse its effects in this setting,
following the same argument as in \cite[\S 5.1]{XS}.

\begin{dfn}
 A weakly admissible filtered isocrystal $(\NF,\iota)$ over $K$ with coefficients in $F$
 is \emph{Hodge-Newton reducible} if there exists a break point of $\Newt(\NF,\iota)$
 which also lies on $\Hdg(\NF,\iota)$.
 If $z$ is such a point, we say that $(\NF,\iota)$ is Hodge-Newton reducible \emph{at} $z$.
\end{dfn}

\begin{prp}\label{Hdg-Nwt-fic}
 Let $(\NF,\iota)$ be a weakly admissible filtered isocrystal over $K$ with coefficients in $F$
 and suppose that $(\NF,\iota)$ is Hodge-Newton reducible at $z=(x,y)$.
 Then, $z$ is also a break point of the Harder-Narasimhan polygon $\HN(\NF,\iota)$.
 Furthermore, if $(\NF_1,\iota_1)$ denotes the corresponding subobject in $\Fil\Isoc_{K,F}^\wa$,
 then $\Newt(\NF_1,\iota_1)$ and $\Hdg(\NF_1,\iota_1)$ equal respectively the restriction of
 $\Newt(\NF,\iota)$ and $\Hdg(\NF,\iota)$ to $[0,x]$.
 If $(\NF_2,\iota_2)$ denotes the quotient of $(\NF,\iota)$ by $(\NF_1,\iota_1)$, 
 then $\Newt(\NF_2,\iota_2)$ and $\Hdg(\NF_2,\iota_2)$ equal respectively the rest of
 $\Newt(\NF,\iota)$ and $\Hdg(\NF,\iota)$ after $z$
 (up to a shift of coordinates setting the origin in $z$).
\end{prp}

\begin{proof}
 Write $\NF=(N,\phi,\Fil^\bullet N_K)$.
 Since $z$ is a break point of $\Newt(\NF,\iota)=\Newt(N,\phi,\iota)$,
 we have a decomposition:
 \[
  (N,\phi,\iota)=(N_1,\phi_1,\iota_1)\oplus(N_2,\phi_2,\iota_2)
 \]
 in $\Isoc(\k)_F$,
 with $\Newt(N_1,\phi_1,\iota_1)$ equal to the restriction of $\Newt(\NF,\iota)$ to $[0,x]$ and 
 $\Newt(N_2,\phi_2,\iota_2)$ equal to the rest of $\Newt(\NF,\iota)$ after $z$,
 up to a shift of coordinates setting the origin in $z$;
 in particular, $\dim_{K_0}N_1=\height(N_1,\phi_1)=dx$.
 
 We endow $N_{1,K}$ with the induced filtration
 $\Fil^\bullet N_{1,K}:=N_{1,K}\cap\Fil^\bullet N_K$,
 so that $\NF_1:=(N_1,\phi_1,\Fil^\bullet N_{1,K})$ is a sub-filtered-isocrystal of $\NF$.
 Since $\iota_1$ respects this filtration, we obtain a subobject
 $(\NF_1,\iota_1)$ of $(\NF,\iota)$ in $\Fil\Isoc_{K,F}$.
 Let then $(\NF_2,\iota_2)\in\Fil\Isoc_{K,F}$ be the quotient of $(\NF,\iota)$ by
 $(\NF_1,\iota_1)$ and note that the underlying isocrystal with coefficients in $F$ of
 $(\NF_2,\iota_2)$ identifies with $(N_2,\phi_2,\iota_2)$.
 Thus, $\Newt(\NF_1,\iota_1)=\Newt(N_1,\phi_1,\iota_1)$ equals the restriction of
 $\Newt(\NF,\iota)$ to $[0,x]$ and $\Newt(\NF_2,\iota_2)=\Newt(N_2,\phi_2,\iota_2)$ equals
 the rest of $\Newt(\NF,\iota)$ after $z$ (up to the usual shift of coordinates).
 
 \bigskip\noindent
 We claim that $(\NF_1,\iota_1)$ is weakly admissible;
 note that this implies that $(\NF_2,\iota_2)$ is weakly admissible too.
 In addition, the following argument allows us to find the Hodge polygon of these two objects.
 
 Since $\NF_1$ is a sub-filtered-isocrystal of $\NF$, which is weakly admissible,
 we have that $t_H(\NF_1)\le t_N(\NF_1)$ and only need to check the opposite inequality.
 Pick a field extension $K'$ of $K$ containing all embeddings $\tau$ of $F$
 in an algebraic closure of $K$ and let:
 \begin{gather*}
   N_{K'}=\bigoplus_{\tau\colon F\rightarrow K'}N_\tau,
    \qquad\Fil^\bullet N_{K'}=\bigoplus_{\tau\colon F\rightarrow K'}\Fil^\bullet N_\tau, \\
   N_{1,K'}=\bigoplus_{\tau\colon F\rightarrow K'}N_{1,\tau},
    \qquad\Fil^\bullet N_{1,K'}=\bigoplus_{\tau\colon F\rightarrow K'}\Fil^\bullet N_{1,\tau}
 \end{gather*}
 be the decompositions as in \eqref{egn-dcp-fic}.
 Note that $N_{1,\tau}=N_{1,K'}\cap N_\tau$ and
 $\Fil^\bullet N_{1,\tau}=N_{1,\tau}\cap\Fil^\bullet N_\tau$,
 so $(N_{1,\tau},\Fil^\bullet N_{1,\tau})$ is a subobject of
 $(N_\tau,\Fil^\bullet N_\tau)$ in $\Fil\Vect_{K'|K'}$,
 with $\dim_{K'}N_{1,\tau}=\dim_{K_0}N_1/d=x$.
 By Lemma~\ref{lem-fvs}, then:
 \begin{equation}\label{HN-red-fic-f1}
  -\deg(N_{1,\tau},\Fil^\bullet N_{1,\tau})\le f_\tau(x)
 \end{equation}
 for every $\tau$, where $f_\tau$ is the type of $(N_\tau,\Fil^\bullet N_\tau)$.
 Thus, using that $z=(x,y)$ also lies on~$\Hdg(\NF,\iota)$:
 \begin{equation}\label{HN-red-fic-f2}
 \begin{aligned}
  t_N(\NF_1) &=-d\Newt(\NF_1,\iota_1)(x)=-dy \\
   &=-d\Hdg(\NF,\iota)(x)=-\sum_\tau f_\tau(x) \\
   &\le\sum_\tau\deg(N_{1,\tau},\Fil^\bullet N_{1,\tau}) \\
   &=\deg(N_1,\Fil^\bullet N_{1,K})=t_H(\NF_1).
 \end{aligned}
 \end{equation}
 This proves the claim.
 
 As a consequence, we have an equality in \eqref{HN-red-fic-f2} and
 hence in \eqref{HN-red-fic-f1} for every $\tau$.
 By Lemma~\ref{lem-fvs} again, the type of
 $(N_{1,\tau},\Fil^\bullet N_{1,\tau})$ equals then the restriction of $f_\tau$ to $[0,x]$.
 Summing over $\tau$ and dividing by $d$,
 we obtain that $\Hdg(\NF_1,\iota_1)$ equals the restriction of $\Hdg(\NF,\iota)$ to $[0,x]$.
 
 Denote now by $\Fil^\bullet N_{2,K}$ the filtration of $\NF_2$ and let:
 \[
   N_{2,K'}=\bigoplus_{\tau\colon F\rightarrow K'}N_{2,\tau},
    \qquad\Fil^\bullet N_{2,K'}=\bigoplus_{\tau\colon F\rightarrow K'}\Fil^\bullet N_{2,\tau}
 \]
 be the decomposition as in \eqref{egn-dcp-fic},
 with respect to the $F$-action given by $\iota_2$.
 Note that $(N_{2,\tau},\Fil^\bullet N_{2,\tau})$ is the quotient of
 $(N_\tau,\Fil^\bullet N_\tau)$ by $(N_{1,\tau},\Fil^\bullet N_{1,\tau})$
 in $\Fil\Vect_{K'|K'}$ for every $\tau$.
 Since we already know that the type of $(N_{1,\tau},\Fil^\bullet N_{1,\tau})$ equals
 the restriction of $f_\tau=f(N_\tau,\Fil^\bullet N_\tau)$ to $[0,x]$,
 it follows from the behaviour of the type in short exact sequences that
 $f(N_{2,\tau},\Fil^\bullet N_{2,\tau})$ equals the rest of $f_\tau$ after $(x,f_\tau(x))$,
 up to a shift of coordinates setting the origin in $(x,f_\tau(x))$.
 Summing over $\tau$ and dividing by $d$,
 we obtain that $\Hdg(\NF_2,\iota_2)$ equals the rest of $\Hdg(\NF,\iota)$ after $z$,
 up to a shift of coordinates setting the origin in $z$.
 
 \bigskip\noindent
 We can finally show that $z$ is a break point of $\HN(\NF,\iota)$ and that
 $(\NF_1,\iota_1)$ is the corresponding subobject of $(\NF,\iota)$ in $\Fil\Isoc_{K,F}^\wa$.
 This is enough to finish the proof of the proposition,
 as we already know that the Newton polygon and the Hodge polygon of $(\NF_1,\iota_1)$ and
 $(\NF_2,\iota_2)=(\NF,\iota)/(\NF_1,\iota_1)$ are as claimed in the statement.
 
 Because $(\NF_1,\iota_1)$ is a subobject of $(\NF,\iota)$ in $\Fil\Isoc_{K,F}^\wa$, we have:
 \[
  -t_N(\NF_1)/d\le\HN(\NF,\iota)(\height(N_1,\phi_1)/d)=\HN(\NF,\iota)(x).
 \]
 Now, on the one hand:
 \[
  -t_N(\NF_1)/d=\Newt(\NF_1,\iota_1)(x)=\Newt(\NF,\iota)(x)=y;
 \]
 on the other hand, by Proposition~\ref{HN<Nwt-fic-end}:
 \[
  \HN(\NF,\iota)(x)\le\Newt(\NF,\iota)(x)=y.
 \]
 Thus, equality holds and, in particular, $z=(x,y)$ lies on $\HN(\NF,\iota)$.
 But $z$ is a break point of $\Newt(\NF,\iota)$ and, by Proposition~\ref{HN<Nwt-fic-end} again, 
 $\HN(\NF,\iota)\le\Newt(\NF,\iota)$. Hence, $z$ is break point of $\HN(\NF,\iota)$ as well.
 
 Lastly, note that $(\NF_1,\iota_1)$ is a subobject of $(\NF,\iota)$ in $\Fil\Isoc_{K,F}^\wa$
 with the property that $(\height(N_1,\phi_1)/d,-t_N(\NF_1)/d)=z$.
 By functoriality of the Harder-Narasimhan filtration
 (see also the argument in the second part of Remark~\ref{rmk-sub-pdv}),
 it follows that $(\NF_1,\iota_1)$ is the subobject of $(\NF,\iota)$ corresponding to $z$.
 This concludes the proof of the proposition.
\end{proof}

\subsection{Hodge-Newton reducible \texorpdfstring{$p$}{p}-divisible groups}\label{S-Hdg-Nwt-pdv}

\begin{dfn}
 A $p$-divisible group $(H,\iota)$ over $\O_K$ with endomorphism structure for $\O_F$
 is \emph{Hodge-Newton reducible} if its associated filtered isocrystal over $K$ with
 coefficients in $F$ is Hodge-Newton reducible, i.e.\ there exists a break point of
 $\Newt(H,\iota)$ which also lies on $\Hdg(H,\iota)$.
 If $z$ is such a point, we say that $(H,\iota)$ is Hodge-Newton reducible \emph{at} $z$.
\end{dfn}

The Hodge-Newton reducibility is an assumption that regards the whole equivariant isogeny class
of a $p$-divisible group $(H,\iota)$ over $\O_K$ with endomorphism structure for $\O_F$.
Using Propositions \ref{HN-lim} and~\ref{key}, we deduce a constraint on $\HN(H[p],\iota)$,
hence concerning $(H,\iota)$ itself.
We can then apply Proposition~\ref{sub-pdv} and obtain the following theorem.

\begin{thm}\label{thm}
 Let $(H,\iota)$ be a $p$-divisible group over $\O_K$ with endomorphism structure for $\O_F$
 and suppose that $(H,\iota)$ is Hodge-Newton reducible at $z$.
 Then, there exists a unique $\iota$-stable sub-$p$-divisible group $H_1$ of $H$ such that,
 if $\iota_1$ denotes the restriction of $\iota$ to $H_1$,
 then $\Newt(H_1,\iota_1)$, $\Hdg(H_1,\iota_1)$ and $\HN(H_1,\iota_1)$ equal respectively
 the part of $\Newt(H,\iota)$, $\Hdg(H,\iota)$ and $\HN(H,\iota)$ between the origin and $z$.
 Furthermore, if $H_2$ denotes the quotient of $H$ by $H_1$, with induced $\O_F$-action $\iota_2$,
 then $\Newt(H_2,\iota_2)$, $\Hdg(H_2,\iota_2)$ and $\HN(H_2,\iota_2)$ equal respectively
 the rest of $\Newt(H,\iota)$, $\Hdg(H,\iota)$ and $\HN(H,\iota)$ after $z$
 (up to a shift of coordinates setting the origin in $z$).
\end{thm}

\begin{proof}
 Let $(\NF,\iota)$ be the filtered isocrystal with coefficients in $F$
 associated to $(H,\iota)$.
 By Proposition~\ref{Hdg-Nwt-fic}, $z$ is a break point of $\HN(\NF,\iota)=\HN(H,\iota)$.
 Let then $(\NF_1,\iota_1)$ be the corresponding subobject of
 $(\NF,\iota)$ in $\Fil\Isoc_{K,F}^\wa$ and let $(\NF_2,\iota_2)$ be the quotient of
 $(\NF,\iota)$ by $(\NF_1,\iota_1)$.
 
 By Propositions \ref{HN-lim} and~\ref{key}, we have:
 \[
  \HN(H,\iota)\le\HN(H[p],\iota)\le\Hdg(H,\iota).
 \]
 Since $z$ lies on both $\HN(H,\iota)$ and $\Hdg(H,\iota)$,
 this implies that $z$ also lies on $\HN(H[p],\iota)$.
 Then, by Proposition~\ref{sub-pdv},
 there exists a unique $\iota$-stable sub-$p$-divisible group $H_1$ of $H$ whose
 associated filtered isocrystal with coefficients in $F$ is $(\NF_1,\iota_1)$.
 Furthermore, if $H_2$ denotes the quotient of $H$ by $H_1$,
 with induced $\O_F$-action $\iota_2$, then $(\NF_2,\iota_2)$ is
 the filtered isocrystal with coefficients in $F$ associated to $(H_2,\iota_2)$.
 Denoting by $\iota_1$ the restriction of $\iota$ to $H_1$,
 it follows then from Proposition~\ref{Hdg-Nwt-fic} that
 the polygons of $(H_1,\iota_1)$ and $(H_2,\iota_2)$ are as claimed.
 
 As for uniqueness, the prescription on the Harder-Narasimhan polygon of
 $(H_1,\iota_1)$ ensures that its associated filtered isocrystal with coefficients in $F$ is
 the $(\NF_1,\iota_1)$ considered above.
 The uniqueness of $H_1$ follows then from Proposition~\ref{sub-pdv}.
\end{proof}

\subsection{The polarised case}

Assume here that $F$ carries a field involution $(\ )^*\colon F\rightarrow F$,
possibly equal to the identity of $F$.
For $n\in\N$ and an element $f=(a_i)_{i=1}^n\in\Q^n_+$ of the Newton set,
we define its \emph{dual} element to be $f^\vee:=(1-a_{n+1-i})_{i=1}^n\in\Q^n_+$ or,
as a polygon:
\[
 f^\vee\colon x\mapsto x+f(n-x)-f(n).
\]
We clearly have ${f^\vee}^\vee=f$.
If $f\in\Q^n_+$ satisfies $f=f^\vee$, we say that $f$ is \emph{symmetric}.

For $R$ a complete Noetherian commutative local ring with residue field of characteristic~$p$ and
$(H,\iota)$ a $p$-divisible group over $R$ with endomorphism structure for $\O_F$,
let $\iota^\vee\colon\O_F\rightarrow\End(H^\vee)$ be the $\O_F$-action on
the dual $p$-divisible group~$H^\vee$ induced by $\iota$ through functoriality of Cartier duality;
denote by $\iota^{\vee,*}$ the composition $\iota^\vee\circ(\ )^*$.
Then, $(H^\vee\!,\iota^{\vee,*})$ is a new
$p$-divisible group with endomorphism structure for $\O_F$.

When $R=\O_K$, Remark~\ref{HN-vee-pdv} shows that the Harder-Narasimhan
polygon of $(H^\vee\!,\iota^{\vee,*})$ is the dual of $\HN(H,\iota)\in\Q^n_+$,
where $n=\height H/d$ (and recall that $\HN(H,\iota)(n)=\dim H/d$).
In fact, the twist introduced by $(\ )^*$ in the endomorphism structure
does not really affect the Harder-Narasimhan polygon,
which depends on the $\O_F$-action only by means of the rescaling factor $1/d=1/[F:\Q_p]$.
We will now see that the other polygons behave in the same way with respect to duality.

\paragraph{Isocrystals and duality.}
Let $(N,\phi)$ be an isocrystal over $\k$.
The \emph{dual isocrystal} is given by the dual $K_0$-vector-space $N^\vee:=\Hom_{K_0}(N,K_0)$,
together with the $\sigma$-linear endomorphism $\phi^\vee\colon f\mapsto\sigma\circ f\circ V$,
where $V=p\phi^{-1}$ is the Verschiebung map of $(N,\phi)$.
If $(N,\phi)$ is isotypical of slope $\lambda=r/s$,
then $(N^\vee\!,\phi^\vee)$ is isotypical of slope $1-\lambda=(s-r)/s$.
Indeed, if $M\subseteq N$ is a $W(\k)$-lattice with $\phi^sM=p^rM$,
then the $W(\k)$-lattice $M^\vee:=\Set{f\in N^\vee|f(M)\subseteq W(\k)}\subseteq N^\vee$
satisfies $\phi^{s-r}M^\vee=p^rM^\vee$.
In general, if the Newton slopes of $(N,\phi)$ are $\lambda_1<\dots<\lambda_m$,
then those of $(N^\vee\!,\phi^\vee)$ will be $1-\lambda_m<\dots<1-\lambda_1$,
with the height of the isotypical component associated to $1-\lambda_i$ equal to
that of the isotypical component of $(N,\phi)$ associated to $\lambda_i$, for $1\le i\le m$.

\paragraph{The Newton polygon and duality.}
Let $(H,\iota)$ be a $p$-divisible group over $\O_K$ with endomorphism structure for $\O_F$ and
recall that $\Newt(H,\iota)$ is defined to be the Newton polygon of
the isocrystal~$(\D(H_\k)\otimes_{W(\k)}K_0,p^{-1}(\phi_{H_\k}\otimes\id))$ over $\k$,
with the induced $F$-action, where $(\D(H_\k),\phi_{H_\k})$ is the Dieudonné module of
the reduction $H_\k$ of $H$ to $\k$.
Now, if $(\D(H_\k^\vee),\phi_{H_\k^\vee})$ is the Dieudonné module of
the dual $p$-divisible group $H_\k^\vee$ over $\k$,
then the isocrystal $((\D(H_\k^\vee)\otimes_{W(\k)}K_0,(\phi_{H_\k^\vee}\otimes\id))$
identifies naturally with the dual of $((\D(H_\k)\otimes_{W(\k)}K_0,(\phi_{H_\k}\otimes\id))$,
as defined in the previous paragraph (cf.\ \cite[\S III, Proposition~6.4]{Fo1}).
Thus, shifting the slopes by $-1$ and computing the Newton polygons
(including the rescaling due to the endomorphism structure),
we obtain that $\Newt(H^\vee\!,\iota^{\vee,*})=\Newt(H,\iota)^\vee\in\Q^n_+$,
where $n=\height H/d$ (and recall that $\Newt(H,\iota)(n)=\dim H/d$).\looseness=-1

\paragraph{The Hodge polygon and duality.}
Let $(H,\iota)$ be a $p$-divisible group over $\O_K$ with endomorphism structure for $\O_F$.
If $K'$ is a sufficiently large field extension of $K$,
then the base change of the exact sequence~\eqref{Hdg-fil} along $\O_K\rightarrow K'$
carries a $\Q_p$-linear $F$-action induced by $\iota$ and splits as a direct sum of exact sequences:
\begin{equation}\label{Hdg-fil-egn}
 0\longrightarrow\omega_{H^\vee\!,K',\tau}\longrightarrow M(H)_{K',\tau}
  \longrightarrow\Lie(H)_{K',\tau}\longrightarrow0
\end{equation}
indexed by the embeddings $\tau$ of $F$ in $K'$,
with $F$ acting through $\tau\colon F\rightarrow K'$ on the respective component.
Recall now that we have a natural identification $\omega_H\cong\Hom_{\O_K}(\Lie(H),\O_K)$ of
finite free $\O_K$-modules, which, after base change along $\O_K\rightarrow K'$,
gives $\omega_{H,K'}\cong\Hom_{K'}(\Lie(H)_{K'},K')$.
By compatibility with respect to the $F$-action induced by $\iota$ on both sides,
this splits as a direct sum of isomorphisms:
\[
 \omega_{H,K',\tau}\cong\Hom_{K'}(\Lie(H)_{K'},K')_\tau=\Hom_{K'}(\Lie(H)_{K',\tau},K')
\]
indexed by $\tau$ as above and
with $F$ acting through $\tau\colon F\rightarrow K'$ on the respective component.
Using the exactness of \eqref{Hdg-fil-egn}, we get that:
\begin{equation}\label{omg-vee}
\begin{aligned}
 \dim_{K'}\omega_{H,K',\tau}
 &=\dim_{K'}\Hom_{K'}(\Lie(H)_{K',\tau},K') \\
 &=\dim_{K'}\Lie(H)_{K',\tau} \\
 &=\dim_{K'}M(H)_{K',\tau}-\dim_{K'}\omega_{H^\vee\!,K',\tau} \\
 &=\height H/d-\dim_{K'}\omega_{H^\vee\!,K',\tau}
\end{aligned}
\end{equation}
for all $\tau$'s.
Then, computing the types of the relevant filtrations and averaging over $\tau$,
we obtain that $\Hdg(H^\vee\!,\iota^{\vee,*})=\Hdg(H,\iota)^\vee\in\Q^n_+$,
where $n=\height H/d$ (and recall that $\Hdg(H,\iota)(n)=\dim H/d$).
Here, the twist introduced by $(\ )^*$ results in
a permutation of the types defining the Hodge polygon;
its influence, therefore, is cleared in the averaging process.

\begin{dfn}
 Let $R$ be a complete Noetherian commutative local ring with residue field of characteristic $p$,
 let $(H,\iota)$ be a $p$-divisible group over $R$ with endomorphism structure for $\O_F$ and
 consider the $\O_F$-action given by $\iota^{\vee,*}=\iota^\vee\circ(\ )^*$ on
 the dual $p$-divisible group~$H^\vee$.
 A \emph{polarisation} on $(H,\iota)$ is an $\O_F$-equivariant isomorphism
 $\lambda\colon H\rightarrow H^\vee$ such that,
 under the natural identification ${H^\vee}^\vee\cong H$,
 we have that $\lambda^\vee=-\lambda$ (i.e.\ $\lambda$ is \emph{antisymmetric}).
 We call $(H,\iota,\lambda)$ a \emph{polarised} $p$-divisible group over $R$ with
 endomorphism structure for $\O_F$.
\end{dfn}

\begin{rmk}
 The choice of defining a polarisation on a $p$-divisible group to be antisymmetric
 (as opposed to symmetric, i.e.\ such that $\lambda^\vee=\lambda$) follows the book~\cite{RZ}.
 This is based on the fact that if $H$ is the $p$-divisible group associated to an abelian scheme~$A$,
 then a symmetric map from $A$ to its dual induces an antisymmetric map~$H\rightarrow H^\vee$
 (cf.\ \cite[Proposition~1.12]{O}).
 In particular, since polarisations on abelian schemes are defined to be symmetric,
 then a principal polarisation on $A$ induces a polarisation on $H$ as defined here.\looseness=-1
\end{rmk}

If $(H,\iota,\lambda)$ is a polarised $p$-divisible group over $\O_K$ with
endomorphism structure for $\O_F$,
then the isomorphism $(H,\iota)\cong(H^\vee\!,\iota^{\vee,*})$ given by the polarisation
$\lambda$ implies that the Newton polygon, the Hodge polygon and the Harder-Narasimhan
polygon of $(H,\iota)$ coincide with the respective polygons of $(H^\vee\!,\iota^{\vee,*})$.
According to the compatibility relations found above, this means that these polygons are
symmetric elements of $\Q^n_+$, where $n=\height H/d$.
In particular, if $(H,\iota)$ is Hodge-Newton reducible at $z=(x,y)$,
then it is also Hodge-Newton reducible at the symmetric point
$z^\vee=(x^\vee\!,y^\vee)$ given by:
\begin{equation}\label{vee-pt}
 x^\vee=\height H/d-x, \qquad y^\vee=\dim H/d-x+y.
\end{equation}
Thus, applying Theorem~\ref{thm} for both points and using the uniqueness property,
we find the following corollary.

\begin{cor}\label{cor}\enlargethispage{\baselineskip}
 Let $(H,\iota,\lambda)$ be a polarised $p$-divisible group over $\O_K$ with endomorphism
 structure for $\O_F$, suppose that $(H,\iota)$ is Hodge-Newton reducible at $z=(x,y)$ and
 let $z^\vee=(x^\vee\!,y^\vee)$ be the symmetric point as in \eqref{vee-pt}.
 Assume without loss of generality that $x\le x^\vee$.
 Then, there exists a unique filtration of $(H,\iota)$ by
 sub-$p$-divisible groups with endomorphism structure for $\O_F$:
 \[
  (H_1,\iota_1)\subseteq(H_1',\iota_1')\subseteq(H,\iota)
 \]
 satisfying the following property:
 if $H_2$ denotes the quotient of $H_1'$ by $H_1$, with induced $\O_F$-action $\iota_2$,
 and $H_3$ denotes the quotient of $H$ by $H_1'$, with induced $\O_F$-action $\iota_3$,
 then $\Newt(H_i,\iota_i)$, $\Hdg(H_i,\iota_i)$ and $\HN(H_i,\iota_i)$ equal respectively the parts of
 $\Newt(H,\iota)$, $\Hdg(H,\iota)$ and $\HN(H,\iota)$ between the origin and $z$ if $i=1$,
 between $z$ and $z^\vee$ if $i=2$ (up to a shift of coordinates setting the origin in $z$)
 and between $z^\vee$ and $(\height H/d,\dim H/d)$ if $i=3$
 (up to a shift of coordinates setting the origin in $z^\vee$).
 Furthermore, if $H_2'$ denotes the quotient of $H$ by $H_1$, endowed with the induced $\O_F$-action,
 then $\lambda$ induces $\O_F$-equivariant isomorphisms:
 \[
  H_3\cong H_1^\vee \qquad\text{and}\qquad H_2'\cong H_1'^{\,\vee},
 \]
 where $H_1^\vee$ and $H_1'^{\,\vee}$ carry the $\O_F$-action given respectively by
 $\iota_1^{\vee,*}$ and $\iota_1'^{\,\vee,*}$.
 In addition, $\lambda$ induces a polarisation on $(H_2,\iota_2)$.
\end{cor}

\begin{rmk}
 With notation as above, note that the polarisation $\lambda$ induces an isomorphism
 $\omega_H\xrightarrow{\sim}\omega_{H^\vee}$ of $\O_K$-modules,
 which is $\O_F$-equivariant up to the involution $(\ )^*$ of $F$.
 Now, if this involution is trivial (that is, equal to the identity of $F$),
 then we get from \eqref{omg-vee} that $\dim_{K'}\omega_{H^\vee\!,K',\tau}=\height H/2d$
 for all embeddings $\tau$ of $F$ in a sufficiently large extension~$K'$ of $K$.
 In this case, therefore, we must have that $\height H\in 2d\N$  and:
 \[
  \Hdg(H,\iota)=(1^{(\height H/2d)},0^{(\height H/2d)}).
 \]
 In particular, if $(\ )^*$ is trivial then the above corollary simply recovers
 the multiplicative-bilocal-étale filtration of $H$,
 in its symmetric version due to the polarisation
 (and together with its compatibility with additional endomorphism structures).
 Indeed, because $\HN(H,\iota)\le\Newt(H,\iota)$,
 then $z$ must be the first break point of $\HN(H,\iota)$ after the slope~$1$ and
 $z^\vee$ its last break point before the slope~$0$ (see also Remark~\ref{rmk-sub-pdv-mbe}).
\end{rmk}

\subsection{The special fibre}\label{spf}

In this final section, we would like to compare the reduction to $\k$ of the filtration obtained in
Theorem~\ref{thm} with other known results at the level of $p$-divisible groups over $\k$.
This will ultimately lead to the fact that the reduction of our filtration is split.
Let us first make a preliminary observation about the Newton polygon.

\begin{rmk}\label{rmk-Nwt}
 Let $(H,\iota)$ be a $p$-divisible group over $\O_K$ with
 endomorphism structure for $\O_F$ and write $H_\k$ for the reduction of $H$ to $\k$.
 By definition, a slope $\lambda$ features in the polygon~$\Newt(H,\iota)$ if and
 only if $-\lambda$ is a Newton slope of
 the isocrystal~$(\D(H_\k)\otimes_{W(\k)}K_0,p^{-1}(\phi_{H_\k}\otimes\id))$ over $\k$
 (recall the minus sign in the original definition of the Newton polygon for isocrystals).
 Now, the factor $p^{-1}$ accounts for a shift of the Newton slopes by $-1$,
 so that the above is equivalent to $1-\lambda$ being a Newton slope of
 the isocrystal~$(\D(H_\k)\otimes_{W(\k)}K_0,(\phi_{H_\k}\otimes\id))$.
 But, as observed in the previous section,
 these are exactely the Newton slopes of the dual isocrystal,
 which in turn can be identified with $(\D(H_\k^\vee)\otimes_{W(\k)}K_0,\phi_{H_\k^\vee}\otimes\id)$.
 Note, finally, that the latter is the isocrystal associated to $(\D(H_\k^\vee),\phi_{H_\k^\vee})$,
 which coincides with the contravariant Dieudonné module of $H_\k$.
\end{rmk}

\paragraph{Comparison with the slope filtration.}
Let now $(H,\iota)\in\pdiv_{\O_K,\O_F}$ be Hodge-Newton reducible.
Then, the previous observation allows to recognise the reduction~$H_{1,\k}\subseteq H_\k$ to
$\k$ of the filtration from Theorem~\ref{thm} as
part of the slope filtration of $H_\k$ from \cite[Corollary~13]{Z2}.
This is immediate from the configuration of the Newton polygons stated in the theorem; in fact,
these polygons are really an invariant of the reduction of the respective $p$-divisible group and
they are only affected by $\iota$ in terms of a rescaling factor.

\paragraph{Comparison with the Hodge-Newton decomposition.}
More specifically, keep the notation from above and
consider the induced $\O_F$-actions $\iota$ on $H_\k$ and $\iota_1$ on $H_{1,\k}$.
Then, $(H_{1,\k},\iota_1)$ is one piece of the Hodge-Newton decomposition of
$(H_\k,\iota)$ obtained (in terms of $F$-crystals with $\O_F$-action) in \cite[1.3.2]{BH}.
In particular, the filtration $H_{1,\k}\subseteq H_\k$ is $\O_F$-equivariantly split.
Here, we should warn that the Hodge polygon (and consequently the notion of Hodge-Newton reducibility)
considered in loc.\ cit.\ does not, in general, coincide with the one defined here.
The former polygon, in fact, is an invariant of $p$-divisible groups
(or, more generally, $F$-crystals) over $\k$ with endomorphism structure for $\O_F$.
If $F$ is unramified over $\Q_p$, we saw in Remark~\ref{rmk-Hdg-unr} that
this is also the case for our definition and the two polygons amount indeed to the same invariant.
In general, Example~\ref{ex-Hdg-ram} shows that the two definitions cannot coincide.
However, the following considerations allow us to relate Theorem~\ref{thm} to \cite[1.3.2]{BH}.

First of all, enlarging the field $K$ if necessary,
write $\omega_H=\bigoplus_\upsilon\omega_{H,\upsilon}$ using \eqref{OK-dcp} and,
for each $\upsilon\colon F^\nr\rightarrow K_0$,
consider the decomposition~$\omega_{H,K,\upsilon}:=\omega_{H,\upsilon}\otimes_{\O_K}K=
\bigoplus_{\tau|\upsilon}\omega_{H,K,\tau}$ from \eqref{egn-dcp-ram}.
For every $\upsilon$, then,
choose an ordering~$\Set{\tau_{\upsilon,1},\dots,\tau_{\upsilon,e(F\mid\Q_p)}}$ of
the embeddings~$\tau\colon F\rightarrow K$ with $\tau|\upsilon$ and
set $d_{\upsilon,i}:=\dim_K\omega_{H,K,\tau_{\upsilon,i}}$ for $1\le i\le e(F|\Q_p)$.
Now, the fact that $(H_\k,\iota)$ comes from the object~$(H,\iota)$ over $\O_K$ ensures that
the $F$-crystal~$(\D(H_\k^\vee),\phi_{H_\k^\vee})$, together with the $\O_F$-action induced by $\iota$,
satisfies a \emph{Pappas-Rapoport condition}, as defined in \cite[1.2.1]{BH},
for the tuple of integers consisting of the $d_{\upsilon,i}$'s.
To be precise, this condition is given by the reduction to $\k$, for each $\upsilon$,
of the increasing filtration of $\omega_{H,\upsilon}$ defined by
$\omega_{H,\upsilon,j}:=\omega_{H,\upsilon}\cap\bigoplus_{i=1}^j\omega_{H,K,\tau_{\upsilon,i}}$,
for $0\le j\le e(F|\Q_p)$.
Recall, indeed, that we have
an $\O_F$-equivariant identification~$\omega_{H_\k}\cong\D(H_\k^\vee)/\phi_{H_\k^\vee}\D(H_\k^\vee)$.
Moreover, argueing as in the proof of Lemma~\ref{key-ram}, we see that
the definition above yields a filtration of $\omega_{H,\upsilon}$ by $\O_K$-direct-summands,
with $\rk_{\O_K}\omega_{H,\upsilon,j}/\omega_{H,\upsilon,j-1}=\dim_K\omega_{H,K,\tau_{\upsilon,j}}=
d_{\upsilon,j}$ for all $j=1,\dots,e(F|\Q_p)$ and with $\O_F$ acting on each of
these graded pieces through the corresponding embedding~$\tau_{\upsilon,j}$;
in particular, after reducing to $\k$, any uniformiser of $\O_F$ acts as zero on the graded pieces.
The presence of the Pappas-Rapoport condition ensures, by \cite[1.3.1]{BH},
that the Hodge polygon of $(\D(H_\k^\vee),\phi_{H_\k^\vee})$, as defined in \cite[1.1.7]{BH},
is bounded between its ``$O_F$-Newton polygon'' and
the ``Pappas-Rapoport polygon'' associated to the tuple~$(d_{\upsilon,i})_{\upsilon,i}$.
Before comparing these invariants (introduced respectively in \cite[1.1.9]{BH} and \cite[\S 1.2]{BH})
with those of $(H,\iota)$ in use here, note that the polygons considered in loc.\ cit.\ are convex;
one can pass from one point of view to the other by simply reversing the order of the slopes
(and the inequalities).
Then, on the one hand, Remark~\ref{rmk-Nwt} and \cite[1.1.12]{BH} allow to
recognise the $O_F$-Newton polygon of $(\D(H_\k^\vee),\phi_{H_\k^\vee})$ as
(the convex version of) $\Newt(H,\iota)$.
On the other hand, a simple computation using the formula~\eqref{omg-vee} shows that
the Pappas-Rapoport polygon associated to $(d_{\upsilon,i})_{\upsilon,i}$ is
(the convex version of) $\Hdg(H,\iota)$.
In particular, if $(H,\iota)$ is Hodge-Newton reducible,
then so is $(\D(H_\k^\vee),\phi_{H_\k^\vee})$ in the sense of \cite[1.3.2]{BH},
with respect to the same break point of the Newton polygon (up to reversing the order of the slopes).
Indeed, the double bound assured above implies that the Hodge polygon of
$(\D(H_\k^\vee),\phi_{H_\k^\vee})$ passes through this break point too.
In conclusion,
the decomposition of $(\D(H_\k^\vee),\phi_{H_\k^\vee})$ following from \cite[1.3.2]{BH} corresponds,
by Dieudonné theory, to a decomposition of $(H_\k,\iota)$.
It is then immediate from the configuration of the Newton polygons
(and the uniqueness of the slope filtration),
that one piece of this decomposition is $(H_{1,\k},\iota_1)$.

\vfill

\noindent
\textsc{Andrea Marrama}\\
CMLS, École Polytechnique\\
91128 Palaiseau Cedex, France\\
E-mail: \href{mailto:andrea.marrama@polytechnique.edu}{andrea.marrama@polytechnique.edu}

\end{document}